\documentclass[12pt]{amsart}

\usepackage{amsmath, amssymb, amsfonts, amsthm, amscd}
\usepackage{manfnt}
\usepackage{mathtools}
\usepackage{fullpage}
\usepackage{url}
\usepackage[all]{xy}
   \SelectTips{cm}{10}
\usepackage{txfonts}
\usepackage{enumitem}
\usepackage{bbm}
\newtheorem{thm}{Theorem}[section]
\newtheorem{lemma}[thm]{Lemma}
\newtheorem{prop}[thm]{Proposition}
\newtheorem{cor}[thm]{Corollary}

\newtheorem{prop-conj}[thm]{Proposition-Conjecture}

\theoremstyle{definition}
\newtheorem{defn}[thm]{Definition}

\theoremstyle{definition}
\newtheorem{rmk}[thm]{Remark}

\theoremstyle{definition}
\newtheorem{assumption}[thm]{Assumption}

\theoremstyle{definition}

\theoremstyle{definition}

\theoremstyle{definition}
\newtheorem{question}[thm]{Question}

\theoremstyle{definition}
\newtheorem{eg}[thm]{Example}

\DeclareFontFamily{U}{wncy}{}
    \DeclareFontShape{U}{wncy}{m}{n}{<->wncyr10}{}
    \DeclareSymbolFont{mcy}{U}{wncy}{m}{n}
    \DeclareMathSymbol{\Sha}{\mathord}{mcy}{"58}

\newcommand{\Q}{\mathbb{Q}}

\newcommand{\Z}{\mathbb{Z}}
\newcommand{\CC}{\mathbb{C}}
\newcommand{\RR}{\mathbb{R}}

\newcommand{\Ql}{\mathbb{Q}_{\ell}}

\newcommand{\Fp}{\mathbb{F}_p}

\DeclareMathOperator{\Hom}{Hom}

\DeclareMathOperator{\End}{End}
\DeclareMathOperator{\Aut}{Aut}
\DeclareMathOperator{\Out}{Out}

\DeclareMathOperator{\Sym}{Sym}
\DeclareMathOperator{\Ad}{Ad}

\DeclareMathOperator{\rec}{rec}
\DeclareMathOperator{\tr}{tr}
\DeclareMathOperator{\im}{im}
\DeclareMathOperator{\coker}{coker}

\DeclareMathOperator{\rk}{rk}
\DeclareMathOperator{\pr}{pr}

\DeclareMathOperator{\Spec}{Spec}

\DeclareMathOperator{\Lie}{Lie}

\newcommand{\gal}[1]{\Gamma_{#1}} 
\newcommand{\Gal}{\mathrm{Gal}} 

\newcommand{\into}{\hookrightarrow}
\newcommand{\onto}{\twoheadrightarrow}

\newcommand{\mc}{\mathcal}
\newcommand{\mf}{\mathfrak}
\newcommand{\mr}{\mathrm}

\newcommand{\mbb}{\mathbb}

\newcommand{\br}{\bar{\rho}} 
\newcommand{\fg}{\mathfrak{g}}
\newcommand{\fgm}{\mathfrak{g}_{\mu}}
\newcommand{\fgder}{\mathfrak{g}^{\mathrm{der}}}
\newcommand{\fb}{\mathfrak{b}}

\DeclareMathOperator{\Lift}{\mathrm{Lift}}
\DeclareMathOperator{\Def}{\mathrm{Def}}
\newcommand{\Sets}{\mathbf{Sets}}

\newcommand{\tv}{\tilde{v}}

\newcommand{\tF}{\widetilde{F}}
\newcommand{\Tan}{\mathrm{Tan}}
\newcommand{\wt}{\widetilde}
\newcommand{\wh}{\widehat}
\newcommand{\inv}{\mathrm{inv}}

\newcommand{\un}[1]{\underline{#1}}

\title{Lifting irreducible Galois representations}

\thanks{We would like to thank Prakash Belkale and Arvind Nair for useful conversations and correspondence. We are very grateful to Michael Larsen for proving Lemma \ref{larsen} for us. C.K. was supported by NSF grant DMS-1601692 and a Humboldt Research Award, and would like to thank TIFR, Mumbai for its hospitality, in periods when some of the work was carried out. S.P. was supported by NSF grants DMS-1700759 and DMS-1752313.}

\begin{document}
\author[N.~Fakhruddin]{Najmuddin Fakhruddin}
\address{School of Mathematics, Tata Institute of Fundamental Research, Homi Bhabha Road, Mumbai 400005, INDIA}
\email{naf@math.tifr.res.in}
\author[C.~Khare]{Chandrashekhar Khare}
\address{UCLA Department of Mathematics, Box 951555, Los Angeles, CA 90095, USA}
\email{shekhar@math.ucla.edu}
\author[S.~Patrikis]{Stefan Patrikis}
\address{Department of Mathematics, The University of Utah, 155 S 1400 E, Salt Lake City, UT 84112, USA}
\email{patrikis@math.utah.edu}

\begin{abstract}
We study irreducible odd mod $p$ Galois representations $\br \colon \Gal(\overline{F}/F) \to G(\overline{\mathbb{F}}_p)$, for $F$ a totally real number field and $G$ a general reductive group. For $p \gg_{G, F} 0$, we show that any $\br$ satisfying a certain multiplicity-free condition on its adjoint representation, and satisfying some local ramification hypotheses, has a geometric $p$-adic lift. We also prove non-geometric lifting results without any oddness or multiplicity-free assumptions. 
\end{abstract}

\maketitle

\section{Introduction}
Let $k$ be a finite extension of $\mathbb{F}_p$, and let $\mc{O}= W(k)$ be its ring of Witt vectors. Let $G$ be a smooth group scheme over $\mc{O}$ such that $G^0$ is a split connected reductive group. The starting point of this paper is the following basic question:
\begin{question}\label{question}
Let $F$ be a number field with algebraic closure $\overline{F}$ and absolute Galois group $\gal{F}= \Gal(\overline{F}/F)$, and let $\br \colon \Gal(\overline{F}/F) \to G(k)$ be a continuous homomorphism. Does there exist a lift $\rho$
\[
\xymatrix{
& G(\mc{O}) \ar[d] \\
\gal{F} \ar@{-->}[ur]^{\rho} \ar[r]_{\br} & G(k) 
}
\]
that is geometric in the sense of Fontaine-Mazur?
\end{question}
This question has attracted a great deal of attention, at least since Serre proposed his modularity conjecture (\cite{serre:conjectures}). We begin by recalling a few instances of this general problem, beginning with Serre's conjecture. Serre proposed that every absolutely irreducible representation 
\[
\br \colon \gal{\Q} \to \mr{GL}_2(k)
\]
that was moreover \textit{odd} in the sense that $\det \br(c)=-1$ for any complex conjugation $c \in \gal{\Q}$ should be isomorphic to the mod $p$ reduction of a $p$-adic Galois representation attached to a classical modular eigenform. In particular, such a $\br$ should a admit a geometric $p$-adic lift. The papers \cite{khare-wintenberger:serre0}, \cite{khare:serrelevel1}, \cite{khare-wintenberger:serre1}, \cite{khare-wintenberger:serre2} proved Serre's modularity conjecture.  The proof uses as a key ingredient the modularity lifting results of Wiles and Taylor (\cite{wiles:fermat}, \cite{taylor-wiles:fermat}). In contrast, prior to the resolution of Serre's conjecture, Ramakrishna (\cite{ramakrishna:lifting}, \cite{ramakrishna02}) developed a beautiful, purely Galois-theoretic, method that in most cases settled Question \ref{question} in the setting of Serre's conjecture ($F= \Q$, $G= \mr{GL}_2$, $\br$ odd and absolutely irreducible). 

We might then turn to asking Question \ref{question} for $\br \colon \gal{\Q} \to \mr{GL}_2(k)$ that are \textit{even}, in the sense that $\det(\br(c))=1$. For instance, suppose that the image of $\br$ is $\mr{SL}_2(\mathbb{F}_p)$. Any geometric lift would (for $p \neq 2$) itself be even, and so conjecturally would be the $p$-adic representation $\rho$ attached to an algebraic Maass form. Such a $\rho$ should up to twist have finite image (because up to twist the associated motive should have Hodge realization of type $(0,0)$); but for $p > 5$, Dickson's classification of finite subgroups of $\mr{PGL}_2(\CC)$ rules out the possibility of such a lift. Thus one expects that $\br$ has no geometric lift. We have no general means of translating this conjectural heuristic into a proof, but Calegari (\cite[Theorem 5.1]{calegari:even2}) has given an ingenious argument that proves unconditionally that certain such even $\br$ have no geometric lift. 

In other settings, Question \ref{question} is even more mysterious. For instance, if $\mr{G}= \mr{GL}_2$ and $F/\Q$ is quadratic imaginary, we do not even have a reliable heuristic for predicting whether $\br \colon \gal{F} \to \mr{GL}_2(k)$ should have a geometric lift! It is a remarkable and widely-tested phenomenon that torsion cohomology (Hecke eigen-) classes for the locally symmetric spaces associated to congruence subgroups of $\mr{GL}_2/F$ need not lift to characteristic zero; one might hope that after raising the level (passing to a finite covering space of the arithmetic 3-manifold) they lift, and that the corresponding Galois-theoretic statement holds as well. But we have little evidence to support this.

This paper addresses cases of Question \ref{question} for general $G$, but for $\br$ that are \textit{odd} in a sense generalizing Serre's formulation for $\mr{GL}_2$. The following definition is essentially due to Gross (\cite{gross:odd}), who suggested parallels between this class of Galois representations and the ``odd" representations of Serre's original conjecture:
\begin{defn}\label{odd}
We say $\br \colon \gal{F} \to G(k)$ is odd if for all choices of complex conjugation $c_v$ (for $v \vert \infty$),
\[
\dim_k (\fgder)^{\Ad(\br(c_v))=1}= \dim \mr{Flag}_{G^0},
\]
where $\fgder$ is the Lie algebra of the derived group $G^{\mr{der}}$ of $G^0$, and $\mr{Flag}_{G^0}$ is the flag variety of $G^0$.
\end{defn}
Note that for any involution of $\fgder$, the dimension of the space of invariants must be at least $\dim \mr{Flag}_{G^0}$. An adjoint group contains an order 2 element whose invariants have dimension $\dim \mr{Flag}_{G^0}$ if and only if $-1$ belongs to the Weyl group of $G$. When $-1$ does not belong to the Weyl group, we can (after choosing a pinning) find such an order two element in $G \rtimes \Out(G)$; for more details, see \cite[\S 4.5, \S 10.1]{stp:exceptional}. Also note that the definition implies that $F$ is totally real. That said, the ``odd" case does have implications in certain CM settings. For example, let $F$ be quadratic imaginary, and let $\br \colon \gal{F} \to \mr{GL}_n(k)$ be an absolutely irreducible representation such that
\[
\br^c \equiv \br^\vee \otimes \mu|_{\gal{F}},
\]
where $\mu \colon \gal{\Q} \to k^\times$ is a character. Moreover assume that when we realize this essential conjugate self-duality as a relation
\[
\br(cgc^{-1})= A {}^t \br(g)^{-1} A^{-1} \mu(g)
\]
for some $A \in \mr{GL}_n(k)$ (and all $g \in \gal{F}$), the scalar $A \cdot {}^t A^{-1}$ (which is easily seen to be $\pm1$) actually equals $+1$. Then the pair $(\br, \mu)$ can be extended to a homomorphism 
\[
\bar{r} \colon \gal{\Q} \to (\mr{GL}_n \times \mr{GL}_1)(k) \rtimes \{1, j\},
\]
where $j^2=1$ and $j(g, a)j^{-1}= (a\cdot {}^t g^{-1}, a)$, and this $\bar{r}$ is odd in the sense of Definition \ref{odd}.

There are essentially two techniques for approaching cases of Question \ref{question}. For classical groups, automorphy lifting and potential automorphy theorems, via a technique introduced in \cite{khare-wintenberger:serre0}, yield the most robust results. For instance, the strongest lifting results in the previous example ($\br$ essentially conjugate self-dual over a quadratic imaginary field) follow from the work of Barnet-Lamb, Gee, Geraghty, and Taylor (\cite{blggt:potaut}). For general $G$, however, we have no understanding of automorphic Galois representations, and we must rely on purely Galois-theoretic methods. Ramakrishna developed the first such method in the paper \cite{ramakrishna02}, which, as noted above, resolved Question \ref{question} in the setting of Serre's original modularity conjecture ($F=\Q$, $G= \mr{GL}_2$, $\br$ odd and absolutely irreducible). Our work develops a broad generalization of Ramakrishna's ideas, also crucially building on the ``doubling method" of \cite{klr} and the work of Hamblen-Ramakrishna (\cite{ramakrishna-hamblen}).   

The greatest challenge in extending Ramakrishna's ideas to general groups is that they break down as the image of $\br$ gets smaller. This phenomenon is not particularly noticeable when $G= \mr{GL}_2$, since by a theorem of Dickson any irreducible subgroup of $\mr{GL}_2(k)$ (for $p \geq 7$) either has order prime to $p$, in which case one can take the ``Teichm\"{u}ller" lift, or has projective image conjugate to a subgroup of the form $\mr{PSL}_2(k')$ or $\mr{PGL}_2(k')$ for some finite extension $k'/\Fp$. This allows Ramakrishna to restrict to the case where the adjoint representation $\mr{ad}^0(\br)$ is absolutely irreducible. For higher-rank $G$, the global arguments of \cite{ramakrishna02} work with little change under the corresponding assumption that the adjoint representation $\br(\fgder)$ (this will be our notation for the Galois module $\fgder$, equipped with the action of $\Ad \circ \br$) is absolutely irreducible. Such a generalization is carried out in \cite{stp:exceptional}. The paper \cite{stp:exceptional} also proves a variant with somewhat smaller image, in which $\im(\br)$ contains (approximately) $\varphi(\mr{SL}_2(k))$, where $\varphi \colon \mr{SL}_2 \to G$ is a principal $\mr{SL}_2$. In this case $\br(\fgder)$ decomposes into $r$ irreducible factors, where $r$ is the semisimple rank of $G$, and the final result depended on an explicit analysis of this decomposition, requiring case-by-case calculations depending on the Dynkin type, with the result only verified for the exceptional groups via a computer calculation. More seriously, the method did not apply to groups of type $D_{2m}$, for which $\fgder$ is not multiplicity-free as an $\mr{SL}_2$-module (one factor occurs with multiplicity two). 


The present paper proves a lifting result for odd irreducible representations. This is done by generalizing the strategy and arguments of \cite{ramakrishna-hamblen}. Our very general setting entails a number of complications, which we overcome under our multiplicity hypothesis (needed only in \S \ref{auxiliarysection}). This multiplicity hypothesis is presumably superfluous and an artifact of the strategy. The methods of \cite{stp:exceptional} generalized the arguments of \cite{ramakrishna02},  while ours in generalizing  the more elaborate arguments of \cite{ramakrishna-hamblen} prove a lifting result with significantly less stringent hypotheses on the residual representation. One might compare the progressive relaxation of global image hypotheses in higher-rank automorphy lifting and potential automorphy theorems. In \cite{clozel-harris-taylor}, Clozel-Harris-Taylor proved a theorem for $\br$ satisfying a very technical ``bigness" condition, which was then by explicit calculation shown to be sufficient for the desired applications to the Sato-Tate conjecture for elliptic curves over $\Q$. Thorne (\cite{thorne:adequate}) weakened this assumption to one of ``adequacy," and then some highly non-trivial finite group theory (Guralnick-Herzig-Thorne-Taylor, \cite[Appendix]{thorne:adequate}) shows that for $p$ sufficiently large, absolute irreducibility implies adequacy. The image difficulties are significantly more troublesome in the Ramakrishna-style lifting methods: in the Taylor-Wiles method, one chooses auxiliary primes to kill a dual Selmer group and allows the Selmer group to grow (and then the automorphic theory does much of the heavy-lifting). In the purely Galois-theoretic methods, one must choose auxiliary primes and (in contrast to Taylor-Wiles conditions) \textit{formally smooth} local conditions at these primes that kill both Selmer and dual Selmer groups. Balancing these demands becomes considerably more challenging for $\br$ with smaller image.

Before describing the method in more detail, we will state our main theorem, which we prove in Theorem \ref{mainthm}. First we emphasize what we do not do in this paper: Ramakrishna's method and its variants require both local arguments--the existence of certain formally smooth and ``sufficiently large" local deformation problems at the primes of ramification of $\br$--and the global (or if one prefers, local-to-global) arguments that produce the final lifts. We do not undertake any of the local theory here, except what is required by our new auxiliary prime arguments. The local theory--at least as is needed for the global applications--is essentially complete at primes $\ell \neq p$ for groups of classical type,\footnote{With a couple of caveats, one specific to our paper and one general: our arguments do not allow us to replace $k$ with a finite extensions, which is the generality in which \cite{clozel-harris-taylor} and \cite{booher:minimal} are written (see Remark \ref{lnotplocal}); and when $-1$ does not belong to the Weyl group of $G$ (i.e. in types $A_n$, $n \geq 2$, $D_n$, $n \geq 3$, and $E_6$), the oddness requirement means working with a group of the form $G^0 \rtimes \Out(G)$, and here the theory for classical groups would only be complete at primes $v$ of ramification of $\br$ such that $\br(\gal{F_v})$ is contained in $G^0(k)$.} by \cite[\S 2.4.4]{clozel-harris-taylor} and \cite[\S 6]{booher:minimal}, but while some examples of the theory we need are known for exceptional type (see e.g. \cite[\S 4.3, 4.4]{stp:exceptional} and \cite[\S 4.2]{stp:exceptional2}), there is as yet no general theory. For $\ell=p$, much work remains to be done, although again we have good results in classical type (\cite[\S 2.4.1]{clozel-harris-taylor}, \cite[\S 5]{booher:JNT}), and some results for all groups (\cite[4.3, 4.4]{stp:exceptional}). 

From now on we will require of $G$ that the component group $\pi_0(G)$ is finite \'{e}tale of order prime to $p$. The following is a somewhat specialized version of our main theorem; for the complete statement, see Theorem \ref{mainthm}.
\begin{thm}\label{mainthmintro}
Let $F$ be a totally real field, and let $\br \colon \gal{F, S} \to G(k)$ be a continuous representation unramified outside a finite set $S$ of finite places containing the places above $p$. Let $\tF$ be the smallest extension of $F$ such that $\br(\gal{\tF})$ is contained in $G^0(k)$. Assume that $p \gg_{G, F} 0$ and that $\br$ satisfies the following:
\begin{itemize}
 \item $\br$ is odd, i.e. for all infinite places $v$ of $F$, $h^0(\gal{F_v}, \br(\fgder))= \dim(\mr{Flag}_{G^{\mr{der}}})$.
 \item $\br|_{\gal{\tF(\zeta_p)}}$ is absolutely irreducible.
 \item For each $G$-orbit of simple factors $\br(\oplus_i \fg_i)$ of $\br(\fgder)$, each $\br(\fg_i)$ is multiplicity-free as $\Fp[\gal{\tF}]$-module. Moreover, each simple $\Fp[\gal{\tF}]$-constituent $W$ of $\br(\fgder)$ satisfies $\End_{\Fp[\gal{\tF}]}(W) \cong k$. (But see Remark \ref{intrormk}.)
 \item For all $v \vert p$, $F_v$ does not contain $\zeta_p$, and $\br|_{\gal{F_v}}$ either 
 \begin{itemize}
 \item is trivial; or
 \item is ordinary in the sense of \S \ref{ordsection} and satisfies the conditions (REG) and (REG*); or
 \item $G^0$ is a product of groups of the form $\mr{GL}_N$, $\mr{GSp}_N$, and $\mr{GO}_N$, and the projection of $\br|_{\gal{F_v}}$ to each factor is Fontaine-Laffaille with distinct Hodge-Tate weights in an interval of length less than $p-1$ in the $\mr{GL}$ case and less than $\frac{p-1}{2}$ in the $\mr{GSp}$ and $\mr{GO}$ cases. 
 \end{itemize}
 \item The field $K= \tF(\br(\fgder), \mu_p)$ does not contain $\mu_{p^2}$ (which follows in many cases from the preceding condition at $v \vert p$: see Remark \ref{mup2}).
 \item First assume $G$ is connected. Then for all $v \in S$ not above $p$, there is a formally smooth local deformation condition $\mc{P}_v$ for $\br|_{\gal{F_v}}$ whose tangent space $\Tan^{\mc{P}_v}_{\br|_{\gal{F_v}}} \subset H^1(\gal{F_v}, \br(\mf{g}_\mu))$ has dimension $h^0(\gal{F_v}, \br(\fgm))$. If $G$ is not connected, see Theorem \ref{mainthm} for the precise hypothesis (a minor variant of the condition just stated).
\end{itemize}
Then there exist a finite set of places $T \supset S$ and a geometric lift $\rho$ of $\br$:
 \[
 \xymatrix{
 & G(\mc{O}) \ar[d] \\
 \gal{F, T} \ar@{-->}[ur]^\rho \ar[r]_{\br} & G(k)
 }
 \]
 such that $\rho(\gal{F})$ contains $\wh{G^{\mr{der}}}(\mc{O})=\ker(G^{\mr{der}}(\mc{O}) \to G(k))$, and in particular the Zariski-closure of $\rho(\gal{F})$ contains $G^{\mr{der}}$.
\end{thm}
\begin{rmk}\label{intrormk}
\begin{itemize}
\item The most serious constraint in the theorem is the multiplicity-free hypothesis on the decomposition of (suitable constituents of) $\br(\fgder)$. This intervenes only in \S \ref{auxiliarysection} of the paper; the other arguments uniformly treat any irreducible $\br$. In fact, we prove something more general than Theorem \ref{mainthmintro}, and can allow certain simple constituents $W$ to appear with arbitrary multiplicity: roughly speaking, our arguments fail when a simple constituent $W$ of $\br(\fgder)$ appears with multiplicity greater than 1 \textit{and} the sequence $1 \to W \to \Gamma_2 \to \im(\br) \to 1$ splits, where $\Gamma_2$ is formed by taking the preimage of $\im(\br)$ in $G(\mc{O}/p^2)$ and then quotienting out by an $\Fp[\gal{\tF}]$-complement to $W$ in $\widehat{G}(\mc{O}/p^2)$. We refer the reader to Theorem \ref{mainthm} for the precise set of conditions.
\item The arguments proceed from a somewhat different global image assumption--see Assumption \ref{multfree}--but for $p \gg_G 0$ the absolute irreducibility hypothesis implies the other conditions in Assumption \ref{multfree} (see Lemmas \ref{irr} and \ref{invariants}). 
\item The bound on $p$ can be made effective: for detailed remarks, explaining the contributions of $F$ and $G$, see Remark \ref{effective}.
\item In \S \ref{examples} we give some examples of the theorem.
\item All the lifts $\rho$ produced by the theorem have image whose Zariski closure contains $G^{\mr{der}}$; that is, we find lifts whose image is ``as large as possible" subject to the given $\im(\br)$.
\item The ability to allow $\br|_{\gal{F_v}}$ to be trivial for $v \vert p$ is a curious consequence of arguments similar to our new auxiliary prime arguments; even for $G=\mr{GL}_2$, this strengthens Ramakrishna's original results ((\cite{ramakrishna02})), which forbid $\br|_{\gal{\Q_p}}=1$. 
\end{itemize}
\end{rmk}
We mention two variants of the theorem that are straight-forward given our techniques. The first (see Theorem \ref{notodd}) is a non-geometric but finitely-ramified lifting theorem for $\br$ without any constraints on $\br(c_v)$ for $v \vert \infty$ (and in particular allowing $F$ to be any number field); this holds under the same image hypotheses as Theorem \ref{mainthmintro}. The second removes the multiplicity-free hypothesis from Theorem \ref{mainthmintro} but produces infinitely-ramified lifts, generalizing the main theorem of \cite{klr} from the case $\mr{SL}_2(\Fp) \subset \im(\br) \subset \mr{GL}_2(\Fp)$:
\begin{thm}[See Corollary \ref{infiniteram}]
Let $F$ be any number field. Assume $p \gg_G 0$, and let $\br \colon \gal{F, S} \to G(k)$ be a representation such that $\br|_{\gal{\tF(\zeta_p)}}$ is absolutely irreducible. Assume for simplicity that $G=G^0$, and fix a lift $\nu \colon \gal{F, S} \to A(\mc{O})$ of $\mu \circ \br$. Assume that for all $v \in S$, there are lifts $\rho_v \colon \gal{F_v} \to G(\mc{O})$ of $\br|_{\gal{F_v}}$ with multiplier $\nu$. Then there exists an infinitely ramified lift
\[
\xymatrix{
& G(\mc{O}) \ar[d] \\
\gal{F} \ar@{-->}[ur]^{\rho} \ar[r]_{\br} & G(k)
}
\]
such that $\rho|_{\gal{F_v}}= \rho_v$ for all $v \in S$, and $\rho(\gal{F})$ contains $\wh{G^{\mr{der}}}(\mc{O})$.
\end{thm}

We end the introduction by outlining our techniques. We first briefly recall the original technique of Ramakrishna, as neatly formulated by Taylor (\cite{taylor:icos2}). Under the oddness hypothesis, one defines a global Galois deformation problem by imposing formally smooth local deformation conditions on the restriction of $\br$ to primes in $S$, and whose associated Selmer and dual Selmer groups have the same dimension (we will informally say that Selmer and dual Selmer are ``balanced"). In this setting an application of the Selmer group variant of the Poitou-Tate sequence (see \cite[Lemma 1.1]{taylor:icos2}) implies that if the dual Selmer group vanishes, then the corresponding universal deformation ring is $\mc{O}$ and therefore attests to the existence of a geometric deformation of $\br$. The task, then, is to allow ramification at a set $Q$ of auxiliary primes such that
\begin{itemize}
\item the allowed ramification at each $q \in Q$ is a formally smooth local condition;
\item the conditions at $q \in Q$ have large enough tangent space that the resulting new Selmer and dual Selmer groups remain ``balanced" as we add each $q \in Q$; and
\item when we have allowed the entire auxiliary set $Q$ of ramification, the dual Selmer group, and hence the Selmer group, vanish. 
\end{itemize}
Ramakrishna takes a Steinberg local condition at primes $q \not \equiv 1 \pmod p$ at which $\br$ is unramified with distinct Frobenius eigenvalues with ratio $q$. By comparing splitting conditions on Selmer and dual Selmer classes, he shows (when the projective image of $\br$ contains $\mr{PSL}_2(k)$) that such $q$ can be chosen that inductively decrease the size of Selmer and dual Selmer.

In higher rank, it is better to think of the $\mr{GL}_2$ ``Steinberg" condition as allowing unipotent ramification in the direction of a fixed root space and constraining Frobenius to act by the cyclotomic character on this root space, since a key point in controlling the Selmer and dual Selmer groups simultaneously is that the Ramakrishna deformation condition should intersect the unramified condition in a codimension one subspace. At each step of the inductive argument that decreases the size of the Selmer groups, one has to, given non-zero Selmer and dual Selmer cocycles, be able to choose this root space in a suitably general position with respect to the images of the cocycles. This is not always possible when (as in \cite{stp:exceptional}) the auxiliary primes are chosen so that Frobenius acts by a regular semisimple element. In higher rank, with small residual image, balancing these simultaneous demands---for auxiliary primes $q$ where the shape of $\br(\mr{Frob}_q)$ allows us to define a formally smooth local condition with rank 1 unipotent ramification in a direction adapted to killing fixed Selmer and dual Selmer classes, regardless of which constituents of the adjoint representation support them---seems to force on us an approach quite different from that of \cite{ramakrishna02} and \cite{stp:exceptional}. We resort to using auxiliary primes $q$ having the one behavior we are guaranteed to find in the image of any representation, namely, that $\br|_{\gal{F_q}}$ is trivial; note that as $\br(\mr{Frob}_q)$ is then contained in \textit{every} maximal torus of $G$, we win a great deal of flexibility in the choice of root space in which to allow ramification (contrast the condition (6) in \cite[\S 5]{stp:exceptional} with our Proposition \ref{splitcase}). 

More precisely, we generalize the notion of \textit{trivial primes} from the work of Hamblen and Ramakrishna (\cite{ramakrishna-hamblen}). Hamblen and Ramakrishna show how to deform a \textit{reducible} but indecomposable representation $\gal{\Q} \to \mr{GL}_2(k)$ to an irreducible $p$-adic representation by allowing Steinberg-type ramification at primes $q$ such that $q \equiv 1 \pmod p$, $q \not \equiv 1 \pmod {p^2}$, and $\br|_{\gal{\Q_q}}$ is trivial. The resulting local condition on lifts of $\br|_{\gal{\Q_q}}$ is liftable but is not representable, the latter point being reflected in the fact that the local condition behaves very differently modulo different powers of $p$: while its tangent space is ``too small" for the global applications, certain lifts mod $p^m$ for $m \geq 3$ do indeed witness that the condition is coming from a sufficiently large characteristic zero condition (see Lemma \ref{extracocycles} for a precise formulation of this distinction). 

The consequence of this distinction is that the global argument must treat separately the problems of lifting $\br$ to a mod $p^3$ representation and lifting it modulo higher powers of $p$. We treat these two problems in \S \ref{klrsection} and \S \ref{auxiliarysection}. In \S \ref{klrsection}, we start with any mod ${p^2}$ lift $\rho_2$ of $\br$ (easily seen to exist after enlarging $S$ by a set of trivial primes); to lift it mod $p^3$ we have to modify $\rho_2$ so that all of its local restrictions satisfy formally smooth (and ``large enough") local conditions. This leads to the following question: given local cohomology classes $z_T= (z_w)_{w \in T} \in \bigoplus_{w \in T} H^1(\gal{F_w}, \br(\fgder))$ (here $T$ will be a finite set of primes containing the original set $S$ of ramification), can we find a global class $h \in H^1(\gal{F, T}, \bf(\fgder))$ such that $h|_{\gal{F_w}}=z_w$ for all $w \in T$? The answer is no, so we aim for the next best thing: to enlarge $T$ to a finite set $T \cup U$, and to find a class $h^U \in H^1(\gal{F, T \cup U}, \br(\fgder))$ such that $h^U|_T= z_T$. This would allow us to modify $\rho_2$ to some $(1+ph^U)\rho_2$ that is well-behaved at primes in $T$. The problem here is that we sacrifice control at the primes in $U$, and this necessitates the use of an idea from \cite{klr} (as exploited in a simpler setting by \cite{ramakrishna-hamblen}), which we will refer to as the ``doubling method": roughly speaking, we consider two such sets $U$ and $U'$, with corresponding cocycles $h^U$ and $h^{U'}$. By considering all possibilities $(1+p(2h^U-ph^{U'}))\rho_2$ as $U$ and $U'$ vary (each through \v{C}ebotarev multi-sets of primes), we show by a limiting argument that there is a pair (not, in fact, a \v{C}ebotarev set!) of $U$ and $U'$ such that $\rho'_2= (1+p(2h^U-ph^{U'}))\rho_2$ both has the desired behavior at $T$ and is under enough control at $U$ and $U'$ (the detailed desiderata come out of Definition \ref{ramtrivlifts}, Lemma \ref{ramextracocycles}, and Lemma \ref{trivfixeddet}) that we can find a mod $p^3$ lift $\rho_3$ of $\rho_2'$. In these arguments, handling the case of general $\im(\br)$ poses a significant challenge beyond the $\mr{GL}_2$ arguments of \cite{klr} and \cite{ramakrishna-hamblen}; in particular, handling multiplicities (which we do in complete generality) in the $\Fp[\gal{F}]$-decomposition of $\br(\fgder)$ requires new techniques.

We then proceed to the argument of \S \ref{auxiliarysection}, which explains how, starting from a well-chosen mod $p^2$ representation (the $\rho_2'$ of the previous paragraph), to use trivial primes $q$, with the added requirement that $\rho_2'|_{\gal{F_q}}$ has the form demanded by Definition \ref{unrtrivlifts} and Lemma \ref{extracocycles}, to annihilate Selmer and dual Selmer classes. Again, handling the case of general $\im(\br)$ poses real difficulties and leads to the multiplicity constraint in our main theorem. The central argument here is Proposition \ref{splitcase}, which exploits the flexibility of using trivial primes to achieve simultaneous control of \v{C}ebotarev conditions in the fixed fields of $\rho_2$ and of a given Selmer class, even when these fixed fields are not linearly disjoint over the fixed field of $\br$. Namely, in Ramakrishna's original arguments, $\br(\mr{Frob}_q)$ is a regular semisimple element at the auxiliary primes $q$; it is therefore contained in a unique maximal torus, and (for $q \not \equiv \pm 1 \pmod p$) there is a unique choice of non-empty Steinberg deformation condition associated to $\br|_{\gal{F_q}}$. In contrast, at trivial primes, $\br|_{\gal{F_q}}$ is contained in \textit{every} maximal torus of $G$. It transpires that our search for auxiliary primes is actually a simultaneous search through triples $(T, \alpha, q)$ where $T$ is a maximal torus, $\alpha$ is a root with respect to $T$, and $q$ is a prime. These features of the problem do not appear in \cite{ramakrishna-hamblen}, where the relevant torus for the local theory is identified using the global theory (recall that Hamblen-Ramakrishna work with a non-split globally reducible 2-dimensional representation).

In \S \ref{groupsection} we present the group theory arguments needed to streamline some of the global hypotheses on $\br$ in \S \ref{klrsection} and \S \ref{auxiliarysection} to an irreducibility hypothesis. We combine the arguments of \S \ref{klrsection} and \S \ref{auxiliarysection} to complete the global argument in \S \ref{finallift}. This last step at least is fairly routine. In \S \ref{examples} we gather a few examples of the main theorem. Finally, we remark that a number of arguments are made technically more intricate by the fact that we have worked with groups $G$ having arbitrary (order prime to $p$) component group. The reader interested in the essential number-theoretic novelties of our arguments would do well to focus on the case of connected adjoint groups $G$.

\subsection{Notation and conventions}\label{notation}
We embed local Galois groups into global Galois groups by fixing embeddings $\overline{F} \into \overline{F}_v$. We write $\kappa$ for the $p$-adic cyclotomic character and $\bar{\kappa}$ for its mod $p$ reduction. We once and for all fix a primitive $p^{th}$ root of unity $\zeta \in \mu_p(\overline{F}) \xrightarrow{\sim} \mu_p(\overline{F}_v)$, and this allows us to identify the Tate dual $V^*= \Hom(V, \mu_p(\overline{F}))$ of an $\Fp[\gal{F}]$-module $V$ with $\Hom(V, \Fp(\bar{\kappa}))$. The reader should always assume we are doing this; only in the proof of Lemma \ref{localduality} will we make the identifications explicit.

For any finite set of primes $S$ of $F$, we let $\gal{F, S}$ denote $\Gal(F(S)/F)$, where $F(S)$ is the maximal extension of $F$ inside $\overline{F}$ that is unramified outside the primes in $S$; here we impose no constraint on the ``ramification" at $\infty$, but for notational convenience we do not want the set $S$ to contain the archimedean places (as would often be the convention for what we are referring to as $\gal{F, S}$).

\section{Deformation theory preliminaries}\label{defprelimsection}
Let $G$ be a smooth group scheme over $\mc{O}$ such that $G^0$ is split connected reductive, and $G/G^0$ is finite \'{e}tale of order prime to $p$. We will sometimes write $\pi_0(G)$ for this quotient $G/G^0$.
Write $\mu \colon G \to A$ for the map from $G$ to its maximal abelian quotient $A$, and let $G_{\mu}$ be $\ker(\mu) \subseteq G$. We let $G^{\mr{der}}$ denote the derived group of $G^0$; note that $G^{\mr{der}}$ is also the derived group of $G_{\mu}^0$, but that the latter is not necessarily semisimple. We denote by $\fgder$ and $\fgm$ the Lie algebras of $G^{\mr{der}}$ and $G_{\mu}$, and we let $\mf{z}_{\mu}$ be the Lie algebra of the center $Z_{G_{\mu}^0}$ of $G_{\mu}^0$. \textbf{The following assumptions on $p$ will implicitly be in effect for the remainder of the paper:}
\begin{assumption}\label{minimalp}
We assume that $p \neq 2$ is very good (\cite[\S 1.14]{carter:finitelie}) for $G^{\mr{der}}$, which in particular holds if $p \geq 7$ and $p \nmid n+1$ whenever $G^{\mr{der}}$ has a simple factor of type $A_n$. We also assume for simplicity that the canonical central isogenies $G^{\mr{der}} \times Z_G^0 \to G^0$, and similarly for $G_{\mu}^0$, have kernels of order prime to $p$ (in particular are \'{e}tale). Finally, we assume $p$ does not divide the order of the torsion subgroup of $\coker(X^\bullet(A^0) \to X^\bullet(Z^0_{G^0}))$.
\end{assumption}
Then in particular we have $G$-equivariant direct sum decompositions $\fgm= \fgder \oplus \mf{z}_\mu$ and $\fg= \fgm \oplus \mf{a}$, $\fgder$ is irreducible as $G_\mu^0$-representation, and there is a non-degenerate $G$-invariant trace form $\fgder \times \fgder \to k$ (\cite[1.16]{carter:finitelie}). The isogeny $G^{\mr{der}} \to G^{\mr{ad}}$ to the adjoint group of $G^{\mr{der}}$ also induces an isomorphism on Lie algebras. We will moreover assume that $\fgm^G=0$: if $G$ is connected, this condition says that $(\fgder)^G=0$, which follows from the ``very good" hypothesis; but in general it is an additional condition (at least on $p$), as can be seen by taking $G$ to be the normalizer of a maximal torus in $\mr{SL}_2$ (then $\fgm^G \neq 0$ if $\mr{char}(k)=2$, whereas all primes are very good for $G$, since the root system of $G^0$ is trivial).

Let $\Gamma$ be a profinite group, and let $\br \colon \Gamma \to G(k)$ be a continuous homomorphism. Set $\bar{\nu}= \mu \circ \br$, and once and for all fix a lift $\nu \colon \Gamma \to A(\mc{O})$ of $\bar{\nu}$. Let $\mc{C}_{\mc{O}}$ be the category of complete local noetherian algebras $R$ with $\mc{O} \to R$ inducing an isomorphism of residue fields (and morphisms the local homomorphisms), and let $\mc{C}_{\mc{O}}^f$ be the full subcategory of those algebras that are artinian.  Note that for any $R \in \mc{C}_{\mc{O}}$, $\pi_0(G)(R) \xrightarrow{\sim} \pi_0(G)(k)$, so we will just identify any $\pi_0(G)(R)$ to this fixed finite group $\pi_0(G)$. 

Define the lifting and deformation functors 
\[
\Lift_{\br}, \Def_{\br}, \Lift_{\br}^\nu, \Def_{\br}^\nu \colon \mc{C}_{\mc{O}} \to \Sets 
\]
by letting $\Lift_{\br}(R)$ be the set of lifts of $\br$ to $G(R)$, and by letting $\Lift_{\br}^\nu(R) \subset \Lift_{\br}(R)$ be the subset of lifts $\rho$ such that $\mu \circ \rho= \nu$; and then letting the corresponding deformation functors be the quotients by the equivalence relation 
\[
\text{$\rho \sim \rho' \iff \rho= g \rho' g^{-1}$ for some $g \in \wh{G}(R)= \ker(G(R) \to G(k))$.} 
\]
The tangent spaces of the lifting functors are canonically isomorphic to $Z^1(\Gamma, \br(\fg))$ and $Z^1(\Gamma, \br(\fg_{\mu})$; the tangent spaces of the corresponding deformation functors (note that $\wh{G} \subset G^0$) are canonically isomorphic to $H^1(\Gamma, \br(\fg))$ and $H^1(\Gamma, \br(\fgm))$, and (by the remarks in the first paragraph of this section) the latter is a direct summand of the former. In some cases we will have a global Galois representation valued in a non-connected group $G$, but it will be convenient to develop certain local deformation conditions only for the group $G^0$: since $\wh{G}$ is contained in $G^0$ (and as above $\pi_0(G)$ has order prime to $p$), a $G^0$-deformation of a $G^0(k)$-valued $\br$ is exactly the same thing as a $G$-deformation of a $G^0$-valued $\br$. 

As usual, when $R \to R/I$ is a small extension the obstruction to lifting a $\rho \in \Lift_{\br}(R/I)$ to a $\tilde{\rho} \in \Lift_{\br}(R)$ is a class in $H^2(\Gamma, \br(\fg) \otimes_k I)$ (the two-cocyle one defines by choosing a topological lift of $\rho$ to $G(R)$ takes values in $\ker(G(R) \to G(R/I))= \ker( G^0(R) \to G^0(R/I))= \exp(\fg \otimes_k I)$).
\section{Local deformation theory}
\subsection{Trivial primes}\label{trivialsection}
Let $F/\Q_{\ell}$ be a finite extension with residue field of order $q$. Assume $q$ is 1 mod $p$ but not 1 mod $p^2$, and let $\br \colon \gal{F} \to G(k)$ be the trivial homomorphism; in particular, all lifts of $\br$ land in $G^0$. Moreover, all lifts of $\br$ factor through the quotient of $\gal{F}$ topologically generated by a lift $\sigma$ of (arithmetic) Frobenius and a generator $\tau$ of the $p$-part of the tame inertia group. At one point we will invoke a calculation (Lemma \ref{localduality}) that depends on the normalization of $\tau$. Suppose we have a fixed $p^{th}$ root of unity $\zeta \in \mu_p(F_v)$ (in the global setting, this will come from a global choice, as in \S \ref{notation}). We then choose $\tau$ such that for any uniformizer $\varpi$ of $F$, $\frac{\tau(\varpi^{1/p})}{\varpi^{1/p}}= \zeta$.

We will now define the kinds of local lifts of $\br$ that we will make use of at auxiliary primes. First, we introduce some notation. For a split maximal torus $T$ of $G^0$ (over $\mc{O}$) and an $\alpha \in \Phi(G^0, T)$, we let $U_{\alpha} \subset G^0$ denote the root subgroup that is the image of the root homomorphism (``exponential mapping") $u_{\alpha} \colon \mf{g}_{\alpha} \to G$. The homomorphism $u_{\alpha}$ is a $T$-equivariant isomorphism $\mf{g}_{\alpha} \to U_{\alpha}$ (see \cite[Theorem 4.1.4]{conrad:luminy}), and its characterizing properties (\textit{loc. cit.}) imply that $Z_{G^0}(\mf{g}_{\alpha})= Z_{G^0}(U_{\alpha})$.
\begin{defn}\label{triviallifts}
Fix a split maximal torus $T$ of $G^0$ (over $\mc{O}$) and an $\alpha \in \Phi(G^0, T)$. Define $\Lift^{\alpha}_{\br}(R)$ to be the set of lifts $\widehat{G}(R)$-conjugate to one satisfying 
\begin{itemize}
\item $\rho(\sigma) \in T \cdot Z_{G^0}(\mf{g}_{\alpha})(R)$
\item Under the composite (note that $T$ normalizes the centralizer) 
\[
T \cdot Z_{G^0}(\mf{g}_{\alpha})(R) \to T(R)/(T(R) \cap Z_{G^0}(\mf{g}_{\alpha})(R)) \xrightarrow{\alpha} R^\times,
\] 
$\rho(\sigma)$ maps to $q$.
\item $\rho(\tau) \in U_{\alpha}(R)$.
\end{itemize}
\end{defn}
\begin{lemma}\label{trivsmooth}
 For any pair $(T, \alpha)$ consisting of a split maximal torus $T$ of $G^0$ and an $\alpha \in \Phi(G^0, T)$, the functor $\Lift^{\alpha}_{\br}$ is formally smooth, i.e. for all maps $R \to R/I$ in $\mc{C}^f_\mc{O}$ with $I \cdot \mf{m}_R=0$, $\Lift_{\br}^{\alpha}(R) \to \Lift^{\alpha}_{\br}(R/I)$ is surjective.
\end{lemma}
\proof
It is convenient to begin with a slightly different description of $\Lift_{\br}^{\alpha}$ that will circumvent the need to know that $Z_{G^0}(\mf{g}_{\alpha})$ is smooth over $\mc{O}$. To that end, let $Z_{\alpha}$ be the open subscheme of $Z_{G^0}(\mf{g}_{\alpha})$ obtained by removing all non-identity components of the special fiber $Z_{G^0_k}(\mf{g}_{\alpha} \otimes_{\mc{O}} k)$. Set $\mf{g}_{\alpha, k}=\mf{g}_{\alpha} \otimes_{\mc{O}} k$. We first claim the special fiber $Z_{\alpha, k} \to \Spec k$ is smooth. By our assumptions on $p$, $Z_{G^0_k}(\mf{g}_{\alpha, k})$ is smooth if and only if $Z_{G^{\mr{der}}_k}(\mf{g}_{\alpha, k})$ is smooth, and then the assumption that $p$ is very good for $G^{\mr{der}}$ implies, by a criterion of Richardson (\cite[Theorem 2.5]{jantzen:nilporbits}), that $Z_{G^{\mr{der}}_k}(\mf{g}_{\alpha, k})$ is smooth (recall that $\fgder$ has a non-degenerate trace form). In particular, $Z_{\alpha, k}$ is smooth. Since $Z_{\alpha, k}$ has a single irreducible component, we can now apply \cite[Remark 4.3, Lemma 4.4]{booher:minimal} to deduce that $Z_{\alpha} \to \Spec \mc{O}$ is smooth.

We next claim that $\Lift^{\alpha}_{\br}$ is equivalently defined by replacing $Z_{G^0}(\mf{g}_{\alpha})$ with $Z_{\alpha}$ in Definition \ref{triviallifts}. First note that for any object $R$ of $\mc{C}^f_{\mc{O}}$, the fiber over the identity of $Z_{G^0}(\mf{g}_{\alpha})(R) \to Z_{G^0}(\mf{g}_{\alpha})(k)$ is contained in $Z_{\alpha}(R)$, and that $T$ normalizes $Z_{\alpha}$ (as functors of Artin rings). Now let $x \in T(R)Z_{G^0}(\mf{g}_{\alpha})(R)$ be an element in the fiber over $1 \in G(k)$, and correspondingly write $x=t\cdot c$. Writing $\bar{c}$ for the image of $c$ in $G(k)$, we have $\bar{c} \in \ker(\alpha|_T)$. This kernel is smooth (our assumptions on $p$ imply that $X^{\bullet}(T)/\Z \alpha$ has no $p$-torsion), so we can lift $\bar{c}$ to an element $t' \in \ker(\alpha|_T)(R)$. Then writing $x=(tt')({t'}^{-1}c)$ we have exhibited $x$ as an element of $T(R) Z_{\alpha}(R)$. Since $\br$ is trivial, we conclude that $\Lift^{\alpha}_{\br}$ can equivalently be defined with $Z_{\alpha}$ in place of $Z_{G^0}(\mf{g}_{\alpha})$.

With this reinterpretation, we can now check formal smoothness of $\Lift_{\br}^{\alpha}$. Let $\rho$ be any element of $\Lift^{\alpha}_{\br}(R/I)$. Since $\widehat{G}$ is formally smooth, we may assume $\rho$ satisfies the three bulleted items of Definition \ref{triviallifts}. Write $\rho(\sigma)=t_{\sigma}c_{\sigma}$ and $\rho(\tau)=u_{\alpha}(x)$ for some $t_{\sigma} \in T(R/I)$ satisfying $\alpha(t_{\sigma})=q$, $c_{\sigma} \in Z_{\alpha}(R/I)$, and $x \in R/I$. Since $T$ and $Z_{\alpha}$ are formally smooth, we can choose lifts $\wt{t}_{\sigma} \in T(R)$, $\wt{c}_{\sigma} \in Z_{\alpha}(R)$, and $\wt{x} \in R$. We can write $\alpha(\wt{t}_{\sigma})=q+i$ for some $i \in I$, and then we replace $\wt{t}_{\sigma}$ by $\wt{t}_{\sigma} \alpha^\vee(1-\frac{i}{2})$ ($p \neq 2$). Since $I \cdot \mf{m}_R=0$, we then find that $\tilde{\rho}(\sigma)=\wt{t}_{\sigma} \wt{c}_{\sigma}$, $\tilde{\rho}(\tau)=u_{\alpha}(\wt{x})$ defines a lift $\tilde{\rho} \in \Lift^{\alpha}_{\br}(R)$ of $\rho$.
\endproof
\begin{rmk}
The lemma does not rely on the assumption (relevant for later arguments in this section) that $q \not \equiv 1 \pmod{p^2}$. We will make use of this flexibility in Corollary \ref{infiniteram} (see too Remark \ref{infinitetriv}).
\end{rmk}
\begin{rmk}
We could have argued directly with the original definition using $Z_{G^0}(\mf{g}_{\alpha})$, but lacking a generalization to all groups of \cite[\S 4.4]{booher:minimal}---namely, sections of $Z_{G^0}(\mf{g}_{\alpha})$ hitting any irreducible component in the special fiber---we would only have obtained the smoothness for $p \gg_G 0$ but non-effective (resorting to a spreading-out argument).
\end{rmk}

Write $\Tan^{\alpha}_{\br}$ for the tangent space of this lifting functor. It is easy to see that $\Tan^{\alpha}_{\br}$ identifies to the set of homomorphisms $\phi \colon \Gamma_F \to \mf{g}$ such that 
\begin{align*}
\phi(\sigma) &\in \ker(\alpha|_{\mf{t}}) + \mr{Cent}_{\mf{g}}(\mf{g}_{\alpha}) \\
\phi(\tau) &\in \mf{g}_{\alpha}.
\end{align*}
The dimension will depend on $\alpha$, but it is certainly less than $\dim \mf{g}= h^0(\gal{F}/I_F, \br(\fg))$, so for our global argument we will need to show that there are ``extra cocycles" stabilizing certain subsets of $\Lift^{\alpha}_{\br}(\mc{O}/p^m)$ for $m \geq 3$. We will in fact need two such constructions, given in Lemma \ref{extracocycles} and Lemma \ref{ramextracocycles}. 

First, note that $\oplus_{\beta \in \Phi(G, T)} \mf{g}_{\beta}$ is actually the direct sum 
\[
 \bigoplus_{\beta \in \Phi(G, T)} \mf{g}_{\beta}= \left( \bigoplus_{\beta \in \Phi(G, T)} \mf{g}_{\beta} \cap \mr{Cent}_{\mf{g}}(\mf{g}_{\alpha})\right) \oplus \bigoplus_{\beta \in \Phi^{\alpha}} \mf{g}_{\beta},
\]
where $\Phi^{\alpha}$ is the subset of roots $\beta$ such that $[\mf{g}_{\alpha}, \mf{g}_{\beta}] \neq 0$.
\begin{defn}\label{unrtrivlifts}
 For any $m \geq 2$, let $\Lift^{\alpha}_{\br, 2}(\mc{O}/p^m) \subset \Lift^{\alpha}_{\br}(\mc{O}/p^m)$ be the subset of lifts $\rho$ such that $\rho_2 := \rho \pmod {p^2}$ satisfies:
\begin{itemize}
 \item $\rho_2$ is unramified, with $\rho_2(\sigma) \in T(\mc{O}/p^2)$.
 \item For all $\beta \in \Phi^{\alpha}$, $\beta(\rho_2(\sigma)) \neq 1 \pmod {p^2}$.
\end{itemize}
\end{defn}
Note that for all $m \geq 2$, $\Lift^{\alpha}_{\br, 2}(\mc{O}/p^m)$ is $\wh{G}(\mc{O}/p^m)$-stable. The next result generalizes \cite[Proposition 24, Corollary 25]{ramakrishna-hamblen}
\begin{lemma}\label{extracocycles}
 For $m \geq 3$, let $\rho$ be an element of $\Lift^{\alpha}_{\br,2}(\mc{O}/p^m)$. 
Consider the subgroup of unramified cocycles $\mc{S}^{\alpha}_{\br} \subset H^1(\gal{F}, \br(\mf{g}))$ of the form
\[
 \mc{S}^{\alpha}_{\br}= \left\{ c \in H^1(\gal{F}, \br(\mf{g})): c(\sigma) \in \bigoplus_{\beta \in \Phi^{\alpha}} \mf{g}_{\beta}, c(\tau)=0\right\}.
\]
Then for all $c \in \mc{S}^{\alpha}_{\br}$, $\exp(c \otimes_k p^{m-1})\rho$ is $\widehat{G}(\mc{O}/p^m)$-conjugate to $\rho$, hence is an element of $\Lift^{\alpha}_{\br,2}(\mc{O}/p^m)$. In particular, the space $L^{\alpha}_{\br}= \Tan^{\alpha}_{\br}+ \mc{S}^{\alpha}_{\br}$ has dimension $\dim \mf{g}$ and preserves $\Lift^{\alpha}_{\br, 2}(\mc{O}/p^m)$ for $m \geq 3$.
\end{lemma}
\proof
We may assume $c(\sigma)$ lies in a particular root space $\fg_{\beta}$, for some $\beta \in \Phi^{\alpha}$. Write $u_{\beta} \colon \mathbb{G}_a \to G$ for an associated root homomorphism (over $\mc{O}$). It is convenient to fix an embedding $G \into \mr{GL}_N$ so as to be able to do ``matrix calculations,'' but we will leave this embedding implicit. Note that the embedding maps $\exp(c \otimes p^{m-1})$ to $(1+p^{m-1}c)$, so we will use these expressions interchangeably. First note that for any $A, B, X \in M_N(\mc{O})$, we have, by a straightforward calculation, the following identity in $\mr{GL}_N(\mc{O}/p^m)$ (for $m \geq 3$):
\begin{equation}\label{matrix}
 (1+p^{m-2}X)(1+pA+p^2B)(1-p^{m-2}X+p^{2m-4}X^2)= (1+p^{m-1}[X, A])(1+pA+p^2B).
\end{equation}
(Here $[\cdot, \cdot]$ is the Lie bracket, which of course we can compute either in $\fg$ or in $\mf{gl}_N$.) Observe that the term $1+p^{m-1}[X, A]$ is independent of $B$ and depends only on $X \mod p$, and that $p^{m-1}[X, A]$ depends linearly on $X \mod p$ (this last observation shows that if a set of cocycles preserves some class of lifts, then so does its $k$-span). We now compute $g \rho(\sigma) g^{-1}$, where $g= u_\beta(zp^{m-2})$ for some element $z \in \mc{O}^\times$. By Equation (\ref{matrix}),  the result depends only on $\rho_2(\sigma)$ and on $z \pmod p$; we can therefore replace $\rho(\sigma)$ by some lift of $\rho_2(\sigma)$ to $T(\mc{O})$, and then we find
\[
 g \rho(\sigma)g^{-1}\rho(\sigma)^{-1}=u_\beta\left(p^{m-2}z-p^{m-2}z\beta(\rho_2(\sigma))\right), 
\]
so that
\[
 g \rho(\sigma) g^{-1}= u_{\beta}\left(p^{m-1}z \frac{1-\beta(\rho_2(\sigma))}{p}\right) \rho(\sigma).
\]
Likewise, $g \rho(\tau) g^{-1}= \rho(\tau)$, so the cocycle $c$ translates $\rho$ to a $\wh{G}(\mc{O}/p^m)$-conjugate.

\endproof
We will use a number of times the following calculation of the local duality pairing at trivial primes:
\begin{lemma}\label{localduality}
Let $W$ be a trivial $k[\gal{F}]$-module. Then the local duality pairing
\[
\inv_F(\cdot \cup \cdot) \colon H^1(\gal{F}, W) \times H^1(\gal{F}, W^*) \to k
\]
has the following properties: if $\phi$ is unramified, then
\[
\inv_F(\phi \cup \psi)=- \langle \phi(\sigma), \psi(\tau) \rangle,
\]
and if $\psi$ is unramified, then
\[
\inv_F(\phi \cup \psi)= \langle \phi(\tau), \psi(\sigma) \rangle.
\]
\end{lemma}
\proof
First we recall the description of the $k$-linear version of the local duality pairing. Since $W^*= \Hom_k(W, \mu_p \otimes_{\Fp} k)$, cup-product and the invariant map together give a canonical pairing 
\[
H^1(\gal{F}, W) \times H^1(\gal{F}, W^*) \to H^2(\gal{F}, \mu_p \otimes_{\Fp} k)= H^2(\gal{F}, \mu_p)\otimes_{\Fp} k\xrightarrow[\sim]{\inv_F} \frac{1}{p} \Z \big/\Z \otimes_{\Fp} k =k.
\]
Since $W$ is trivial, the lemma reduces to the case $W=k$, and the above description of the $k$-linear duality pairing, which for $W=k$ is the $k$-linear extension of the $\Fp$-linear duality pairing on the trivial module $\Fp$, shows we can further reduce to the case $W=k= \Fp$.

Then the calculation can be performed, for instance, using the identity
\[
 \inv_F(\phi \cup \delta(a))= \phi(\rec_F(a))
\]
for any $\phi \in H^1(\gal{F}, W)= \Hom(\gal{F}^{\mr{ab}}, \Fp)$ and $a \in F^{\times}/(F^{\times})^p \xrightarrow[\sim]{\delta} H^1(\gal{F}, \mu_p)= H^1(\gal{F}, \Fp^*)$ (the last identification is the canonical one). If $\phi$ is unramified, then $\phi(\rec_F(a))$ is simply $-v(a)\phi(\sigma)$ (writing $v$ for the normalized valuation, and normalizing $\rec_F$ to take uniformizers to geometric frobenii). On the other hand, if $\psi= \delta(a)$, then 
\[
\psi(\tau)= \delta(a)(\tau)=\frac{\tau(a^{1/p})}{a^{1/p}}= \left(\frac{\tau(\varpi^{1/p})}{\varpi^{1/p}}\right)^{v(a)}= \zeta^{v(a)}.
\]
Then $\langle \phi(\sigma), \psi(\tau)\rangle= \zeta^{v(a)\phi(\sigma)}$, and via our isomorphism $\zeta \colon \Fp \to \mu_p$ we thus identify $\langle \phi(\sigma), \psi(\tau)\rangle= -\inv_F(\phi \cup \psi)$, as desired. Now suppose $\psi$ is unramified. Then we identify $W=W^{**}$ and apply the previous step to find 
\[
\inv_F(\phi \cup \psi)=-\inv_F(\psi \cup \phi)= \langle \psi(\sigma), \phi(\tau) \rangle= \langle \phi(\tau), \psi(\sigma) \rangle.
\]
\endproof
In particular:
\begin{cor}
 The orthogonal complement $(L^{\alpha}_{\br})^{\perp}$ of the space $L_{\br}^{\alpha}$ under the local duality pairing is the space of $\psi \colon \gal{F} \to \mf{g}^*$ such that $\langle \psi(\sigma), \fg_{\alpha} \rangle=0$ and $\langle \psi(\tau), \ker(\alpha|_{\mf{t}}) \oplus \oplus_{\beta \in \Phi(G, T)} \mf{g}_{\beta}\rangle=0$.
\end{cor}
\begin{proof}
 That such $\psi$ are contained in $(L^{\alpha}_{\br})^\perp$ is clear from Lemma \ref{localduality}. The result follows since $\dim (L^{\alpha}_{\br})^\perp= \dim \fg$.
\end{proof}

Next we give an analogue of \cite[Proposition 28, Corollary 29]{ramakrishna-hamblen}.
\begin{defn}\label{ramtrivlifts}
For any $m \geq 2$, let $\Lift^{\alpha}_{\br,2, \mr{ram}}(\mc{O}/p^m) \subset \Lift^{\alpha}_{\br}(\mc{O}/p^m)$ be the subset of lifts $\rho$ such that $\rho_2$ satisfies
\begin{itemize}
 \item $\rho_2(\tau)= u_{\alpha}(py)$ where $y \in \mc{O}^\times$, and where as before $u_{\alpha} \colon \mf{g}_{\alpha} \to G$ denotes the root group homomorphism over $\mc{O}$.
 \item $\rho_2(\sigma) \in T(\mc{O}/p^2)$, and for all $\beta \in \Phi^{\alpha}$, $\beta(\rho_2(\sigma)) \neq 1 \pmod {p^2}$.
\end{itemize}
\end{defn}
Again, $\Lift^{\alpha}_{\br, 2, \mr{ram}}(\mc{O}/p^m)$ is $\wh{G}(\mc{O}/p^m)$-stable.
\begin{lemma}\label{ramextracocycles}
 For $m \geq 3$, let $\rho$ be a lift in $\Lift^{\alpha}_{\br,2, \mr{ram}}(\mc{O}/p^m)$. Then there exists a subspace of cocycles $\mc{S}^{\alpha}_{\br} \subset H^1(\gal{F}, \br(\mf{g}))$ such that
\begin{itemize} 
 \item for all $c \in \mc{S}^{\alpha}_{\br}$, $\exp(c \otimes_k p^{m-1}) \rho$ is $\wh{G}(\mc{O}/p^m)$-conjugate to $\rho$; and
 \item $\mc{S}^{\alpha}_{\br}$ has a basis of cocycles $\{c_{\beta}\}_{\beta \in \Phi^{\alpha}}$ such that $c_{\beta}(\sigma)$ spans $\mf{g}_{\beta}$, and $c_{\beta}(\tau)$ is determined by $\rho_2$, $\alpha$, $\beta$, and $c_{\beta}(\sigma)$.
\end{itemize} 
In particular, $L^{\alpha}_{\br}= \Tan^{\alpha}_{\br}+ \mc{S}^{\alpha}_{\br}$ has dimension $\dim (\mf{g})$ and preserves $\Lift^{\alpha}_{\br, 2, \mr{ram}}(\mc{O}/p^m)$ for any $m \geq 3$.
\end{lemma}
\proof
Fix $\beta \in \Phi^{\alpha}$, and as in Lemma \ref{extracocycles} fix an embedding $G \into \mr{GL}_N$ that will remain implicit in the calculations. Consider basis vectors $X_{\beta}= d(u_{\beta})(1)$ of $\mf{g}_{\beta}(\mc{O})$ and $X_{\alpha}= d(u_{\alpha})(1)$ of $\fg_{\alpha}(\mc{O})$. Set $g= u_{\beta}\left(p^{m-2}\frac{p}{1-\beta(\rho_2(\sigma))}\right)$, so that just as in Lemma \ref{extracocycles} we find
\[
 g \rho(\sigma) g^{-1}=u_{\beta}(p^{m-1}) \rho(\sigma).
\]
We now perform the analogous calculation of $g \rho(\tau) g^{-1}$, again using the basic matrix Equation (\ref{matrix}) from Lemma \ref{extracocycles}. Equation (\ref{matrix}) implies that
\begin{align*}
 g \rho(\tau) g^{-1}&= \left(1+p^{m-1}\left[\frac{p}{1-\beta(\rho_2(\sigma))}X_{\beta}, y X_{\alpha}\right]\right)\rho(\tau)\\
&=\left(1+p^{m-1}\frac{py}{1-\beta(\rho_2(\sigma))}[X_\beta, X_{\alpha}]\right)\rho(\tau).
\end{align*}
Thus if we define $c_{\beta}$ by $c_{\beta}(\sigma)= X_{\beta}$ and $c_{\beta}(\tau)= \frac{py}{1-\beta(\rho_2(\sigma))}[X_\beta, X_\alpha]$, then we find that $(1+p^{m-1}c_{\beta})\rho$ is $\wh{G}(\mc{O}/p^m)$-conjugate to $\rho$.

\endproof
Finally, in the application we will use the fixed-multiplier variant of the above constructions. Fix a character $\nu \colon \gal{F} \to A(\mc{O})$ lifting $\bar{\nu}= \mu \circ \br= 1$. Let $\Lift^{\nu, \alpha}_{\br} \subset \Lift^\alpha_{\br}$ be the subfunctor of fixed multiplier lifts.
\begin{lemma}\label{trivfixeddet}
Under the hypotheses of Lemma \ref{trivsmooth}, $\Lift^{\nu, \alpha}_{\br}$ is formally smooth with tangent space $\Tan^{\alpha}_{\br} \cap H^1(\gal{F}, \br(\fg_{\mu}))$. Under the hypotheses of Lemmas \ref{extracocycles} (in which case we assume $\nu$ is unramified) and \ref{ramextracocycles}, and in the notation of those lemmas we obtain spaces $L^{\nu, \alpha}_{\br}$ of cocycles preserving $\Lift^{\nu, \alpha}_{\br, 2}(\mc{O}/p^m)$ (for $m \geq 3$) such that $L^{\nu, \alpha}_{\br}= L_{\br}^{\alpha} \cap H^1(\gal{F}, \br(\fgm))$ has dimension $\dim \fgm$.
\end{lemma}
\proof
Let $R \to R/I$ be a small extension, and let $\rho \in \Lift^{\nu, \alpha}_{\br}(R/I)$. By Lemma \ref{trivsmooth}, $\rho$ lifts to an element $\tilde{\rho} \in \Lift^{\alpha}_{\br}(R)$. The character $\nu \cdot (\mu \circ \tilde{\rho})^{-1}$ lands in $\ker(A(R) \to A(R/I))= \exp(\mf{a} \otimes_k I)$. Since $\mf{g} = \fgm \oplus \mf{g}^G \cong \fgm \oplus \mf{a}$ is an isomorphism, $\nu \cdot (\mu \circ \tilde{\rho})^{-1}$ canonically lifts to a character $\psi \colon \gal{F} \to Z_{G}^0(R)$ (whose reduction mod $I$ is trivial), and replacing $\tilde{\rho}$ by $\psi \cdot \tilde{\rho}$ we obtain a lift of $\rho$ to $\Lift^{\nu, \alpha}_{\br}(R)$. The rest of the lemma is similarly straightforward.
\endproof
\subsection{The case $\ell = p$}\label{ordsection}
In this subsection as well $\br$ will be valued in $G^0(k)$, so that all lifts automatically factor through $G^0$. To simplify the notation we will allow ourselves to refer abusively to a ``Borel subgroup of $G$" or a ``maximal torus of $G$" as the corresponding objects for $G^0$. Our goal is to describe two kinds of ordinary local conditions at primes $v \vert p$. The first will be essentially as in \cite[\S 4.1]{stp:exceptional}, while the second will be a variant modeled on \S \ref{trivialsection} that will also allow us to include the case $\br|_{\gal{F_v}}=1$. To handle both cases, we first introduce some notation. Fix a finite extension $F/\Q_p$. Fix a Borel subgroup $B \subset G$ with unipotent radical $N$, and let $T_G$ be the ``canonical maximal torus'' $B/N$. We write $\Phi^+$ and $\Phi^-$ for the positive and negative roots of $(G, T)$ with respect to $B$. For any residual representation $\br \colon \gal{F} \to G(k)$ with image contained in $B(k)$, let $\br_{T_G} \colon \gal{F} \to T_G(k)$ be its push-forward to $T_G(k)$, and let $\overline{\chi}_{T_G}$ be the inertial restriction $\br_{T_G}|_{I_F}$. Fix a lift $\chi_{T_G} \colon I_F \to T_G(\mc{O})$ that extends to $\gal{F}$ (we do not choose an extension), and likewise write $\chi_{T_G}$ for its image in $T_G(R)$ for any $\mc{O}$-algebra $R$.
\begin{defn}\label{ordinary}
 Let $\Lift_{\br}^{\chi_{T_G}} \colon \mc{C}^f_\mc{O} \to \Sets$ be the subfunctor of $\Lift_{\br}$ such that $\Lift_{\br}^{\chi_{T_G}}(R)$ consists of all lifts $\rho \colon \gal{F} \to G(R)$ of $\br$ such that
\begin{enumerate}
 \item there exists a $g \in \wh{G}(R)$ such that ${}^g \rho(\gal{F}) \subset B(R)$; and
 \item the composite $I_F \xrightarrow{{}^g \rho} B(R) \to T_G(R)$ is equal to $\chi_{T_G}$.
\end{enumerate}
If we moreover fix a similitude character $\nu \colon \gal{F} \to A(\mc{O})$ and require that $\chi_{T_G}$ pushes forward to $\nu$ under the natural map $T_G \to A$ (factoring the restriction of $\mu$ to $B$), then we can similarly define the subfunctor of lifts with fixed multiplier character $\Lift^{\nu, \chi_{T_G}}_{\br}$.
\end{defn}
Recall from \cite[\S 2]{tilouine:defs} and \cite[\S 4.1]{stp:exceptional} the following two conditions on $\br$:
\begin{align*}
 &\text{(REG) $H^0(\gal{F}, \br(\fg/\fb))=0.$}\\ 
 &\text{(REG*) $H^0(\gal{F}, \br(\fg/\fb)(1))=0.$}
\end{align*}
We will either consider $\br$ that satisfy both (REG) and (REG*) or the trivial representation $\br$. In the first case we recall:
\begin{lemma}\label{regord}
 Assume $\br$ satisfies (REG) and (REG*), and that $F$ does not contain $\zeta_p$. Then $\Lift_{\br}^{\chi_{T_G}}$ is formally smooth and pro-represented by a power series ring over $\mc{O}$ in $\dim_k(\fg)+\dim_k(\mf{n})[F:\Q_p]$ variables. Moreover, if $\chi_{T_G}$ is chosen so that for all simple roots $\alpha$, $\alpha \circ \chi_{T_G}= \kappa^{r_{\alpha}}$ for a \emph{positive} integer $r_{\alpha}$, then any element of $\Lift^{\chi_{T_G}}_{\br}$ is de Rham. Similarly, $\Lift^{\nu, \chi_{T_G}}_{\br}$ is formally smooth with tangent space $L^{\nu, \chi_{T_G}}_{\br}$ of dimension $h^0(\gal{F}, \br(\fgm))+\dim_k \mf{n}[F:\Q_p]$.
\end{lemma}
\begin{proof}
 This is proven in \cite[Lemma 4.2, Proposition 4.4, Corollary 4.5, Lemma 4.8]{stp:exceptional}.
\end{proof}

Next we consider the case in which $\br$ is the trivial representation; in particular, the condition (REG) is not satisfied. Consider an ordinary lifting condition $\Lift_{\br}^{\chi_{T_G}}$ as in Definition \ref{ordinary}. We do not claim the functor $\Lift_{\br}^{\chi_{T_G}}$ is pro-representable, but it is still formally smooth:
\begin{lemma}\label{ptrivsmooth}
 Assume that $\br$ is trivial, and that $F$ does not contain $\zeta_p$. Then $\Lift_{\br}^{\chi_{T_G}}$ and $\Lift^{\nu, \chi_{T_G}}_{\br}$ are formally smooth.
\end{lemma}
\proof
This follows from the argument of \cite[Proposition 4.4]{stp:exceptional}: to see this, note that (REG*) holds, and that the natural map $H^1(\gal{F}, \br(\fb)) \to H^1(I_F, \br(\fb/\mf{n}))^{\gal{F}/I_F}$ is surjective (in \cite{stp:exceptional}, the latter fact was deduced from a combination of conditions (REG) and (REG*), but here it is clear). 
\endproof
In general, the tangent space of $\Lift_{\br}^{\chi_{T_G}}$ is
\[
 \mr{Tan}_{\br}^{\chi_{T_G}}= \im \left( \ker(H^1(\gal{F}, \br(\mf{b})) \to H^1(I_F, \br(\mf{b}/\mf{n})) ) \to H^1(\gal{F}, \br(\mf{g}))\right),
\]
which for trivial $\br$ simplifies to
\[
 \mr{Tan}_{\br}^{\chi_{T_G}}= \Hom(\gal{F}, \mf{n}) \oplus \Hom(\gal{F}/I_F, \br(\mf{b}/\mf{n})).
\]
Thus $\mr{Tan}_{\br}^{\chi_{T_G}}$ has dimension $\dim \mf{n}+[F:\Q_p]\dim \mf{n}+ \dim \mf{b}/\mf{n}= \dim \mf{b}+ [F:\Q_p]\dim \mf{n}$. We need a space of cocycles $L^{\chi_{T_G}}_{\br}$ preserving a certain subset (lifts with favorable mod $p^2$ reduction)  of $\Lift^{\chi_{T_G}}_{\br}(\mc{O}/p^m)$ for $m \geq 3$ and having $\dim L^{\chi_{T_G}}_{\br}= 
h^0(\gal{F}, \br(\mf{g}))+[F:\Q_p]\dim \mf{n}= \dim \mf{g}+ [F:\Q_p] \dim \mf{n}$. Thus in this case we need $\dim \mf{n}$ independent extra cocycles. 
\begin{defn}\label{trivord}
Fix a maximal torus $T$ of $B$, i.e. a lift of $T_G$ to $B$. Via the canonical isomorphism $T \xrightarrow{\sim} T_G$, identify $\chi_{T_G}$ to a character $\chi_T \colon I_F \to T(\mc{O})$. Let $\Lift_{\br, 2}^{\chi_T}$ be the subset of $\Lift_{\br}^{\chi_{T_G}}$ of lifts whose mod $p^2$ reductions factor through $T$, and thus are identified with $\chi_T \pmod {p^2}$. Note that $\Lift_{\br, 2}^{\chi_T}$ is $\wh{G}$-stable: indeed, since $\br=1$, $\wh{G}$-conjugation preserves the mod $p^2$ reduction of any lift. 
\end{defn}
\begin{lemma}\label{ordextracocycles}
 Continue to assume $\br=1$. Choose $\chi_T$ such that $\beta \circ \chi_T$ is non-trivial modulo $p^2$ for all $\beta \in \Phi^-$. For $m \geq 3$, let $\rho$ be an element of $\Lift^{\chi_T}_{\br, 2}(\mc{O}/p^m)$, with reduction $\rho_2 \colon \gal{F} \to T(\mc{O}/p^2)$. Consider for each $\beta \in \Phi^-$, the cocycle 
\[
c_{\beta}(\sigma)= \frac{1- \beta(\chi_T(\sigma))}{p} X_\beta,
\]
with $X_\beta= (du_\beta)(1)$ a $\mc{O}$-basis of $\mf{g}_\beta$. Then for all $\beta \in \Phi^-$, $(1+p^{m-1}c_{\beta})\rho$ is $\wh{G}(\mc{O}/p^m)$-conjugate to, and thus is an element of, $\Lift^{\chi_T}_{\br, 2}(\mc{O}/p^m)$. In particular, the space
\[
 L^{\chi_T}_{\br}= \Tan^{\chi_T}_{\br}+ \sum_{\beta \in \Phi^-} k \cdot c_\beta
\]
has dimension $\dim \mf{g}+ [F:\Q_p] \dim \mf{n}$ and preserves the set $\Lift^{\chi_T}_{\br, 2}(\mc{O}/p^m)$. 

The corresponding statements hold for $\Lift^{\nu, \chi_T}_{\br, 2}$ and the $(\dim(\fgm)+[F:\Q_p]\dim(\mf{n}))$-dimensional space of stabilizing cocycles $L^{\nu, \chi_T}_{\br}$. 
\end{lemma}
\proof
First note that each $c_{\beta}$ is in fact a cocycle (homomorphism): for all $\sigma, \tau \in \gal{F}$,
\[
 c_{\beta}(\sigma \tau)-c_{\beta}(\sigma)-c_{\beta}(\tau)= \frac{(1-\beta(\chi_T(\sigma)))(\beta(\chi_T(\tau))-1)}{p}X_\beta =0 \pmod p.
\]
Let $g= u_\beta(p^{m-2}) \in \wh{G}(\mc{O}/p^m)$. Again calculating with Equation (\ref{matrix}), we find
\begin{align*}
 g \rho(\sigma) g^{-1}&= \left(1+p^{m-2}(X_{\beta}-\Ad(\rho_2(\sigma))X_{\beta}) \right) \rho(\sigma)\\
&= \left(1+p^{m-1} \frac{1-\beta(\rho_2(\sigma))}{p} X_{\beta} \right) \rho(\sigma).
\end{align*}
By our assumptions on $\chi_T$, all the cocycles $c_{\beta}$, $\beta \in \Phi^-$, are non-trivial; they are clearly linearly independent from one another and from $\Tan^{\chi_T}_{\br}$, so the result follows.
\endproof
We also want to know that a compatible collection $(\rho_m)_{m \geq 2}$, with $\rho_m \in \Lift^{\chi_T}_{\br, 2}(\mc{O}/p^m)$ gives rise to a de Rham $p$-adic limit $\rho= \varprojlim \rho_m$ when $\chi_T$ is chosen appropriately:
\begin{lemma}\label{trivordlim}
 Let $\rho \colon \gal{F} \to G(\mc{O})$ be a continuous homomorphism lifting $\br$ such that for all $m \geq 2$, the reduction $\rho_m= \rho \pmod {p^m}$ is an element of $\Lift_{\br}^{\chi_{T_G}}(\mc{O}/p^m)$. Then there exists $g \in \wh{G}(\mc{O})$ such that ${}^g \rho$ is valued in $B(\mc{O})$, and such that the composite $I_F \xrightarrow{{}^g \rho} B(\mc{O}) \to T_G(\mc{O})$ is equal to $\chi_{T_G}$. In particular, if for all simple roots $\alpha$, $\alpha \circ \chi_{T_G}= \kappa^{r_{\alpha}}$ for some positive integer $r_{\alpha}$, then $\rho$ is de Rham.
\end{lemma}
\begin{proof}
 The final claim follows as in Lemma \ref{regord}. To construct the element $g$, consider the inverse system of sets
\begin{align*}
 U_m= \{ &\text{$g \in \wh{G}(\mc{O}/p^m)$ such that $g \rho_m g^{-1}(\gal{F}) \subset B(\mc{O}/p^m)$ and}\\ 
 &\text{$I_F \xrightarrow{{}^g \rho_m} B(\mc{O}/p^m) \to T_G(\mc{O}/p^m)$ equals $\chi_{T_G}$} \}.
\end{align*}
Since $\rho_m$ reduces to $\rho_{n}$ for all $m \geq n$, reduction mod $p^{n}$ induces compatible maps $U_m \to U_{n}$. By assumption (the very definition of $\Lift^{\chi_{T_G}}_{\br}$), each $U_m$ is non-empty, and each is clearly finite. It follows that $\varprojlim U_m$ is non-empty. Let $g_{\bullet}= (g_m)_{m \geq 1}$ be an element of $\varprojlim U_m$, and take $g$ to be the element $g= \varprojlim g_m \in \wh{G}(\mc{O})$.
\end{proof}

Finally, in the application we will require that our global residual representations $\br \colon \gal{F} \to G(k)$, with $F$ now a number field, have the property that $K= F(\br(\fg), \mu_p)$ does not contain $\mu_{p^2}$. In practice we will deduce this condition from the analogous local statement, using our local hypotheses on $\br|_{\gal{F_v}}$ for $v \vert p$:
\begin{lemma}\label{p2}
 Let $F/\Q_p$ be a finite extension that is unramified at $p$. Assume that $\br$ is ordinary in the sense of Definition \ref{ordinary} and moreover for all $\alpha \in \Phi^+$, $\alpha \circ \overline{\chi}_T$ is a non-trivial power of $\bar{\kappa}$. Then $\mu_{p^2}$ is not contained in $K= F(\br(\fg), \mu_p)$. Of course, if we instead assume that $\br$ is trivial, then the same conclusion holds.
\end{lemma}
\proof
We will show that under the assumptions on $\br$, the fixed field of the kernel of $\br$ cannot contain the unique $p$-extension $F(\varepsilon)/F$ inside $F(\mu_{p^2})$. Let $L= F(\mu_p)$; it is the fixed field of $\overline{\chi}_T$, and we set $\Delta= \Gal(L/F) \cong (\Z/p)^\times$. Let $L_{\br}$ be the the fixed field of $\br|_{I_{F}}$, and filter $G_0= \Gal(L_{\br}/F)$ by the subgroups $G_i= \{g \in G_0: \br(g) \in N_{\geq i}(k)\}$, where $N_{\geq i}$ is the closed subgroup of $B$ whose Lie algebra is spanned by root spaces of height at least $i$. Since $T$ acts on each $N_{\geq i}$, and on each quotient $N_{\geq i}/N_{\geq i+1}$, $\Delta$ acts, necessarily semi-simply, on each of the $\Fp$-vector spaces $G_i/G_{i+1}$. Our assumption that $\alpha \circ \overline{\chi}_T$ is non-trivial for all $\alpha$ implies that $\Delta$ in fact acts by a direct sum of non-trivial powers of $\bar{\kappa}$ on each graded piece $G_i/G_{i+1}$. Now, we are assuming that $F(\varepsilon)$ is contained in $L_{\br}$, and therefore corresponds to a subgroup $H \subset G_0$. Let $i$ be maximal such that $H$ contains $G_i$. We then get a well-defined, non-trivial, $\Delta$-equivariant map $G_{i-1}/G_i \to G/H$, necessarily surjective since $G/H \cong \Z/p$. But $F(\mu_{p^2})/F$ is abelian, so $\Delta$ acts trivially on $G/H$, whereas it acts by a sum of non-trivial characters on $G_{i-1}/G_i$. Contradiction.
\endproof
When $G^0= \mr{GL}_N$ (or a product of copies of general linear groups), we can also impose the condition that $\br$ lie in the image of the Fontaine-Laffaille functor, i.e. be torsion-crystalline with Hodge-Tate weights in an interval of length less than $p-1$. In this case ramification bounds for torsion-crystalline representations due to Abrashkin and Fontaine (\cite[\S 2, Assertion 8.1]{abrashkin:modular}, \cite[Th\'eor\`eme 2]{fontaine:overZ}; a convenient reference is \cite{hattori:survey}) imply that $\mu_{p^2}$ is not contained in $F(\br(\fg), \mu_p)$. For Fontaine-Laffaille theory, we follow the notation of \cite[\S 2]{hattori:survey}.
\begin{lemma}\label{FLp2}
Assume $F/\Q_p$ is unramified. Let $L$ be a $k[\gal{F}]$-module isomorphic to $T_{\mr{cris}}(M)$ for some Fontaine-Laffaille module $M$ with Hodge-Tate weights in an interval of length less than $p-1$. Then $\mu_{p^2}$ is not contained in the fixed field $F(L, \mu_p)$ of the set of $\sigma \in \gal{F}$ acting trivially on both $L$ and $\mu_p$.
\end{lemma}
\begin{proof}
This follows immediately from \cite[Theorem 2.19]{hattori:survey} (due to Abrashkin and Fontaine independently) and the familiar calculation of the upper ramification filtration of cyclotomic extensions.
\end{proof}

\section{Modifying the mod $p^2$ lift}\label{klrsection}
Let $F$ be a number field, let $S$ be a finite set of places of $F$ containing all places above $p$, and let $\br \colon \gal{F, S} \to G(k)$ be a continuous homomorphism. In this section we explain how to construct a mod $p^2$ lift of $\br$ that will in \S \ref{finallift} be the starting point for applying the arguments of \S \ref{auxiliarysection}. We may assume $\br$ surjects onto $\pi_0(G)$ (if not, we replace $G$ by the preimage in $G$ of the image of $\br$ in $\pi_0(G)$; the deformation theory of $\br$ is unchanged by this replacement). There is then a unique finite Galois extension $\tF/F$ such that $\br$ induces an isomorphism $\Gal(\tF/F) \to \pi_0(G)$. We make the following assumptions on $\br$:
\begin{assumption}\label{multfree}
Assume $p \gg_G 0$, and let $\br \colon \gal{F, S} \to G(k)$ be a continuous representation unramified outside a finite set of finite places $S$; we may and do assume that $S$ contains all places above $p$. Assume that $\br$ satisfies the following:
\begin{itemize} 
 \item The field $K= \tF(\br(\fgder), \mu_p)$ does not contain $\mu_{p^2}$.
 \item $H^1(\Gal(K/F), \br(\fgder))=0$.
 \item $H^1(\Gal(K/F), \br(\fgder)^*)$=0.
 \item $\br(\fgder)$ and $\br(\fgder)^*$ are semisimple $\mathbb{F}_p[\gal{F}]$-modules (equivalently, semisimple $k[\gal{F}]$-modules) having no common $\Fp[\gal{F}]$-sub-quotient, and neither contains the trivial representation.
\end{itemize}
How large $p$ must be given (the root datum of) $G$ can be extracted from the arguments of this section, but we do not make it explicit.
\end{assumption}
\begin{rmk}\label{infinitetriv}
We could carry out the analysis of this section without the assumption that $K$ does not contain $\mu_{p^2}$; the difference is that the sets of ``trivial" primes $w$ that we produce would not necessary satisfy $N(w) \not \equiv 1 \pmod {p^2}$. In particular, Corollary \ref{infiniteram} does not require this assumption. In the next section, when we use the results of this section as input for the \cite{ramakrishna-hamblen} approach to annihilating dual Selmer groups, we will need the additional hypothesis, and we will need the auxiliary primes produced in this section to be ``trivial" in the sense used in \S \ref{trivialsection}.
\end{rmk}
We decompose 
\[
\br(\fgder)= \oplus_{i \in I} W_i^{\oplus m_i}
\] 
where each $W_i$ is an irreducible $\Fp[\gal{F}]$-module, and $W_i \not \cong W_j$ for $i \neq j$. Dually we obtain the decomposition $\br(\fgder)^*= \oplus (W_i^*)^{m_i}$, where $W^*= \Hom_{\Fp}(W, \Fp)(1)$ is the $\Fp$-dual. Each $W_i$ is a $k_{W_i}= \End_{\Fp[\gal{F}]}(W_i)$-module, and since $\mr{Br}(\Fp)=0$ $k_{W_i}$ is a finite field extension of $\Fp$. We may then also regard $W_i^*$ as the $k_{W_i}$-dual, with the trace identifying the $k_{W_i}$-vector spaces
\[
\tr_{k_{W_i}/\Fp} \colon \Hom_{k_{W_i}}(W_i, k_{W_i}) \xrightarrow{\sim} \Hom_{\Fp}(W_i, \Fp).
\]  

We begin by finding some mod $p^2$ lift of $\br$. First note that since $G/G^{\mr{der}}(k)$ has order prime to $p$ (being an extension of the prime-to-$p$ component group by the $k$-points of a torus), the reduction map $G/G^{\mr{der}}(\mc{O}/p^2) \to G/G^{\mr{der}}(k)$ has a (group-theoretic) section $s$. In particular $(\br \mod G^{\mr{der}})$ has a lift $s(\br \mod G^{\mr{der}})$ to $G/G^{\mr{der}}(\mc{O}/p^2)$. The argument of Lemma \ref{trivfixeddet} shows that we can replace it with a lift having multiplier character $\nu$. In particular, the obstruction $\mr{obs}_{\br} \in H^2(\gal{F, S}, \br(\mf{g}_{\mu}))$ to lifting $\br$ to $G(\mc{O}/p^2)$ with multiplier $\nu$ has trivial $\mf{z}_{\mu}$-component under the decomposition $\fgm= \mf{z}_{\mu} \oplus \fg^{\mr{der}}$. Next let $T \supset S$ be a finite set of places with $T \setminus S$ consisting of trivial primes such that $\Sha^1_T(\gal{F, T}, \br(\mf{g}^{\mr{der}}))=0$ and $\Sha^1_T(\gal{F, T}, \br(\mf{g}^{\mr{der}})^*)=0$. That $T$ can be so arranged follows from the first three items of Assumption \ref{multfree}: the cocycles in question restrict non-trivially to $\gal{K}$, and then we choose places $v$ that are split in $K$ and non-split in both $K(\mu_{p^2})$ and the fixed field (over $K$) of the cocycle (the latter two conditions are compatible whether or not the fixed field is disjoint from $K(\mu_{p^2})$, since they are both just the condition of being non-trivial). By global duality, $\Sha^2_T(\gal{F, T}, \br(\mf{g}^{\mr{der}}))$ also vanishes, so to produce some lift $\rho_2 \colon \gal{F, T} \to G(\mc{O}/p^2)$ of $\br$ with $\mu \circ \rho_2= \nu$ it suffices to check there are no local lifting obstructions: 
\begin{lemma}\label{localmodp^2}
Assume $p \gg_G 0$. Then for all finite places $v$, the set of mod $p^2$ local lifts $\Lift_{\br|_{\gal{F_v}}}(\mc{O}/p^2)$ is non-empty (and similarly for fixed-similitude character lifts).
\end{lemma}
\proof
For $p \gg_G 0$, there exists a faithful representation $r \colon G \into \mr{GL}_N$ such that $\mf{g}$ is a direct summand of $\mf{gl}_n$. The induced map $H^2(\gal{F_v}, \br(\fg)) \to H^2(\gal{F_v}, r\circ\br(\mf{gl}_n))$ is thus injective, and it clearly sends the obstruction to lifting $\br$ (to $G(\mc{O}/p^2)$) to the obstruction to lifting $r \circ \br$. But the latter is unobstructed by \cite[Theorem 1.1]{bockle:modp2}). The fixed multiplier character analogue follows similarly.
\endproof
In fact, in the application we will make a stronger assumption on $\br_{\gal{F_v}}$ for $v \in S$, obviating the need for this lemma; for now we are trying to proceed without superfluous hypotheses.

In what follows, it will be technically convenient to enlarge the set $T$ by trivial primes, beyond what is necessary to annihilate $\Sha^1_T(\gal{F, T}, \br(\fgder))$ (and the dual version). We may and do assume that our $T$ strictly contains whatever initial choice of $T$ was used to annihilate the Shafarevich-Tate groups; more precise enlargements of this set $T$ will follow. We can compute such an enlargement's effect on global Galois cohomology. More generally, if $W$ is any $\Fp[\gal{F, T}]$-module, we have the analogous notion of a trivial prime $v$ for $W$: $W|_{\gal{F_v}}$ is trivial, and $N(v) \equiv 1 \pmod p$ but $N(v) \not \equiv 1 \pmod {p^2}$. If moreover $W$ satisfies $\Sha^1_T(\gal{F, T}, W)=0$, then for any trivial prime $v \not \in T$ we have an exact sequence
\[
0 \to H^1(\gal{F, T}, W) \to H^1(\gal{F, T \cup v}, W) \to H^1(\gal{F_v}, W)/H^1_{\mr{unr}}(\gal{F_v}, W) \to 0,
\]
where the second map is given by evaluation at $\tau_v$; surjectivity of this map follows from the Greenberg-Wiles Euler-characteristic formula (\cite[Theorem 2.19]{DDT2}). In particular, the cokernel of the inflation map has dimension $\dim W$.

We must modify our initial $\rho_2$ so that its local behavior allows further lifting. We will now fix certain local lifts to $G(\mc{O}/p^2)$ that we would like to interpolate into a global mod $p^2$ representation. In \S \ref{finallift}, we will be more particular about what lifts we choose, and we will make additional assumption on the local behavior of $\br$; so as to be clear about what assumptions are used at what point in the paper, we delay imposing these additional hypotheses. Thus, for the present section, we require the following:
\begin{itemize}
\item For $w \in S$, fix any lift $\lambda_w \in \Lift_{\br|_{\gal{F_w}}}^{\nu}(\mc{O}/p^2)$. 
\item For $w \in T\setminus S$, fix a lift $\lambda_w$ in the set $\Lift^{\nu, \alpha}_{\br|_w, 2}(\mc{O}/p^2)$ of Definition \ref{unrtrivlifts} (and see Lemma \ref{trivfixeddet}), for some pair $(T_w, \alpha_w)$ of a maximal torus and root (we will frequently omit the $w$-dependence from the notation). Moreover, we if necessary enlarge the set $T$ and choose the lifts $\lambda_w$ so that the elements $\lambda_w(\sigma_w)$ generate $\wh{G}^{\mr{der}}(\mc{O}/p^2)$. That this is possible relies on the hypothesis $p \gg_G 0$ and is explained in Lemma \ref{bigimage} 
below. 
\end{itemize}
\begin{lemma}\label{bigimage}
As in our running Assumption \ref{multfree}, we assume $p \gg_G 0$. For any trivial prime $w \not \in S$, there are lifts $\lambda_w \in \Lift^{\alpha_w}_{\br|_{\gal{F_w}}, 2}(\mc{O}/p^2)$  (i.e., satisfying the conditions in Definition \ref{unrtrivlifts}) for any choice of pair $(T_w, \alpha_w)$ of a split maximal torus and root. Moreover, we can choose the set $T$, with $T \setminus S$ a finite set of trivial primes, and the lifts $(\lambda_w)_{w \in T \setminus S}$, such that the elements $(\lambda_w(\sigma_w))_{w \in T \setminus S}$ generate $\wh{G}^{\mr{der}}(\mc{O}/p^2)$.
\end{lemma}
\proof
Fix a trivial prime $w \not \in S$, and let $q= N(w)$. To construct some lift $\lambda_w$, fix any pair $(T, \alpha)$. Let $q^{1/2}$ denote the square-root of $q$ in $\mc{O}/p^2$ that is congruent to 1 modulo $p$. Consider elements of $\wh{T}^{\mr{der}}(\mc{O}/p^2)$ of the form $t_b= (1+pb)\alpha^\vee(q^{1/2})$ for $b \in \ker(\alpha)$. Any such $t_b$ satisfies $\alpha(t_b)=q$ and reduces to $1 \in G(k)$; we will find $b$ such that $\beta(t_b) \neq 1$ for all $\beta \in \Phi^{\alpha}$. If there is no $b_{\beta} \in \ker(\alpha)$ such that $\beta(b_{\beta})= -\frac{q-1}{2p}\langle \beta, \alpha^\vee \rangle$, then the condition $\beta(t_b) \neq 1$ will be satisfied automatically for any $b \in \ker(\alpha)$, so we now restrict to the subset $\Phi^{\alpha, *}$ of $\Phi^{\alpha}$ for which such $b_{\beta}$ do exist (and we fix one such $b_{\beta}$ for each $\beta$). Then we choose $b$ in the complement of the union of hyperplanes
\[
\bigcup_{\beta \in \Phi^{\alpha, *} \setminus \{-\alpha\}} \left(b_{\beta}+ \ker(\beta|_{\ker(\alpha)}) \right)
\]
inside $\ker(\alpha)$ (in the case where $\fg^{\mr{der}}$ is a product of simple factors, note that $\beta \in \Phi^\alpha$ implies $\beta$ is in the same simple factor as $\alpha$, and so $\ker(\beta|_{\ker(\alpha)})$ is indeed a hyperplane in $\ker(\alpha)$, since we take $\beta \neq -\alpha$). The total number of such hyperplanes is bounded in a way depending only on the Dynkin type of $G^{\mr{der}}$, so for $p \gg_G 0$, this complement is non-empty. Clearly for such a $b$, using avoidance of the hyperplanes $b_{\beta}+ \ker(\beta|_{\ker(\alpha)})$, we have $\beta(t_b) \neq 1$ for all $\beta \in \Phi^{\alpha}$.

For the second claim, it suffices to show that the union
\[
\bigcup_{T} \bigcup_{\alpha \in \Phi(G, T)} \{x \in \mf{t}: \alpha(x)= \frac{q-1}{p}\} \cap \left(\mf{t} \setminus \cup_{\beta \in \Phi^{\alpha}} \ker(\beta) \right)
\]
spans $\mf{g}^{\mr{der}}$ as $\mathbb{F}_p$-vector space, where the first union is taken over all $k$-rational split maximal tori of $G$. We easily reduce to the case where $\fgder$ is simple. It then suffices to show that for some pair $(T, \alpha)$, the intersection $\{x \in \mf{t}: \alpha(x)= \frac{q-1}{p}\} \cap \left(\mf{t} \setminus \cup_{\beta \in \Phi^{\alpha}} \ker(\beta) \right)$ is non-empty: given such an $x$, we get a corresponding element $\Ad(g)x$ (associated to the pair $(\Ad(g)\mf{t}, \Ad(g)\alpha)$) for each $g \in G(k)$, and then the result follows from irreducibility of $\wh{G}(\mc{O}/p^2) \cong \mf{g}^{\mr{der}}$ as an $\mathbb{F}_p[G(k)]$-module. This irreducibility claim follows from the irreducibility of $\fgder$ as $k[G(k)]$-module (a consequence of $p>3$ being very good) and the fact that for any $\sigma \in \Gal(k/\Fp)$, $\sigma(\fgder)$ is not isomorphic to $\fgder$ as $k[G(k)]$-module for $p>3$ (see \cite[Lemma 3.8]{stp:exceptional2}). 
\endproof

Having fixed the local mod $p^2$ lifts $\lambda$ as in the discussion preceding Lemma \ref{bigimage}, we have that for each $w \in T$ there is a class $z_w \in H^1(\gal{F_w}, \br(\mf{g}_{\mu}))$ such that 
\begin{equation}\label{lambdas}
(1+pz_w)\rho_2|_w \sim \lambda_w
\end{equation}
(with $\sim$ denoting strict equivalence). We wish to modify $\rho_2$ by a global cohomology class so that the resulting lift of $\br$ matches the specified local lifts $\lambda_w$. In fact, we will only need a somewhat weaker result that globally interpolates the $\br(\fg^{\mr{der}})$-components $z_w^{\mr{der}}$ of the classes $z_w$. Thus for the remainder of this section we focus on this problem; we explain the reduction to this setting in the proof of Theorem \ref{mainthm}.

If there exists a global class $h \in H^1(\gal{F, T}, \br(\fgder))$ mapping to $z_T^{\mr{der}}= (z_w^{\mr{der}})_{w \in T}$ under the localization map
\[
\Psi_T \colon H^1(\gal{F, T}, \br(\fgder)) \to \oplus_{w \in T} H^1(\gal{F_w}, \br(\fgder)),
\]
then we proceed to \S \ref{auxiliarysection}. For the remainder of this section, we assume there is no such $h$. Denote by $\Psi_T^*$ the corresponding localization map for $\br(\fgder)^*$.
To construct auxiliary primes, we will need the following lemma:
\begin{lemma}\label{lindisjoint}
Let $W$ be an irreducible $\mathbb{F}_p[\gal{F, T}]$-module such that $H^1(\Gal(F(W)/F), W)=0$.
Set $k_W= \End_{\mathbb{F}_p[\gal{F, T}]}(W)$ (a finite extension of $\mathbb{F}_p$, since $\mr{Br}(\Fp)=0$), and set $K= F(W, \mu_p)$. Let $\psi_1, \ldots, \psi_s$ be a $k_W$-basis of $H^1(\gal{F, T}, W)$. Then the fixed field $K_{\psi_1}, \ldots, K_{\psi_s}$ of the cocycles $\psi_i$ are strongly linearly disjoint over $K$, and for each $i$, $\Gal(K_{\psi_i}/K) \xrightarrow[\psi_i]{\sim} W$. If moreover $\mu_{p^2}$ is not contained in $K$, and $W$ is not isomorphic to the trivial representation, then for any $w \in W$ and any non-zero class $\psi \in H^1(\gal{F, T}, W)$, there exists a \v{C}ebotarev set of trivial primes $v$ such that $\psi(\sigma_v)=w$.
\end{lemma}
\begin{proof}
We must show that restriction gives an isomorphism
\[
 \Gal(K_{\psi_1}\cdots K_{\psi_{s}}/K) \xrightarrow{\sim} \prod_{i=1}^{s} \Gal(K_{\psi_i}/K) \xrightarrow{\sim} \prod_{i=1}^{s} W.
\]
To see this, we induct on the number of factors. For $s=1$, the isomorphism follows from simplicity of the $\mathbb{F}_p[\gal{F}]$-module $W$ (note that $\psi_i|_{\gal{K}} \neq 0$). If the linear disjointness is known for $\psi_1, \ldots, \psi_i$, and if $K_{\psi_{i+1}}$ is contained in the composite $K_{\psi_1}\cdots K_{\psi_i}$, then we have a map of $\mathbb{F}_p[\gal{F, T}]$-modules
\[
 W^{\oplus i} \xleftarrow[\psi_1, \ldots, \psi_i]{\sim} \Gal(K_{\psi_1}\cdots K_{\psi_i}/K) \onto \Gal(K_{\psi_{i+1}}/K) \xrightarrow[\psi_{i+1}]{\sim} W.
\]
Since $W$ is irreducible, the composite $W^{\oplus i} \to W$ has the form $(a_1, \ldots, a_i)$ for some $a_i \in k_W$, and we deduce that $\psi_{i+1}= \sum_{j=1}^i a_j \psi_j$, contradicting linear independence. We conclude that $K_{\psi_{i+1}}$ is not contained in $K_{\psi_1} \cdots K_{\psi_i}$, but again since $W$ is irreducible this forces these fields to be linearly disjoint over $K$.

The last claim is clear if $K_{\psi}$ and $K(\mu_{p^2})$ are linearly disjoint over $K$. Otherwise, $K(\mu_{p^2})$ is contained in $K_{\psi}$, and so $W$ has the $\Fp[\gal{F}]$-quotient $\Gal(K(\mu_{p^2})/K)$ with trivial $\gal{F}$-action; by assumption, this quotient is non-zero, so that $W$ itself must be the trivial representation. 

%
\end{proof}

We next explain in Proposition \ref{klrstep1''} how to interpolate the class $z_T^{\mr{der}}$ by a global class after allowing ramification at a finite number of additional primes. An important technical point in the proof of Proposition \ref{doublingprop} requires that we impose an additional (at this point rather unmotivated) condition on our trivial primes. Let $K'$ be the composite of all abelian $p$-extensions $L$ of $K$ that are Galois over $F$ and that satisfy
\begin{itemize}
\item $L/F$ is unramified outside $T$;
\item and the $\Fp[\Gal(K/F)]$-module $\Gal(L/K)$ is isomorphic to one of the $W_i$.
\end{itemize}
Because the extensions $L/F$ are unramified outside $T$ with absolutely bounded degree, $K'$ is a finite extension of $F$. Primes split in $K'$ are of course also split in $K$, and $K'$ and $K(\mu_{p^{2}})$ are linearly disjoint over $K$ since no $W_i$ is the trivial representation. In what follows, we will refer to trivial primes split in $K'$ as $K'$-trivial primes.
\begin{prop}\label{klrstep1''} Continue to assume that  $\Sha^1_T(\gal{F, T},\br(\fgder))=0$ and $\Sha^1_T(\gal{F, T},\br(\fgder)^*)=0$, and hence by duality that $\Sha^2_T(\gal{F, T}, \br(\fgder))=0$.
Then there is an $\Fp$-basis $\{Y_i\}_{i=1}^r$ of the cokernel of the restriction map $\Psi_T \colon H^1(\gal{F, T}, \br(\fgder)) \rightarrow \oplus_{v \in T}  H^1(\gal{F_v}, \br(\fgder))$, and, for each $i$, a \v{C}ebotarev set $C_i$ of $K'$-trivial primes $v \not \in T$, a maximal torus $T_i$ and root $\alpha_i \in \Phi(G^0, T_i)$, and for each $v \in C_i$ a class $h^{(v)} \in H^1(\gal{F, T\cup v}, \br(\fgder))$ such that
\begin{itemize}
\item $h^{(v)}|_T= Y_i$; and
\item $h^{(v)}(\tau_v)$ spans $\mf{g}_{\alpha_i}$.
\end{itemize}
\end{prop}
\begin{rmk}\label{multfreermk}
 One can ask whether it is possible to hit the class $z_T^{\mr{der}}$ by allowing only one additional prime of ramification; this is how the analogous argument in \cite{ramakrishna-hamblen} (for reducible two-dimensional $\br$) works. We have only been able to show such a statement when $\br(\fgder)$ is multiplicity-free as an $\Fp[\gal{F}]$-module, and even then only at the expense of arguments considerably more technical than those given here. Proposition \ref{klrstep1''} allows us to avoid this image restriction.
\end{rmk}

\begin{proof}

Under the assumptions on $T$, the Poitou-Tate sequence yields a short exact sequence
\[
0 \to H^1(\gal{F, T}, \br(\fgder)) \xrightarrow{\Psi_T} \bigoplus_{v \in T} H^1(\gal{F_v}, \br(\fgder)) \to \left( H^1(\gal{F, T}, \br(\fgder)^*)\right)^\vee \to 0.
\]
In particular, if $\dim_{\Fp} \coker(\Psi_T)=r$ is non-zero, then $H^1(\gal{F, T}, \br(\fgder)^*)$ contains a non-zero class $\psi_1$. We claim that we can choose a triple $(T_1, \alpha_1, X_{\alpha_1})$ consisting of a maximal torus $T_1$, a root $\alpha_1 \in \Phi(G^0, T_1)$, and a root vector $X_{\alpha_1} \in \mf{g}_{\alpha_1}$ such that $\psi_1(\gal{K'})$ is not contained in $(\Fp X_{\alpha_1})^\perp$ (note that we work with the $\Fp$-span of $X_{\alpha_1}$ rather than the full root space). Indeed, for any $\Fp$-subspace $U$ not equal to the whole of $\br(\fgder)$, there is a root vector not in $U$. To check this, we must check that the $\Fp$-span, or equivalently the $k$-span, of all root vectors in $\fgder$ is equal to the whole of $\fgder$. This claim in turn reduces to the case in which $\fgder$ is simple, where again (using $p \gg_G 0$) it follows from irreducibility of $\fgder$ as a $k[G(k)]$-module. Thus to find the desired triple it suffices to note that $\psi_1(\gal{K'})$ is non-trivial, using the fact that $\br(\fgder)$ and $\br(\fgder)^*$ have no common sub-quotient.

Now we let $\mc{C}_1$ be the collection of $K'$-trivial primes $v$ such that $\psi_1(\sigma_v)$ is not in $(\Fp X_{\alpha_1})^\perp$, and for each $v_1 \in \mc{C}_1$ let $L_{v_1}= \{\phi \in H^1(\gal{F_{v_1}}, \br(\fgder)): \phi(\tau_{v_1}) \in \Fp X_{\alpha_1}\}$. A few applications of the Greenberg-Wiles formula imply that 
\[
\ker\left(H_{L_{v_1}}(\gal{F, T \cup v_1}, \br(\fgder)) \to \bigoplus_{v \in T} H^1(\gal{F_v}, \br(\fgder))\right)=0
\]
(this Selmer group notation means that we impose no condition at the places in $T$), the cokernel of this restriction map has dimension $r-1$, $h^1_{L_{v_1}}(\gal{F, T \cup v_1}, \br(\fgder))-h^1(\gal{F, T}, \br(\fgder))=1$, and $h^1_{L_{v_1}^\perp}(\gal{F, T}, \br(\fgder)^*)=r-1$. Now, if $r-1>0$, then for each $v_1 \in \mc{C}_1$ we can choose a non-zero $\psi_2 \in H^1_{L_{v_1}^\perp}(\gal{F, T}, \br(\fgder)^*)$. Note that $\psi_2$ depends on $v_1$. Then we can repeat the above argument, choosing $(T_2, \alpha_2, X_{\alpha_2})$ such that $\psi_2(\gal{K'})$ is not contained in $(\Fp X_{\alpha_2})^\perp$, and then define a \v{C}ebotarev set $\mc{C}_2(v_1)$ (the notation includes the dependence on the initial choice of $v_1$) as the set of $K'$-trivial $v$ such that $\psi_2(\sigma_v) \not \in (\Fp X_{\alpha_2})^\perp$. The same argument with the Greenberg-Wiles formula shows that 
\[
\ker\left(H^1_{L_{v_1}, L_{v_2}}(\gal{F, T \cup v_1 \cup v_2}, \br(\fgder)) \to \bigoplus_{v \in T} H^1(\gal{F_v}, \br(\fgder)) \right)=0,
\]
and consequently that the dimension of the cokernel of this map is now $r-2$. Proceeding inductively, we obtain \v{C}ebotarev sets $\mc{C}_s(v_{s-1})$, depending on $v_{s-1} \in \mc{C}_{s-1}(v_{s-2})$ (and so on), for $s=1, \ldots, r$, such that for all tuples $(v_1, \ldots, v_r)$ with each $v_s \in \mc{C}_s(v_{s-1})$, the restriction map
\[
H^1_{L_{v_1}, \ldots, L_{v_r}}(\gal{F, T \cup v_1, \ldots, v_r}, \br(\fgder)) \to \bigoplus_T H^1(\gal{F_v}, \br(\fgder))
\]
is surjective. 

In particular, the above argument produces an $\Fp$-basis $\psi_1, \ldots, \psi_r$ of $H^1(\gal{F, T}, \br(\fgder)^*)$, a collection of root vectors $X_{\alpha_1}, \ldots, X_{\alpha_r}$, and a collection of elements $Y_1, \ldots, Y_r \in \oplus_T H^1(\gal{F_v}, \br(\fgder))$ that map to a basis of $\coker(\Psi_T)$: for $Y_i$, we take \textit{any} vector in the image of $H^1_{L_{v_i}}(\gal{F, T \cup v_i}, \br(\fgder)) \to \oplus_T H^1(\gal{F_{v_i}}, \br(\fgder))$ that is not in $\im(\Psi_T)$. (These still span $\coker(\Psi_T)$ because if $\wt{Y}_i$ denotes a lift to $H^1_{L_{v_i}}(\gal{F, T}, \br(\fgder))$ of $Y_i$, the $\{\wt{Y}_i\}_i$ span $\coker(H^1(\gal{F, T}, \br(\fgder))\to H^1_{L_{v_1}, \ldots, L_{v_r}}(\gal{F, T \cup v_1, \ldots, v_r}, \br(\fgder))$, since they are independent for ramification reasons.) For each $i$, we also can fix an $\Fp$-basis $\omega_{i, 1}, \ldots, \omega_{i, r-1}$ of $H^1_{L_{v_i}^\perp}(\gal{F, T}, \br(\fgder)^*)$. Now we define the following \v{C}ebotarev condition:
\[
\mc{C}_i=\left \{\text{$K'$-trivial primes $v$ such that $\psi_i(\sigma_v) \not \in (\Fp X_{\alpha_i})^\perp$ and $\omega_{i, k}(\sigma_v) \in (\Fp X_{\alpha_i})^\perp$ for all $k=1, \ldots, r-1$}\right\}.
\]
We know that $v_i \in \mc{C}_i$, so each $\mc{C}_i$ is in fact a non-empty \v{C}ebotarev condition (without this observation, the conditions defining $\mc{C}_i$ could be incompatible). Now, for all $v \in \mc{C}_i$, we define $L_v$ as before to be those classes $\phi \in H^1(\gal{F_v}, \br(\fgder))$ such that $\phi(\tau_v) \in \Fp X_{\alpha_i}$ and deduce an exact sequence
\[
0 \to H^1_{L_v}(\gal{F, T \cup v}, \br(\fgder)) \to \bigoplus_{w \in T} H^1(\gal{F_w}, \br(\fgder)) \to (H^1_{L_v^\perp}(\gal{F, T}, \br(\fgder)^*))^\vee \to 0;
\]
indeed, injectivity of the first map follows from the same Euler-characteristic argument as above (using that $\psi_i|_v \not \in L_v^\perp$), the composite is clearly zero, and then exactness follows from counting dimensions. We claim that $Y_i$ lies in the image of $H^1_{L_v}(\gal{F, T \cup v}, \br(\fgder))$, for which it suffices to check that $Y_i$ annihilates $H^1_{L_v^\perp}(\gal{F, T}, \br(\fgder)^*)$. We know that $Y_i$ annihilates $H^1_{L_{v_i}^\perp}(\gal{F, T}, \br(\fgder)^*)$ (using exactness of the above sequence for $v_i$), so it suffices (and is in fact necessary) to observe that 
\[
H^1_{L_{v_i}^\perp}(\gal{F, T}, \br(\fgder)^*)= H^1_{L_{v}^\perp}(\gal{F, T}, \br(\fgder)^*);
\]
this holds because both subspaces of $H^1(\gal{F, T}, \br(\fgder)^*)$ are equal to the span of $\omega_{i, 1}, \ldots, \omega_{i, r-1}$.
\end{proof}
We will also need the following simpler variant of Proposition \ref{klrstep1''}.
\begin{lemma}\label{gens}
Continue with the hypotheses of Proposition \ref{klrstep1''}. Let $Z \in \fgder$ be any non-zero element. There is a \v{C}ebotarev set $\mc{C}$ of $K'$-trivial primes and for each $v \in \mc{C}$ a class $h^{(v)} \in H^1(\gal{F, T \cup v}, \br(\fgder))$ such that
\begin{itemize}
\item the restriction $h^{(v)}|_T$ is independent of $v \in \mc{C}$; and
\item $h^{(v)}(\tau_v)$ spans the line $\Fp Z$.
\end{itemize}
\end{lemma}
\begin{proof}
Recall from the discussion preceding Lemma \ref{localmodp^2} that we have enlarged $T$ to ensure that for all $i \in I$, $H^1(\gal{F, T}, W_i^*) \neq 0$. Since $(\Fp Z)^\perp \subset (\fgder)^*$ is a proper subspace, it does not contain some isotypic piece $(W_i^*)^{\oplus m_i}$, hence it does not contain some $\gal{F, T}$-equivariantly embedded $W_i^* \hookrightarrow (W_i^*)^{\oplus m_i}$, and so there is a $\psi \in H^1(\gal{F, T}, \br(\fgder)^*)$ such that $\psi(\gal{K'})$ is not contained in $(\Fp Z)^\perp$ (namely, a $\psi$ supported on a suitable copy of $W_i^*$). We can now repeat the argument of Proposition \ref{klrstep1''}. In brief, fix a $K'$-trivial prime $v_1$ such that $\psi(\sigma_{v_1})$ does not belong to $(\Fp Z)^\perp$, and as before define $L_{v_1}$ to be the set of $\phi \in H^1(\gal{F_{v_1}}, \br(\fgder))$ such that $\phi(\tau_{v_1}) \in \Fp Z$. The same analysis shows that there is an element
\[
Y \in \im(H^1_{L_{v_1}}(\gal{F, T \cup v_1}, \br(\fgder)) \to \bigoplus_{w \in T} H^1(\gal{F_w}, \br(\fgder)) \setminus \im(\Psi_T),
\]
and that if we let $\omega_1, \ldots, \omega_s$ be a basis of the (codimension 1) subspace $H^1_{L_{v_1}^\perp}(\gal{F, T}, \br(\fgder)^*)\subset H^1(\gal{F, T}, \br(\fgder)^*)$, then
\[
\mc{C}_Z= \{\text{$K'$-trivial primes $v: \psi(\sigma_v) \not \in (\Fp Z)^\perp$ and $\omega_j(\sigma_v) \in (\Fp Z)^\perp$ for all $i=1, \ldots, s$}\}
\]
is a \textit{non-empty} (because $v_1 \in \mc{C}_Z$) \v{C}ebotarev condition. Then as in Proposition \ref{klrstep1''}, we also see that for all $v \in \mc{C}_Z$ there is a class $h^{(v)} \in H^1_{L_v}(\gal{F, T \cup v}, \br(\fgder))$ such that $h^{(v)}|_T= Y$, and $h^{(v)}(\tau_v)$ spans $\Fp Z$.
\end{proof}
Now we fix any finite set of root vectors (for possibly different split maximal tori) $\{X_{\alpha_a}\}_{a \in A}$ such that
\[
\sum_{a \in A} \Fp[\gal{F}]X_{\alpha_a}= \fgder.
\]
(Such a collection $\{X_{\alpha_a}\}$ clearly exists, since for any proper subspace $U$ of $\fgder$, there is some root vector not in $U$: see the proof of Proposition \ref{klrstep1''}.) Lemma \ref{gens} yields \v{C}ebotarev sets $\mc{C}_a= \mc{C}_{X_{\alpha_a}}$ and classes $Y_a \in \bigoplus_{w \in T} H^1(\gal{F_w}, \br(\fgder))$ such that for all $v \in \mc{C}_a$, there is a class $h^{(v)} \in H^1(\gal{F, T \cup v}, \br(\fgder))$ satisfying $h^{(v)}(\tau_v) \in \Fp X_{\alpha_a} \setminus 0$ and $h^{(v)}|_T=Y_a$. Consider the class
\[
(z_T^{\mr{der}})'= z_T^{\mr{der}}- \sum_{a \in A} Y_a.
\]
This new element may or may not be in the image of $\Psi_T$, but we can in any case invoke Proposition \ref{klrstep1''} to produce a finite set $\{Y_b\}_{b \in B} \subset \bigoplus_{w \in T} H^1(\gal{F_w}, \br(\fgder))$ that spans $\coker(\Psi_T)$ over $\Fp$, and, for each $b\in B$, a \v{C}ebotarev set $\mc{C}_b$ of $K'$-trivial primes and a root vector $X_{\alpha_b}$, and for each $v \in \mc{C}_b$ a class $g^{(v)}\in H^1(\gal{F, T \cup v}, \br(\fgder))$ such that $g^{(v)}|_T=Y_b$ and $g^{(v)}(\tau_v) \in \Fp X_{\alpha_b} \setminus 0$ (the reason for the shift in notation to $g^{(v)}$ will become apparent at the end of this paragraph). In particular, 
we can write
\[
(z_T^{\mr{der}})'= h^{\mr{old}}+ \sum_{b \in B} c_bY_b
\]
for some class $h^{\mr{old}} \in H^1(\gal{F, T}, \br(\fgder))$ and some $c_b \in \Fp$. We discard those $b \in B$ such that $c_b=0$. Thus, for all tuples 
\[
(v_a)_{a \in A} \times (v_b)_{b \in B} \in \prod_{a \in A} \mc{C}_a \times \prod_{b \in B} \mc{C}_b,
\]
we can write 
\[
z_T^{\mr{der}}= h^{\mr{old}}|_T+\sum_{a \in A} h^{(v_a)}|_T + \sum_{b \in B} c_b g^{(v_b)}|_T.
\]
Note that the vectors $h^{(v_a)}(\tau_{v_a})$ are non-zero multiples of $X_{\alpha_a}$ for all $a \in A$, so the collection $\{h^{(v_a)}(\tau_{v_a})\}_{a \in A}$ is a set of $\Fp[\gal{F}]$-generators of $\fgder$. Having made note of this, we will in the argument that follows not need to preserve the distinction between the sets $A$ and $B$, so we set $N= A \cup B$. In order to preserve this uniformity of notation, for all $b \in B$ and $v \in \mc{C}_b$ we set $h^{(v)}= c_b g^{(v)}$, so we can re-express the above equality as
\[
z_T^{\mr{der}}= h^{\mr{old}}|_T+ \sum_{n \in N} h^{(v_n)}|_T
\]
for any $\un{v}=(v_n)_{n \in N} \in \prod_{n \in N} \mc{C}_n$.

We will need to argue in terms of Dirichlet densities of $N$-tuples of primes. In what follows, we define the Dirichlet density of a subset $P$ of $\{\text{primes of $F$}\}^N$ to be (if it exists)
\[
\delta(P)= \lim_{s \to 1^+} \frac{\sum_{\underline{v} \in P} N(\underline{v})^{-s}}{\sum_{\text{all $\underline{v}$}} N(\underline{v})^{-s}},
\]
where $N(\underline{v})= \prod_{n \in N} N(v_n)$. In particular, the density of a product $P= \prod_{n \in N} P_n$ of sets $P_n$ of primes exists if each $P_n$ has a density, and in this case $\delta(P)= \prod_{n \in N} \delta(P_n)$. We make a corresponding definition of upper Dirichlet density $\delta^+(P)$ of a set of $N$-tuples of primes. In particular, the preceding discussion yields a \v{C}ebotarev set $\mc{C}= \prod_{n \in N} \mc{C}_n$ of positive Dirichlet density.

The following argument substantially uses global duality, and we need to preface with a technical clarification of what coefficients we can take in the duality pairings. We have the $\Fp[\gal{F}]$-isotypic decomposition 
\[
\br(\fgder)= \bigoplus_{i \in I} V_i = \bigoplus_{i \in I} W_i^{\oplus m_i},
\]
where the various $W_i$ are mutually non-isomorphic irreducible $\Fp[\gal{F}]$-modules with endomorphism algebras $k_{W_i} = \mr{End}_{\Fp[\gal{F}]}(W_i)$ (a finite extension of $\Fp$). We may (and do) fix an isomorphism of $V_i$
with $W_i \otimes_{\Fp}\mbb{F}_{p^{m_i}}$ as
$k_{W_i}[\gal{F}]$-modules (with trivial Galois action on
$\mbb{F}_{p^{m_i}}$). This gives $V_i$ the structure of an
$A_i[\gal{F}]$-module, where
$A_i := k_{W_i} \otimes_{\Fp} \mbb{F}_{p^{m_i}}$, with $V_i$ being
finite free as an $A_i$-module. In the sequel, duals and duality
pairings will be considered with respect to this fixed structure.\footnote{We
  use the trace maps to identify duals over the various etale
  $\Fp$-algebras that we consider.} 

\begin{prop}\label{doublingprop}
 There is a finite set of $K'$-trivial
  primes $Q$ disjoint from $T$ and a class
  $h \in H^1(\gal{F, T \cup Q}, \br(\fgder))$ such that
  $h|_T = z_T^{\mr{der}}$ and $(1+ph)\rho_2|_w$ belongs to
  $\Lift^{\nu, \alpha_w}_{\br|_w,2, \mr{ram}}(\mc{O}/p^2)$ for all
  $w \in Q$ (for some $\alpha_w \in \Phi(G^0,T_w)$, which may depend on $w \in Q$).
\end{prop}
\begin{proof}
We have seen that there is a class $h^{\mr{old}} \in H^1(\gal{F, T}, \br(\fgder))$ such that for any $N$-tuple $\un{v}= (v_n)_{n \in N} \in \mc{C}= \prod_{n \in N} \mc{C}_n$, the global class $h(\un{v})=h^{\mr{old}}+ \sum_{n \in N} h^{(v_n)}$ satisfies $h(\un{v})|_T= z^{\mr{der}}_T$. Since we cannot say anything about the restrictions $h(\un{v})|_{v_n}$, we will use the ``doubling method" of \cite{klr} to find the desired $Q$ and $h$. To that end, for any two $N$-tuples $\un{v}, \un{v}' \in \mc{C}$, we consider the class
\begin{equation}\label{doublingeqtn}
h= h^{\mr{old}}- \sum_{n \in N} h^{(v_n)}+ 2 \sum_{n \in N} h^{(v'_n)} \in H^1(\gal{F, T \cup \{v_n\} \cup \{v'_n\}}, \br(\fgder)),
\end{equation}
which still satisfies $h|_T= z_T^{\mr{der}}$ (and the inertial conditions dictated by the construction of the classes $h^{(v_n)}$). The argument will show that for a suitable choice of $\un{v}$ and $\un{v}'$, $h$ will satisfy the conclusion of the Proposition with the set $Q$ equal to $\{v_n\}_{n \in N} \cup \{v'_n\}_{n \in N}$.

We first restrict to a positive upper-density subset $\mf{l} \subset \mc{C}$ (now no longer necessarily a product of \v{C}ebotarev sets) such that the $N$-tuples $(\sum_{n \in N} h^{(v_n)}(\sigma_{v_m}))_{m \in N}$, $(h^{\mr{old}}(\sigma_{v_m}))_{m \in N}$, and $(h^{(v_n)}(\tau_{v_n}))_{n \in N}$ are independent of the choice of $\un{v} \in \mf{l}$; this is possible since as we vary over $\mc{C}$, these $N$-tuples all take finitely many values. In particular, we write $X_n$ for the now independent-of-$\un{v}$ value $h^{(v_n)}(\tau_{v_n})$ (for all $n \in N$, this is a non-zero multiple of $X_{\alpha_n}$). Recall the decomposition $\br(\fgder)= \bigoplus_{i \in I} V_i = \bigoplus_{i \in I} W_i^{\oplus m_i}$
into $\Fp[\gal{F}]$-isotypic components. For $i \in I$ we let $X_{n, i}$ denote the $V_i$-component of $X_n$. By construction, therefore, we have
\[
\sum_{n \in N} \Fp[\gal{F}] X_{n, i}= V_i.
\]

We will show that for any fixed $N$-tuples $(C_m)_{m \in N}$ and $(C'_m)_{m \in N}$ of elements of $\fgder$, there exist $\un{v}, \un{v}' \in \mf{l}$ such that
\begin{align}\label{doubling}
&\sum_{n \in N} h^{(v'_n)}(\sigma_{v_m})= C_m, \\ \label{doubling2}
&\sum_{n \in N} h^{(v_n)}(\sigma_{v'_m})= C'_m,
\end{align}
for all $m \in N$. This will suffice to prove the Proposition, since, by Equation (\ref{doublingeqtn}), it will allow us to prescribe the values $h(\sigma_{v_m})$ and $h(\sigma_{v'_m})$ for all $m \in N$; we then 
choose the $C_m$ and $C'_m$ such that the values $(1+ph)\rho_2(\sigma_{v_m})$ and $(1+ph)\rho_2(\sigma_{v'_m})$ satisfy the conditions of Definition \ref{ramtrivlifts} (with $\alpha= \alpha_m$, of course). We can make such a choice by the argument of Lemma \ref{bigimage}.   

We will now study the condition, for \textit{fixed} $\un{v}= (v_n)_{n \in N} \in \mf{l}$, imposed on $\un{v'}$ by Equations (\ref{doubling}) and (\ref{doubling2}), beginning with Equation (\ref{doubling2}). For each $n \in N$, consider the maximal Galois extension $K^{(v_n)}$ of $F$ inside $K_{h^{(v_n)}}$ that is unramified at $v_n$; this contains $K$, and 
\begin{equation}\label{span}
\sum_{n \in N} h^{(v_n)}(\Gal(K_{h^{(v_n)}}/K^{(v_n)}))= \fgder,
\end{equation}
since the $n$-component of this sum contains $X_n$ and is $\Fp[\gal{F}]$-stable (only the $n \in A$ are needed to guarantee Equation (\ref{span}) holds).

For each $m \in N$, we consider the \v{C}ebotarev condition $\mf{w}_m$ on trivial primes $w$ requiring that $w$ split in all $K^{(v_n)}$ and that
\[
\sum_{n \in N} h^{(v_n)}(\sigma_w)= C_m.
\]
Since the composite of the fields $K^{(v_n)}$ is still unramified at each $v_n$, Equation (\ref{span}) implies that this condition is non-empty. Moreover, since $K_{h^{(v_n)}} \cap K'=K^{(v_n)}$, $\mf{w}_m$ induces a non-empty \v{C}ebotarev condition $\mf{w}'_m$ where we further impose the condition that all primes in $\mf{w}'_m$ are split in $K'$.

Now we turn to the condition needed to satisfy Equation (\ref{doubling}). For all $m \in N$ and $i \in I$, letting $\{\eta^{(v_m)}_{i, j}\}_{j=1}^{d_i}$ be elements of $H^1(\gal{F, T \cup v_m}, W_i^*)$ that lift a $k_{W_i}$-basis of $H^1(\gal{F, T \cup v_m}, W_i^*)/H^1(\gal{F, T}, W_i^*)$, we have for all $m, n, i, j$ the global duality relation
\begin{align*}
\langle \eta^{(v_m)}_{i, j}(\tau_{v_m}), h_i^{(v'_n)}(\sigma_{v_m}) \rangle&= -\sum_{x \in T} \langle \eta^{(v_m)}_{i, j}, h_i^{(v'_n)}\rangle_x - \langle \eta^{(v_m)}_{i, j}(\sigma_{v'_n}) , h_i^{(v'_n)}(\tau_{v'_n})\rangle,\\
&= -\sum_{x \in T} \langle \eta^{(v_m)}_{i, j}, h_i^{(v_n)}\rangle_x - \langle \eta^{(v_m)}_{i, j}(\sigma_{v'_n}) , X_{n, i}\rangle,
\end{align*}
where we systematically work with the $A_i$-linear pairings. Summing over $n$, we want to show that for all $m \in N$, $i \in I$, $j=1, \ldots, d_i$, we can prescribe by a \v{C}ebotarev condition (depending on our fixed $\un{v}$) on $\un{v}' \in \mf{l}$ the values
\[
\sum_{n \in N} \langle \eta^{(v_m)}_{i, j}(\sigma_{v'_n}), X_{n, i} \rangle \in A_i,
\]
for then we can achieve the same for the values
\[
\langle \eta^{(v_m)}_{i, j}(\tau_{v_m}), \sum_{n \in N} h_i^{(v'_n)}(\sigma_{v_m}) \rangle.
\]
Prescribing these values for varying $m, i, j$ will allow us to achieve the equality of Equation (\ref{doubling}).

The splitting fields $K_{\eta^{(v_m)}_{i, j}}$ are strongly linearly disjoint over $K$ as we vary $m\in N$, $i\in I$, and $j=1, \ldots, d_i$,\footnote{Since the $W_i^*$ are irreducible, the fields are disjoint as $m$ varies because $K_{\eta^{(v_m)}_{i, j}}$ is ramified at $v_m$ (and not at $v_{m'}$ for $m' \neq m$); they are disjoint as $i$ varies because the $W_i^*$ are mutually non-isomorphic; and they are disjoint as $j$ varies by Lemma \ref{lindisjoint}.} so it suffices for this last claim to note that for any fixed non-zero vector $w_i^* \in W_i^*$ we have
\[
\sum_{n \in n} \langle W_i^*, X_{n, i} \rangle= \sum_{n \in n} \langle \Fp[\gal{F}]w_i^*, X_{n, i} \rangle = \sum_{n \in n} \langle w_i^*, \Fp[\gal{F}]X_{n, i} \rangle= \langle w_i^*, V_i \rangle= A_i.
\]

The \v{C}ebotarev condition on $\un{v}'$ thus obtained is independent from the \v{C}ebotarev condition $\prod_{m \in N} \mf{w}'_m$ determined above, since $K'$ and the $K_{h^{(v_n)}}$ are strongly linearly disjoint from the $K_{\eta^{(v_m)}_{i, j}}$ ($\br(\fgder)$ and $\br(\fgder)^*$ have no $\Fp[\gal{F}]$-subquotient in common). In sum, we obtain a non-empty \v{C}ebotarev condition $\mf{l}_{\un{v}}$ on tuples $\un{v}'$ of $K'$-trivial primes such that the desired equalities
\begin{align*}
&\sum_{n \in N} h^{(v'_n)}(\sigma_{v_m})= C_m \\
&\sum_{n \in N} h^{(v_n)}(\sigma_{v'_m})= C'_m
\end{align*}
hold for any $\un{v}' \in \mf{l} \cap \mf{l}_{\un{v}}$. We have no assurance that this intersection is non-empty, so now we must invoke the limiting logic of \cite{klr} and \cite{ramakrishna-hamblen} that allows the doubling method to succeed. 
If for each member of a finite subset $\{\un{v}_1, \ldots, \un{v}_s\} \subset \mf{l}$, the intersection $\mf{l} \cap \mf{l}_{\un{v}_k}$ is empty, then $\mf{l} \setminus \{\un{v}_1, \ldots, \un{v}_s\}$ is contained in $\mf{l} \cap \bigcap_{k=1}^s \overline{\mf{l}_{\un{v}_k}}$. We will control the upper-density of this latter intersection. For each $k=1, \ldots, s$, let $K_{h^{(\underline{v}_k)}}$ denote the composite of the fields $K_{h^{(\un{v}_{k, n})}}$ for $n \in N$, and let $K_{\eta^{(\un{v}_k)}}$ denote the composite of the fields $K_{\eta^{(\un{v}_{k, n})}_{i, j}}$ for $n \in N$, $i \in I$, $j=1, \ldots, d_i$. For fixed $k$, the fields $K_{h^{(\underline{v}_k)}}$ and $K_{\eta^{(\un{v}_k)}}$ are linearly disjoint over $K$, and $\mf{l}_{\un{v}_k}$ is a \v{C}ebotarev condition in their composite. As $k$ varies, the $K_{\eta^{(\un{v}_k)}}$ will be strongly linearly disjoint, but the $K_{h^{(\un{v}_k)}}$ may not be disjoint over $K$. This is where the field $K'$ becomes significant.

We now replace the fields $K_{h^{(\underline{v}_k)}}$ and $K_{\eta^{(\un{v}_k)}}$ with their composites $K'_{h^{(\underline{v}_k)}}$ and $K'_{\eta^{(\un{v}_k)}}$ with $K'$. 
We will finally be able to make the limiting argument by observing that, even as $k$ varies, the fields $K'_{h^{(\underline{v}_k)}}$ and $K'_{\eta^{(\un{v}_k)}}$ are now all strongly linearly disjoint over $K'$.
In the intersection $\mc{C} \cap \bigcap_{k=1}^s \overline{\mf{l}_{\un{v}_k}}$, each term $\mc{C} \cap \overline{\mf{l}_{\un{v}_k}}$ is for some finite extension $L_k/K'$ a \v{C}ebotarev condition on primes in $F$ picking out a proper subset of elements of $\Gal(L_k/K')$: for the properness, note that each $\overline{\mf{l}_{\un{v}_k}}$ is a union of complements of the proper conditions we have imposed on each $K'_{h^{(\un{v}_{k, n})}}$ for $n \in N$ and $K'_{\eta^{(\un{v}_k)}}$ for $n \in N$, $i \in I$, $j=1, \ldots, d_i$, so we get a proper condition by the disjointness of these fields over $K'$. Also note that the degrees of the extensions $L_k/K'$ are bounded independently of $\un{v}_k \in \mf{l}$, in terms of $\# \fgder$, $\# N$, and $\sum_{i \in I} d_i$. Finally, we can conclude that 
\[
\delta^+(\mf{l} \setminus \{\un{v}_1, \ldots, \un{v}_s\}) \leq \delta(\mc{C})(1- \varepsilon)^s
\]
for some $\varepsilon > 0$. Letting $s$ tend to infinity, 
we see that $\delta^+(\mf{l})$ is less than any positive number, contradicting the fact that $\mf{l}$ has positive upper-density. We conclude that for some $\un{v} \in \mf{l}$, there is a $\un{v}' \in \mf{l} \cap \mf{l}_{\un{v}}$, and so the proof is complete.
\end{proof}
The arguments of this section yield a generalization to any reductive group of the main theorem of \cite{klr}. We sketch here a somewhat simplified version:
\begin{cor}\label{infiniteram}
Let $\br \colon \gal{F, S} \to G(k)$ satisfy Assumption \ref{multfree}, except we do not require that $K$ does not contain $\mu_{p^2}$. (In particular, the results of \S \ref{groupsection} will show that for $p \gg_G 0$, it suffices here to assume $\br|_{\gal{\tF(\zeta_p)}}$ is absolutely irreducible.) Assume for simplicity that $G=G^0$, and fix a lift $\nu \colon \gal{F, S} \to A(\mc{O})$ of $\mu \circ \br$. Assume that for all $v \in S$, there are lifts $\rho_v \colon \gal{F_v} \to G(\mc{O})$ of $\br|_{\gal{F_v}}$ with multiplier $\nu$. Then there exists an infinitely ramified lift
\[
\xymatrix{
& G(\mc{O}) \ar[d] \\
\gal{F} \ar@{-->}[ur]^{\rho} \ar[r]_{\br} & G(k)
}
\]
such that $\rho|_{\gal{F_v}}= \rho_v$ for all $v \in S$, and $\rho(\gal{F})$ contains $\wh{G^{\mr{der}}}(\mc{O})$.
\end{cor}
\begin{rmk}
In this degree of generality, it is not known, but certainly expected, that local lifts $\rho_v$ as above always exist.
\end{rmk}
\begin{proof}
See Remark \ref{infinitetriv} for an explanation of the slight modification of our hypotheses. For any $G(\mc{O})$-valued representation $\lambda$, write $\lambda_n$ for its reduction modulo $p^n$. Applying Proposition \ref{doublingprop}, we can find a lift $\rho_2 \colon \gal{F, T \cup Q_1} \to G(\mc{O}/p^2)$ such that $\rho_2|_{\gal{F_v}}= \rho_{v, 2}$ for all $v \in S$, and for all $v \in T \cup Q_1 \setminus S$, $\rho_2|_{\gal{F_v}}$ admits a lift $\rho_v$ to $G(\mc{O})$ (by Lemma \ref{trivsmooth}). We then iterate the argument of Proposition \ref{doublingprop}: there are no obstructions to lifting $\rho_2$ to $G(\mc{O}/p^2)$, and then by introducing further trivial primes $Q_2$ of ramification we may find a lift $\rho_3 \colon \gal{F, T \cup Q_1 \cup Q_2} \to G(\mc{O}/p^3)$ such that $\rho_3|_{\gal{F_v}}= \rho_{v, 3}$ for all $v \in T \cup Q_1$, and for $v \in Q_2$, $\rho_3|_{\gal{F_v}}$ lies on some $\Lift_{\br|_{\gal{F_v}}}^{\alpha}$, so again by Lemma \ref{trivsmooth} admits a lift $\rho_v$ to $G(\mc{O})$. We thus inductively construct $\rho= \varprojlim_n \rho_n$ having the desired properties (the statement about $\im(\rho)$ is proven by inductively showing that $\im(\rho_2)$ contains $\wh{G^{\mr{der}}}(\mc{O}/p^2)$ implies $\im(\rho_n)$ contains $\wh{G^{\mr{der}}}(\mc{O}/p^n)$ for all $n$: see the conclusion of the proof of Theorem \ref{mainthm}).
\end{proof}

\section{Construction of auxiliary primes for lifting past $\mc{O}/p^3$}\label{auxiliarysection}
In this section we explain how to construct the auxiliary primes that will allow us (in \S \ref{finallift}) to lift a suitable $\rho_3 \colon \gal{F} \to G(\mc{O}/p^3)$ to characteristic zero. We return to our general hypotheses on $G$ from \S \ref{defprelimsection}, and as in \S \ref{klrsection} we assume that $\br$ surjects onto $\pi_0(G)$. We will need to make some assumptions both about $\br$ and about $\rho_2:= \rho_3 \pmod {p^2}$. The natural setting of this section will therefore be to work with a fixed $\rho_2$. Section \ref{klrsection} has explained under what hypotheses we can begin with a $\br$ and produce a lift $\rho_2$ to which the results of this section apply. Section \ref{finallift} will then combine the results of the current section and \S \ref{klrsection} to prove the main theorem.

We thus for the rest of this section fix a continuous homomorphism
\[
 \rho_2 \colon \gal{F, T} \to G(\mc{O}/p^2)
\]
unramified outside a finite set of finite primes $T$, with $\mu \circ \rho_2= \nu$, and whose mod $p$ reduction we denote by $\br$. 
The arguments of this section will require a more restrictive assumption on $\im(\br)$ than those of \S \ref{klrsection}, and so from now on the following assumptions on $\rho_2$ and $\br$ will be in effect:
\begin{assumption}\label{big}
 Assume that $p$ is greater than a suitable absolute bound depending only on $G$.\footnote{A precise bound can be extracted from our arguments, but we do not do this.} Assume that $\rho_2$ satisfies the following:
\begin{enumerate}
 \item $\im(\rho_2)$ contains $\wh{G}^{\mr{der}}(\mc{O}/p^2)$. 
 \item The field $K= \tF(\br(\fg^{\mr{der}}), \mu_p)$ does not contain $\mu_{p^2}$.
 \item $H^1(\Gal(K/F), \br(\fg^{\mr{der}}))=0$.
 \item $H^1(\Gal(K/F), \br(\fg^{\mr{der}})^*)$=0.
 \item $\br(\fgder)$ and $\br(\fgder)^* \cong \br(\fgder)(1)$ are semisimple $\mathbb{F}_p[\gal{\tF}]$-modules (equivalently, $\Fp[\gal{F}]$, $k[\gal{\tF}]$, or $k[\gal{F}]$-modules) having no common $\Fp[\gal{\tF}]$-sub-quotient, and neither contains the trivial representation of $\gal{\tF}$.
 \item Let $W$ be any simple $\Fp[\gal{\tF}]$-quotient $\br(\fgder)/W' \xrightarrow{\sim} W$. If the group extension
 \[
 1 \to \fg/W' \to \im(\rho_2)/W' \to \im(\br) \to 1
 \]
 splits, then we assume:
 \begin{enumerate}
 \item For each $G$-orbit of simple factors $\br(\oplus_i \fg_i)$ of $\br(\fgder)$, $W$ appears in $\br(\fg_i)$ with $\Fp[\gal{\tF}]$-multiplicity at most 1. 
 \item Moreover, each such $W$ satisfies $\End_{\Fp[\gal{\tF}]}(W) \cong k$ (see Lemma \ref{multfreelemma} below).
 \end{enumerate}
\end{enumerate}
\end{assumption}
\begin{rmk}
Hypotheses (3)-(5) are familiar from the arguments of Ramakrishna or of Taylor-Wiles. The second will for us be an automatic consequence of the local structure of $\br$ at primes above $p$ (Lemmas \ref{p2} and \ref{FLp2}). We have explained how to arrange the first in \S \ref{klrsection}. Hypothesis (6) is our most serious restriction. We have phrased it in this rather technical way to emphasize that the difficulties arise not from high multiplicities or group-theoretic splittings individually, but from their confluence. 

Note that $F(\br(\mf{g}^{\mr{der}})) \subseteq F(\br(\fgm)) \subseteq \tF(\br(\mf{g}^{\mr{der}}))= \tF(\br(\fgm))$, and the first field in this chain has index prime to $p$ in the last field. In particular, the assumptions could equally well be rephrased with $\br(\fgm)$ instead of $\br(\fg^{\mr{der}})$. 
\end{rmk}
\begin{lemma}\label{multfreelemma}
Let $\Gamma$ be a group, and let $V$ be a finite-dimensional $k[\Gamma]$-module, assumed to be multiplicity-free as $\Fp[\gal{F}]$-module. Decompose $V= \oplus_{i \in I} U_i$ into irreducible $k[\Gamma]$-modules. Then each $U_i$ is in fact an irreducible $\Fp[\Gamma]$-module, and $\End_{\Fp[\Gamma]}(U_i)= \End_{k[\Gamma]}(U_i)$ is a finite extension of $k$.
\end{lemma}
\proof
Let $U \subset U_i$ be any $\Fp[\Gamma]$-stable subspace. Since $U_i$ is irreducible as $k[\Gamma]$-module, the natural map $U \otimes_{\Fp} k \to U_i$ is surjective. Now, as $\Fp[\Gamma]$-module, any quotient of $U \otimes_{\Fp} k$ is isomorphic to a direct sum of copies of $U$. Since $U_i$ is multiplicity-free as $\Fp[\Gamma]$-module, we conclude that $U= U_i$. For the second point, note that 
\[
\End_{\Fp[\Gamma]}(U_i)= \Hom_{k[\Gamma]}(U_i \otimes_{\Fp} k, U_i)= \oplus_{\sigma \in \Aut(k)} \Hom_{k[\Gamma]}({}^\sigma U_i, U_i).
\]
If there were a non-trivial $k[\Gamma]$-map (necessarily an isomorphism) ${}^\sigma U_i \to U_i$ for some non-identity $\sigma \in \Aut(k)$, then $U_i$ would descend to the fixed field $k^{\sigma}$ (by \cite[Lemme 6.13]{deligne-serre}), i.e. $U_i \cong U'_i \otimes_{k^\sigma} k$ for some $k^{\sigma}[\Gamma]$-module $U'_i$. As in the previous paragraph, this contradicts the multiplicity-free hypothesis on $U_i$. Thus the natural map $\End_{k[\Gamma]}(U_i) \to \End_{\Fp[\Gamma]}(U_i)$ is an isomorphism. Finally, this common endomorphism ring is a finite field extension $k_{U_i}$ of $k$, since the Brauer group of $k$ is trivial.
\endproof
In light of the lemma, the problematic simple factors $W$ arising in part (6) of Assumption \ref{big} are automatically $k[\gal{F}]$-stable for the given $k$-multiplication on $\fgder$, and we after the fact see that this ambient $k$-multiplication exhausts $\End_{\Fp[\gal{\tF}]}(W)$.

Our aim now is to explain how to allow ramification at trivial primes satisfying the conditions in Definition \ref{unrtrivlifts} to kill Selmer and dual Selmer classes for $\br(\fgder)$; a technical point arises in the argument when $G_{\mu}$ is not semisimple, and we will deal with cohomology classes supported on $\br(\mf{z}_{\mu})$, the center of $\fgm$, by a different method in \S \ref{finallift}. The first, easy, step is the following:
\begin{lemma}\label{disjoint}
Assumption \ref{big} is in effect. Suppose we are given non-zero elements $\phi \in H^1(\gal{F, T}, \br(\fgder))$ and $\psi \in H^1(\gal{F, T}, \br(\fgder)^*)$. By hypothesis, their restrictions to $\gal{K}$ cut out Galois extensions $K_{\phi}/F$ and $K_{\psi}/F$ that both strictly contain $K$. Then
\[
\Gal(K(\rho_2(\fg^{\mr{der}}))K_{\phi}K_{\psi}K(\mu_{p^2})/K) \xrightarrow{\sim} \Gal(K(\rho_2(\fg^{\mr{der}}))K_{\phi}/K) \times \Gal(K_{\psi}/K) \times \Gal(K(\mu_{p^2})/K).
\]
\end{lemma}
\proof

First note that $\Gal(K(\rho_2(\fg^{\mr{der}})/K)$ is isomorphic (as $\gal{F}$-module) to $\br(\fg^{\mr{der}})$, using the assumption $\im(\rho_2) \supset \wh{G}^{\mr{der}}(\mc{O}/p^2)$. The assumption on no common sub-quotient implies $K(\rho_2(\fg^{\mr{der}}))K_{\phi}$ and $K_{\psi}$ are linearly disjoint. To show that $K(\mu_{p^2})$ and $K(\rho_2(\fg^{\mr{der}}))K_{\phi}K_{\psi}$ are linearly disjoint over $K$, note that $\Gal(K/F)$ acts trivially on $\Gal(K(\mu_{p^2})/K) \xrightarrow{\sim} \Gal(\Q(\mu_{p^2})/\Q(\mu_p))$ (since $\mu_{p^2}$ is not contained in $K$), whereas $\Gal(K(\rho_2(\fg^{\mr{der}}))K_{\phi}K_{\psi}/K)$ has no proper quotient that is trivial as $\Gal(K/F)$-module, by assumption on $\br(\fgder)$ and $\br(\fgder)^*$.

\endproof

We continue with the assumptions of Lemma \ref{disjoint} and now study the relation between the fields $K(\rho_2)$ and $K_{\phi}$. The main difficulty that arises in our lifting method is that the fields $K(\rho_2)$ and $K_{\phi}$ need not be linearly disjoint over $K$. The multiplicity-free hypothesis in Assumption \ref{big} is a way to control the interaction between these two fields even when, for instance, $K(\rho_2)$ contains $K_{\phi}$. Before proceeding to our arguments, we give an example 
\begin{eg}\label{sl2eg}
Assume for simplicity that $G^0$ is semisimple. Then our assumption on $\im(\rho_2)$ implies that there is a short exact sequence
\[
1 \to \fg \to \im(\rho_2) \to \im(\br) \to 1.
\]
Now let $W'$ be any $\Fp[\gal{F}]$-submodule of $\br(\fg)$, and suppose that the sequence
\begin{equation}\label{projection}
1 \to \fg/W' \to \im(\rho_2)/W' \to \im(\br) \to 1
\end{equation}
splits. Then $(\rho_2 \mod {W'})|_{\gal{K}}$ is an element of $\ker\left(H^1(\gal{K, T}, \br(\fg)/W')^{\gal{F}}\to H^2(\Gal(K/F), \br(\fg)/W')\right)$, hence arises as the restriction of some class $\phi \in H^1(\gal{F, T}, \br(\fg)/W')$. Clearly then $K_{\phi}$ is contained in $K(\rho_2)$. 

It can happen that even when $\br$ is irreducible, and $\im(\rho_2) \to \im(\br)$ does not split, that the sequence (\ref{projection}) does split for some non-trivial constituent $W'$. For instance, consider $\br= \varphi \circ \bar{r}$ where $\bar{r}$ is two-dimensional (with big image), and $\varphi \colon \mr{PGL}_2 \to G$ is a principal $\mr{SL}_2$ (taking $G=G^0$ to be an adjoint group, for simplicity). Then $\br(\mf{g}) = \bigoplus_{i} \Sym^{2m_i} \otimes \det^{-m_i}(\bar{r})$, where the sum ranges over the exponents $m_i$ of $G$. If we take $W$ to be any of the factors other than $m_i=1$ (there is always a unique $\Sym^2$ factor in this decomposition), then sequence (\ref{projection}) splits. Indeed, it is given by a class in $H^2(\mr{PGL}_2(\mathbb{F}_p), \br(\fg)/W')$, which splits if and only if the restriction to a $p$-Sylow subgroup $\begin{pmatrix} 1 & \ast \\ 0 & 1 \end{pmatrix}$ splits. Thus we need only check that 
\[
\varphi\left( \begin{pmatrix} 1 & 1 \\ 0 & 1 \end{pmatrix} \right)^p = \exp \left(p \sum_{\alpha \in \Delta} X_{\alpha}\right)
\]
is the identity in $G(\mc{O}/p^2)/W'$. This is clear, since we have quotiented out by the $\Sym^2$ component. (On the other hand, if $W'$ is complementary to the $\Sym^2$ component, then by the same reasoning sequence (\ref{projection}) does not split.)
\end{eg}

\begin{prop}\label{splitcase}
Assumption \ref{big} is in effect. Let $\phi \in H^1(\gal{F, T}, \br(\fgder))$ and $\psi \in H^1(\gal{F, T}, \br(\fgder)^*)$ be non-zero elements whose support is contained in a common $G$-orbit of simple factors of $\fgder$.\footnote{To be precise, the canonical isogeny $G_1 \times G_2 \times \cdots \times G_s \to G^{\mr{der}}$ from the distinct minimal (non-finite) normal subgroup varieties $G_i$ of $G^{\mr{der}}$ induces a decomposition $\fgder \cong \oplus_i \mf{g}_i$, and a corresponding decomposition $(\fgder)^* \cong \oplus_i \mf{g}_i^*$. We assume $\phi$ is supported on some orbit $\Ad(G)\cdot \fg_i$ and $\psi$ is supported on the corresponding $\Ad(G)\cdot \mf{g}_i^*$.} Then 
there exist a trivial prime $q$, a split maximal torus $T$, and a root $\alpha \in \Phi(G^0, T)$ such that
\begin{enumerate}
\item $\rho_2|_{\gal{F_q}}$ belongs to $\Lift_{\br,2 }^{\alpha}(\mc{O}/p^2)$, i.e. it is in $\Lift_{\br}^{\alpha}(\mc{O}/p^2)$ and moreover satisfies
\begin{enumerate}
\item $\rho_2$ is unramified at $q$;
\item $\rho_2(\sigma_q) \in \wh{T}(\mc{O}/p^2)$;
\item for all $\beta \in \Phi^{\alpha}$, $\beta(\rho_2(\sigma_q)) \neq 1 \pmod {p^2}$.
\end{enumerate}
\item $\phi(\sigma_q) \not \in \ker(\alpha|_{\mf{t}}) \oplus \bigoplus_{\beta \in \Phi(G, T)} \mf{g}_{\beta}$.
\item $\langle \psi(\sigma_q), \mf{g}_{\alpha} \rangle \neq 0$.
\end{enumerate}
That is, $q$ is a trivial prime at which extra cocycles as in Lemma \ref{extracocycles} exist, $\phi|_{\gal{F_q}}$ does not belong to $L^{\alpha}_q:= L^{\alpha}_{\br|_{\gal{F_q}}}$, and $\psi|_{\gal{F_q}}$ does not belong to $(L^{\alpha}_q)^\perp$. (See Lemma \ref{extracocycles} for the description of $L^{\alpha}_{\br|_{\gal{F_q}}}$.)
\end{prop}
\proof
To prove the proposition, we may first replace $\br$ and $\rho_2$ with their projections first to the adjoint quotient $G/Z_{G^0}$, and then to the $G$-orbit of simple factors supporting $\phi$ and $\psi$. Thus we may assume $G^0$ is a product of simple adjoint groups that are permuted transitively by $\pi_0(G)$. Fixing one of these factors $G_1 \subset G^0$, we consider the $\gal{\tF}$-equivariant projection $\im(\phi) \subset \fg \onto \fg_1$ (and likewise for $\im(\psi)$). Since trivial primes are split in $\tF$, it will suffice to prove the proposition with the connected adjoint group $G_1$ in place of $G$: namely, we choose $\sigma_q$ with the desired properties for some torus and root $(T_1, \alpha_1)$ of $G_1$, take any extension to a torus $T$ of $G^0$, and just require that $\rho_2(\sigma_q)$ be valued in $T$ with the previously-constructed projection to $T_1$ (for a root $\alpha$ appearing in $G_1$, the set $\Phi^{\alpha}$ is also contained in the roots of $G_1$). Thus in the rest of the argument we may and do assume that $G$ is a connected adjoint group; in doing so, we replace the irreducible submodules $W_i \subset \br(\fgder)$ with their projections to $\fg_1$. Thus we continue to write $\br(\fg)= \oplus_{i \in I} W_i^{\oplus m_i}$ for irreducible distinct $\Fp[\gal{\tF}]$-modules $W_i$.

We first treat the more difficult case in which $K(\rho_2) \cap K_{\phi}$ properly contains $K$. The $\Fp[\gal{\tF}]$-equivariant quotient $\Gal(K_{\phi}/K) \onto \Gal(K(\rho_2) \cap K_{\phi}/K)$ has image isomorphic to a direct sum $\oplus_{i \in I'} W_i$ (for some subset $I' \subset I$) of simple $\Fp[\gal{\tF}]$-submodules of $\im(\phi)$, and by assumption these particular simple factors $W_i$ (see Part (6) of Assumption \ref{big}) appear with multiplicity $m_i=1$ in $\fg$ and satisfy $\End_{\Fp[\gal{\tF}]}(W_i)=k$. We can then choose an $\Fp[\gal{\tF}]$-equivariant embedding $s \colon \Gal(K(\rho_2) \cap K_{\phi}/K) \to \fg$ such that the diagram
\[
\xymatrix{
\Gal(K_{\phi}/K) \ar[r]^-{\phi} \ar[d]_{\mr{res}} & \fg \ar[d]^{\pr_{I'}} \\
\Gal(K_{\phi} \cap K(\rho_2)/K) \ar[r]_-{\sim}^-{s} & \bigoplus_{i \in I'} W_i
}
\]
commutes, where $\pr_{I'}$ denotes the projection onto the sum of the $W_i$-isotypic components for $i \in I'$. This choice of $s$ induces an $\Fp[\gal{\tF, T}]$-linear map
\begin{equation}\label{eta}
\xymatrix{
\fg \ar@/^2pc/[rrr]^{\eta}  \ar[r]_-{\rho_2^{-1}} & \Gal(K(\rho_2)/K) \ar[r]_-{\mr{res}} & \Gal(K(\rho_2) \cap K_\phi/K) \ar[r]_-{s} & \fg.
}
\end{equation}
The image of $\eta$ is the direct sum $\bigoplus_{i \in I'} W_i$. Since $\eta$ is Galois equivariant, and these $W_i$ appear in $\br(\fg)$ with $\Fp[\gal{\tF}]$-multiplicity one, we must have $\eta(W_i)=W_i$ for $i \in I'$ and $\eta(W_i)=0$ for $i \in I \setminus I'$. Moreover, for $i \in I'$, $\eta|_{W_i}$ is equal to a non-zero element $c_i \in \End_{\Fp[\gal{F, T}]}(W_i)= k$, so $\eta$ is in fact $k$-linear. To satisfy the conclusion of the Proposition, the main step is to find a maximal torus $T$ of $G$ and a root $\alpha \in \Phi(G, T)$ such that, letting $p_{\mf{t}} \colon \fg \to \mf{t}$ denote the $k$-linear projection with respect to the decomposition $\fg= \mf{t} \oplus  \left(\bigoplus_{\beta \in \Phi(G, T)} \mf{g}_\beta\right)$, we have $\alpha \circ p_{\mf{t}} \circ \eta(\mf{t}) \neq 0$ and $\langle \fg_{\alpha}, \im(\psi) \rangle \neq 0$.

Fix a split maximal torus $T_1$ of $G$ over the field $k$. For any $g \in G$, let $T_g= gT_1g^{-1}$ be the conjugate maximal torus. Consider the bad loci
\[
 \Phi_1= \{ g \in G: p_{\mf{t}_g} \circ \eta(\mf{t}_g)=0\},
\]
where $p_{\mf{t}_g}$ is the $\mf{t}_g= \Lie(T_g)$-projection with respect to the root space decomposition, and 
\[
 \Phi_2=  \bigcup_{\alpha \in \Phi(G, T_1)} \{g \in G: \langle \Ad(g)\mf{g}_{\alpha}, \im(\psi) \rangle =0\}.
\]
These are both Zariski-closed in $G$. We will show that $\left(G \setminus(\Phi_1 \cup \Phi_2)\right)$ is a non-empty open subset of $G$, and that for $p \gg_G 0$ it must have a $k$-point. Taking this for granted for the moment, we finish the proof. 

If $g \in  (G \setminus \Phi_1 \cup \Phi_2)(k)$, then there is a root $\alpha \in \Phi(G, T_g)$ such that $\alpha \circ p_{\mf{t}_g} \circ \eta \colon \mf{t}_g \to k$ is non-zero, so its kernel is a hyperplane $H= H_{(g, \alpha)} \subset \mf{t}_g$, and $\langle \mf{g}_{\alpha}, \im(\psi) \rangle \neq0$. Fix a non-zero value $c \in \Fp^\times$. Since $c \neq 0$ (as an element of $k$), $H$ cannot contain the locus $\{\alpha= c\} \subset \mf{t}_g$, so $(\mf{t}_g \setminus H) \cap \{\alpha= c\}$ is a non-empty open subset of $\{\alpha= c\}$. So too are the loci $\{\beta \neq 0 \forall \beta \in \Phi^{\alpha}\} \cap \{\alpha= c\}$. As $\{\alpha= c\}$ is connected, we conclude that
\[
 \left\{\text{$t \in \mf{t}_g \setminus H$ such that $\alpha(t)= c$ and $\beta(t) \neq 0$ for all $\beta \in \Phi^{\alpha}$}\right\}
\]
is non-empty (and open in $\{\alpha= c\}$). Choose any $t$ in the $k$-points of this set (that $t$ can be chosen $k$-rationally holds for $p \gg_G 0$ by an argument very similar to the existence argument for the element $g$), and apply the \v{C}ebotarev density theorem (using Lemma \ref{disjoint}) to the extension $K_{\psi} K(\rho_2(\fg)K_{\phi})K(\mu_{p^2})/F$ to find a positive density set of trivial primes $q$ such that $c= \frac{N_{F/\Q}(q)-1}{p}$, $\langle \fg_{\alpha}, \psi(\sigma_q)\rangle \neq 0$, $\rho_2(\sigma_q)= \exp(p \otimes t)$, and the projection of $\phi(\sigma_q)$ to the isotypic components $W_i^{\oplus m_i}$ for $i \not \in I'$ is trivial. 
The required conditions on $\psi(\sigma_q)$ and $\rho_2(\sigma_q)$ 
are then clearly satisfied, and since we have arranged 
\[
\phi(\sigma_q)= \pr_{I'}(\phi(\sigma_q))=s(\sigma_q)= \eta(t), 
\]
the condition on $\phi(\sigma_q)$ also follows from our choice of $t$.

We next check the above claim that the complement $\Phi \setminus (\Phi_1 \cup \Phi_2)(k)$ is non-empty. First we show that both $\Phi_1$ and $\Phi_2$ have non-empty (open) complement in $G$. We now argue over $\bar{k}$. If $\Phi_2$ were equal to $G$, then the same would hold for one of the closed subschemes $\{g \in G: \langle \Ad(g)\fg_{\alpha}, \im(\psi) \rangle=0\}$, and thus for all $x \in G(\bar{k})$, we would have $\Ad(x)\mf{g}_{\alpha} \subset \langle \im(\psi) \rangle^{\perp}$. This implies $\sum_x \Ad(x) \mf{g}_{\alpha} \subset \langle \im(\psi) \rangle^{\perp}$, which evidently contradicts irreducibility of $\mf{g}$ as a $G$-representation. The corresponding statement that $G\setminus \Phi_1$ is non-empty follows from Lemma \ref{larsen} below, noting that each $W_i$ ($i \in I'$) is in fact absolutely irreducible, since we have $\End_{k[\gal{\tF}]}(W_i)=k$.

Now having shown that $\Phi_1$ and $\Phi_2$ are proper closed subschemes of $G$, we show that $G \setminus(\Phi_1 \cup \Phi_2)(k) \neq \emptyset$. There are integers $N$, $r$, and $d$ effectively bounded in terms of the root datum of $G$ such that $G$, $\Phi_1$, and $\Phi_2$ are closed subschemes of an affine space $\mathbb{A}^N$ cut out by at most $r$ equations of degree at most $d$. Indeed, using the faithful representation of $G$ that defines the trace form $B$, this is clear for $G$ and $\Phi_2$, and it holds for $\Phi_1$ since (see the first paragraph of the proof of Lemma \ref{larsen}) $g\in \Phi_1$ if and only if $B(x, \eta(x))=0$ for all $x \in \mf{t}_g$; taking a basis for $\mf{t}_1$, this amounts to finitely many conditions on the matrix coefficients of $g \in G$ of degree twice those needed to define the faithful representation. By the Grothendieck-Lefschetz trace formula, Deligne's work on the Weil conjectures (\cite{deligne:weil2}, and \cite[Corollary of Theorem 1]{katz:bettisums}, we see that there are constants $c(G)$, $c_1(G)$, and $c_2(G)$, depending only on $G$, such that 
\[
|G \setminus (\Phi_1 \cup \Phi_2)(k)| \geq q^{d_G}-c(G) q^{d_G-1}-c_1(G)q^{d_G-1}-c_2(G) q^{d_G-1},
\]
where $d_G= \dim(G)$. In particular, for $p \gg_G 0$, this complement is non-empty, and the proof of the hard case of the Proposition is complete. (Note that if we bounded in terms of $G$ the number of irreducible components of $\Phi_1$ and $\Phi_2$, we could use a much more elementary argument here, \textit{\`{a} la} Lang-Weil.)

Finally, if $K(\rho_2)$ and $K_{\phi}$ are linearly disjoint over $K$ a much simpler argument suffices, since we can replace the bad locus $\Phi_1$ with $\Phi_1'=\{(g \in G: p_{\mf{t}_g}(\im(\phi))=0\}$; the multiplicities no longer intervene in the argument. We omit the details.
\endproof
We are very grateful to Michael Larsen for explaining the proof of the following lemma, which in turn completes the proof of Proposition \ref{splitcase}:
\begin{lemma}\label{larsen}
Let  $\fg$ be the Lie algebra of a semisimple group $G$ over an algebraically closed field $k$ of characteristic $p \neq 2$; assume $p$ is large enough (relative to the root datum of $G$) that $\fg$ carries a non-degenerate trace form $B$. Suppose there are given $k$-vector space decompositions $\fg= W \oplus W'$ and $W= \oplus_{i \in I} W_i$, and suppose there is a subgroup $\Gamma \subset \mr{GL}(\fg)$ preserving $B$, stabilizing each $W_i$, and separating them in the following sense: for all $i \neq j$, there exists a subgroup $\Gamma_{i, j} \subset \Gamma$ such that $\Gamma_{i, j}$ acts irreducibly and non-trivially on $W_j$ and trivially on $W_i$. Finally, let $\eta \colon \fg \to \fg$ be a $k[\Gamma]$-linear map such that 
\begin{itemize}
\item $\eta(W')=0$.
\item For each $i$, $\eta_{W_i}$ is multiplication by a non-zero scalar $c_i \in k^\times$.
\end{itemize}
Then there is a Cartan sub-algebra $\mf{t}$ of $\fg$ such that $p_{\mf{t}}\circ \eta(\mf{t})$ is non-zero, where $p_{\mf{t}} \colon \fg \to \mf{t}$ denotes the $\mf{t}$-projection whose kernel is the sum of all root spaces with respect to $\mf{t}$.
\end{lemma}
\proof
By our assumption on $p$, there is an invariant trace form $B \colon \fg \times \fg \to k$. For all $x, y, z \in \fg$, we then have $B(x, [y, z])=B([x,y], z])$. In particular, if $x$ and $y$ lie in the same Cartan sub-algebra of $\fg$ (and therefore commute), then $B(x, [y, z])=0$ for all $z$. For any Cartan sub-algebra $\mf{t}$, the annihilator of $\mf{t}$ with respect to $B$ is the sum of root spaces in $\fg$ (by the above identity and the fact that each root $\alpha$ is non-vanishing on $\mf{t}$), so $\ker(p_{\mf{t}})$ is the span of the set $\{[y, z]: y \in \mf{t}, z \in \fgder\}$. Thus, for $x \in \mf{t}$, $B(x, \ker(p_{\mf{t}}))=0$.

Assume that $p_{\mf{t}} \circ \eta(\mf{t})=0$ for all Cartan sub-algebras $\mf{t}$ of $\fg$. Any semisimple element $x$ lies in some Cartan $\mf{t}$, and we have assumed $\eta(x) \in \ker(p_{\mf{t}})$, so $B(x, \eta(x))=0$. Since the semisimple elements are dense in $\fg$, we deduce that $B(x, \eta(x))=0$ identically on $\fg$. Applying this observation to any $x=w \in W$ and $x=w+w' \in W \oplus W'$, we see $B(w', \eta(w))=0$ for all $w' \in W'$, $w \in W$, i.e. $B(W', W)=0$ (clearly $\eta(W)=W$). We next apply the identity to any triple of elements $w_1, w_2, w_1+w_2 \in W$ and find that $B(w_1, \eta(w_2))+B(w_2, \eta(w_1))=0$. Taking $w_1 \in W_{i_1}$, $w_2 \in W_{i_2}$ (for any pair of indices $i_1, i_2$), we find $(c_{i_1}+c_{i_2})B(w_1, w_2)=0$. If $c_{i_1}+c_{i_2}$ is non-zero (and in particular if $i_1=i_2$), then we find $B(W_{i_1}, W_{i_2})=0$. If $c_{i_1}=-c_{i_2}$, then for all $\gamma \in \Gamma_{i_1, i_2}$, we rewrite the above identity as
\begin{align*}
0&=B(w_1, \eta(w_2))+B(w_2, \eta(w_1))= B(\gamma w_1, \gamma \eta(w_2))+B(w_2, \eta(w_1))\\&=B(w_1, c_{i_2}\gamma w_2)+B(w_1, c_{i_1}w_2) = c_{i_1}B(w_1, (1-\gamma)w_2).
\end{align*}
Now the hypothesis that $\Gamma_{i_1, i_2}$ acts irreducibly and non-trivially on $W_{i_2}$ implies that $B(W_{i_1}, W_{i_2})=0$ in this case as well.

Assembling these observations, we see that $B(W, W')=0$ and $B(W, W)=0$, contradicting non-degeneracy of $B$.
\endproof
\begin{rmk}
Our proof of Proposition \ref{splitcase} relies on the map $\eta$ of Equation (\ref{eta}) resembling (roughly speaking) a \textit{projection} onto $\im(\phi)$. This causes serious problems when $\br(\fg)$ has constituents with $\Fp[\gal{F}]$-multiplicity greater than one (for simplicity, in this remark take $G=G^0$ to be semisimple). Indeed, suppose that $V \subset \br(\fgder)$ is an isotypic component, isomorphic to $W^{\oplus{m}}$ for some irreducible $\Fp[\gal{F}]$-module $W$ and $m>1$. If there is a cocycle $\phi$ supported on $V$ and having $\im(\phi) \cong W$, then the kernel $W'$ of $\mf{g} \xrightarrow{\rho_2^{-1}} \Gal(K(\rho_2)/K) \to \Gal(K_{\phi}/K)$ has a non-trivial $W$-isotypic component. It will support other cocycles $\phi'$ that arise from the same class in $H^1(\gal{F, T}, W)$ as $\phi$ (but with $W$ embedded into $\br(\fg)$ in two different ways to yield $\phi$ and $\phi'$), and in particular $K_{\phi'}=K_{\phi}$. But then $\eta$ certainly vanishes on $\im(\phi')$; thus our argument would not allow us to kill the cohomology class $\phi'$.

We further remark that the Galois-theoretic control over the map $\eta$, which allows us to apply Lemma \ref{larsen}, is crucial to the argument of Proposition \ref{splitcase}: for instance, there are non-trivial linear maps $\eta \colon \mf{sl}_2 \to \mf{sl}_2$ with the property that $B(\mf{t}, \eta(\mf{t}))=0$ for every Cartan subalgebra $\mf{t}$ of $\mf{sl}_2$. 
\end{rmk}

\section{Some group theory: irreducible $G(k)$-representations for $p \gg_G 0$}\label{groupsection}
Before proving our main theorem, we will prove a few group-theoretic lemmas showing that the image hypotheses, with the exception of the multiplicity-free condition of Assumption \ref{big}, of \S \ref{klrsection} and \S \ref{auxiliarysection} in fact follow from the seemingly simpler assumption that $\br$ is ``absolutely irreducible,'' as long as $p$ is sufficiently large. We note that the explicit bounds extracted here depend on the classification of finite simple groups. Recall that a subgroup $\Gamma \subset G^0(k)$ is absolutely irreducible if $\Gamma$ is not contained in any proper parabolic subgroup of $G^0_{\overline{k}}$.
\begin{lemma}\label{irr}
 Let $\Gamma \subset G^0(k)$ be an absolutely irreducible finite subgroup. Assume $p > 2(\dim_k(\fgder) +1)$. Then:
\begin{enumerate}
 \item $\fgder$ is a semisimple $k[\Gamma]$-module.
 \item $H^1(\Gamma, \fgder)=0$, and the same holds if the action of $\Gamma$ on $\fgder$ is twisted by a character of $\Gamma$.
\end{enumerate}
\end{lemma}
\begin{proof}
 Let $h_G$ be the maximum of the Coxeter numbers of the simple factors of $G^0$. By \cite[Corollaire 5.5]{serre:CR}, for $p> 2h_G-2$, $\mf{g}$, and hence its summand $\fgder$, is a semisimple $\Gamma$-module. We claim then that $\Gamma$ contains no non-trivial normal subgroup of $p$-power order. Indeed, suppose there were such a subgroup $H \unlhd \Gamma$. Consider any irreducible $\overline{k}[\Gamma]$-summand $U$ of $\fg_{\overline{k}}$. The $\overline{k}$-vector space of invariants $U^H$ is non-trivial (since $H$ is a $p$-group) and is stabilized by $\Gamma$, hence must equal all of $U$. This holds for all $U$, so $\fg$ is a trivial $H$-module, and therefore $H$ is contained in the center $Z_{G^0}(k)$; but the latter clearly has order prime to $p$, a contradiction. Thus $\Gamma$ has no non-trivial normal subgroup of order $p$, and by \cite[Theorem A]{guralnick:CR}, $H^1(\Gamma, \fgder)=0$ for $p>2 (\dim_k(\fgder) +1)$ (to be precise, apply this result to $\Gamma/\Gamma \cap Z_{G^0}(k)$ acting on $\fgder$). 
\end{proof}
The following lemma, with a different proof, also appears in \cite[Lemma 5.1]{bhkt:fnfieldpotaut}:
\begin{lemma}\label{invariants}
Let $G$ be a connected reductive group over $\bar{k}$. Assume $p >5$, and that $p \nmid n+1$ for any simple factor of $G^{\mr{ad}}$ of Dynkin type $A_n$. Let $\Gamma \subset G(\bar{k})$ be absolutely irreducible. Then $H^0(\Gamma, \fgder)=0$.
\end{lemma}
\begin{proof}

By our characteristic assumptions (which imply that $G^{\mr{der}}$ and
$G^{\mr{ad}}$ have isomorphic Lie algebras), we may and do assume
$G=G^0$ is an adjoint group, and by considering each simple factor of
$G^0$ we may and do further assume that $G$ is simple. Let $X$ be an
element of $\fg^\Gamma$. We have the Jordan decomposition $X=X_s+X_n$
into semisimple and nilpotent parts in $\fg$, and uniqueness of
Jordan decomposition implies that both $X_s$ and $X_n$ are
$\Gamma$-invariant. Since $\Gamma$ is then contained in the
intersection $C_G(X_s) \cap C_G(X_n)$, it suffices to show that
$C_G(X)$ is contained in a proper parabolic when $X$ is either
semisimple or nilpotent. In either case, as long as $p>5$ (for $G$
not of type $A_n$) or $p \nmid n+1$ (for $G$ of type $A_n$), $C_G(X)$
is smooth (by a theorem of Richardson: see \cite[2.5
Theorem]{jantzen:nilporbits}). Assume $X$ is a non-zero
nilpotent. Then \cite[5.9 Proposition]{jantzen:nilporbits} implies
that $C_G(X)$ is contained in a proper parabolic subgroup. Now assume
$X$ is a non-zero semisimple element. There is a maximal torus $T$ of
$G$ such that $X$ belongs to $\mf{t}= \Lie(T)$
(\cite[11.8]{borel:linalg}). As usual, we can diagonalize the
$T$-action on $\fg$ to obtain a root system (in the real vector space
$X^\bullet(T) \otimes_{\Z} \RR$). The subgroup $C_G(X)$ is a connected
reductive group containing $T$: for the connectedness, we use that
$p>5$ (ensuring $p$ is not a ``torsion prime'') so that we can invoke \cite[Theorem 3.14]{steinberg:torsion}. By \cite[3.4 Proposition]{borel-tits:reductive}, $C_G(X)$ is determined by the root subgroups it contains (since it contains a maximal torus of $G$). For a root $\alpha \in \Phi(G, T)$, let $u_{\alpha} \colon \mathbf{G}_a \to G$ be the corresponding root subgroup. For $t \in T$, the relation 
\[
u_{\alpha}(y)t u_{\alpha}(y)^{-1}= t\cdot u_{\alpha}((\alpha(t)^{-1}-1)y)
\] 
lets us compute that (passing to the Lie algebra) $C_G(X)$ precisely contains those $U_{\alpha}= \im(u_{\alpha})$ drawn from the subset
\[
\Phi'= \{\alpha \in \Phi(G, T): d\alpha(X)=0\}
\]
of $\Phi= \Phi(G, T)$. We claim that the semisimple rank of $C_G(X)$ is strictly less than that of $G$. Temporarily granting this, we have that the roots $\Phi'$ span a proper subspace $\RR \Phi' \subset \RR \Phi=X^\bullet(T)_{\RR}$. By \cite[VI.1.7 Proposition 23]{bourbaki:lie456}, $\Phi'$ is a root system in the real vector space $\RR \Phi'$, and we can also consider it as a subsystem of the root system $\Phi''= \RR \Phi' \cap \Phi$. The latter, by \cite[VI.1.7 Proposition 24]{bourbaki:lie456} has a basis $I$ that extends to a basis of $\Phi$; and since $\RR \Phi'$ is strictly contained in $\RR \Phi$ this basis of $\Phi''$ is a proper subset of the extended basis of $\Phi$. It follows that $C_G(X)$ is contained in the (proper) Levi subgroup of $G$ associated to $I$, and therefore that $\Gamma$ is reducible.

To complete the proof, we establish the postponed claim that the inclusion $\RR \Phi' \subseteq \RR \Phi$ is proper. It suffices to show that $C_G(X)$ is not semisimple, i.e. has positive-dimensional center. Suppose it were semisimple. Its root system is a (not necessarily simple) subsystem of that of $G$, and so there are only finitely many possibilities for the root systems of the simple factors $H$ of $C_G(X)^{\mr{ad}}$. Under our assumptions on $p$, each of these simple factors satisfies the following two properties:
\begin{itemize}
\item $H^{\mr{sc}} \to H^{\mr{ad}}$ induces an isomorphism on Lie algebras.
\item $\Lie(H)$ has trivial center.
\end{itemize}
Indeed, note that $\Lie(H)$ has non-trivial center only when $p \leq 3$ or $H$ is of type $A_n$ and $p \vert n+1$: see the discussion of \cite[pp. 47-48]{seligman:modular} (which ensures that $\Lie(H)$ has a nonsingular trace form), and then apply \cite[Theorem I.7.2]{seligman:modular}. Thus under our assumptions on $p$, $\Lie(C_G(X))= C_{\fg}(X)$ must have trivial center. But $X$ visibly lies in the center, and we have therefore contradicted the supposed semi-simplicity of $C_G(X)$.


\end{proof}

In the main theorem, we will use the next three lemmas (Lemma \ref{cyclicq}, specifically) to show that $\br(\fgder)$ and $\br(\fgder)^*$ have no common subquotient.

\begin{lemma} \label{larsenpink} Given integers $n, c_1 >0$, there
  exists an integer $c_2 > 0$ (depending only on $n$ and $c_1$) such
  that if $\Gamma \subset \mr{GL}_n(k)$ is a finite subgroup admitting
  a cyclic quotient of order $c_2$ and not containing any normal
  subgroup of order $p^a$ with $a>0$, then the centre of $\Gamma$
  contains a cyclic subgroup of order prime to $p$ and $\geq c_1$.
\end{lemma}

\begin{proof}
  By Theorem 0.2 of \cite{larsen-pink:finite}, for any finite subgroup
  $\Gamma \subset \mr{GL}_n(k)$ there exist normal subgroups
  $\Gamma_3 \subset \Gamma_2 \subset \Gamma_1 \subset \Gamma$ such that
  $\Gamma_3$ is a $p$-group, $\Gamma_2/\Gamma_3$ is an abelian group
  of order prime to $p$, $\Gamma_1/\Gamma_2$ is a product of finite
  simple groups of Lie type and $\Gamma/\Gamma_1$ has order bounded
  by a constant depending only on $n$. Our assumptions imply that
  $\Gamma_3$ is trivial. From the proof of the theorem
  \cite[p. 1156]{larsen-pink:finite} this imples that $\Gamma_2$ is in the
  centre of $\Gamma_1$, so the conjugation action of $\Gamma$ on
  $\Gamma_2$ factors through $\Gamma/\Gamma_1$.

  Let $\Gamma' = (\Gamma_1)^{\mr{der}}$. Clearly $\Gamma'$ lies in the
  kernel of any homomorphism from $\Gamma$ to an abelian group and
  $\Gamma_2$ surjects onto $\Gamma_1/\Gamma'$. Furthermore,
  $\Gamma' \cap \Gamma_2$ has order bounded by a constant depending
  only on $n$: this again follows from the construction of $\Gamma_1$ and $\Gamma_2$
  in \cite[p. 1156]{larsen-pink:finite} (note particularly the construction of the group denoted $G_2$ in loc. cit.). 
  Since the order of $\Gamma/\Gamma_1$ is bounded, if
  $\Gamma$ has a large cyclic quotient, the coinvariants of the action
  of $\Gamma/\Gamma_1$ on $\Gamma_1/\Gamma'$ must also have a large
  cyclic quotient, and so also a large cyclic subgroup. The lemma
  follows since if $A$ is any abelian group with an action of a finite
  group $\Delta$, the kernel of the averaging map from $A_{\Delta}$ to $A^{\Delta}$ is
  killed by the order of $\Delta$.
\end{proof}
\begin{rmk}\label{explicit}
The constant $c_2$ can be effectively bounded by invoking an explicit bound on the index $[\Gamma:\Gamma_1]$ obtained by Collins (\cite{collins:modularjordan}) using (unlike \cite{larsen-pink:finite}) the classification of finite simple groups.
\end{rmk}

\begin{lemma} \label{cent} For $G$ any (split) connected reductive
  group over $k$ there exists a constant $n_G$, depending only on the
  root datum of $G$, such that for any semisimple element
  $s \in G(k)$ the centralizer of $s^n$ in $G$ is a (not necessarily
  proper) Levi subgroup of $G$ for some $n$ dividing $n_G$.
\end{lemma}

\begin{proof}
  Let $T$ be a maximal torus of $G$ containing $s$ and let $t$ be any
  element of $T(\bar{k})$. By the theorem in \S 2.2 of
  \cite{humphreys:conjugacy}, $C_G(t)$ is generated by $T$, the root
  subgroups $U_{\alpha}$ for which $\alpha(t) = 1$ and representatives
  (in $N(T)$) of the subgroup $W(t)$ of the Weyl group $W(G,T)$ fixing
  $t$.  Let $\Phi(t)$ be the subset of $\Phi(G,T)$ consisting of all
  roots which are trivial on $t$. Let $T^{W(t)}$ be the subgroup of
  $T$ fixed pointwise by $W(t)$ and let
  $T^{\Phi(t)} = \cap_{\alpha \in \Phi(t)} \; \mr{Ker}(\alpha)$.  Let
  $n_G'$ be the lcm of the orders of the torsion subgroups of all the
  character groups of the groups of multiplicative type
  $T^{W(t)} \cap T^{\Phi(t)}$ for all $t \in T(\bar{k})$; there are
  only finitely many distinct such subgroups since both $W(G,T)$ and
  $\Phi(G,T)$ are finite sets. Then $n_G'$ depends only on the root
  datum of $G$, and the order of the component group of any subgroup
  $T^{W(t)} \cap T^{\Phi(t)}$ divides $n_G'$.

  It follows that $s_1 := s^{n_G'}$ is contained in a torus $T_1$ such
  that $T_1 \subset T^{W(s)} \cap T^{\Phi(s)}$. We clearly have
  $W(s) \subset W(s_1)$ and $\Phi(s) \subset \Phi(s_1)$. If both inclusions
  are equalities then $C_G(s)$ equals $C_G(s_1)$. Since
  $C_G(s_1) \supset C_G(T_1) \supset C_G(s)$ by construction, it would follow that
  $C_G(s)$ is equal to the centralizer of a torus, hence (by \cite[4.15 Th\'{e}or\`{e}me]{borel-tits:reductive}) a Levi
  subgroup. If either of the inclusions is strict, we repeat the
  procedure after replacing $s$ by $s_1$. Since $W(G,T)$ and
  $\Phi(G,T)$ are both finite, after at most
  $m_G := |W(G,T)| + |\Phi(G,T)|$ steps we must have equality. Thus,
  we may take $n_G$ to be $(n'_G)^{m_G}$.
  \end{proof}

  \begin{lemma} \label{cyclicq}
    For $G$ any split semisimple group over $k$ there exists a
    constant $a_G$ depending only on the root datum of $G$ such that
    if $\Gamma \subset G(k)$ is an absolutely irreducible subgroup
    then $\Gamma$ has no cyclic quotient of order $\geq a_G$.
  \end{lemma}

  \begin{proof}
    We may clearly assume that $G$ is of adjoint type. If $\Gamma$
    contains a nontrivial normal subgroup $U$ of order a power of $p$
    then $U$ is inside a $p$-Sylow of $G(k)$, i.e., the unipotent
    radical of a Borel. By a theorem of Borel-Tits
    \cite[3.1 Proposition]{borel-tits:unipotent}, there is a parabolic
    $P \subset G$ containing $N_G(U)$ whose unipotent radical contains
    $U$. Since $G$ is reductive, $P$ is a proper parabolic if $U$ is
    nontrivial.  Since $U$ is normal in $\Gamma$, this implies
    $\Gamma$ is in a proper parabolic of $G$, contradicting
    irreducibility.

    By embedding $G$ in $\mr{GL}_n$ for some $n$, we may now apply
    Lemma \ref{larsenpink} with $c_1-1 $ equal to the number $n_G$
    obtained from Lemma \ref{cent}, to get $c_2$ such that if $\Gamma$
    has a cyclic quotient of order $\geq c_2$ then the centre of
    $\Gamma$ contains a cyclic subgroup $Z$ of order at least $c_1$
    and of order prime to $p$. By Lemma \ref{cent} there exists an
    integer $n < c_1$ so that $C_G(s^n)$ is a Levi subgroup, where $s$
    is any generator of $Z$. By construction, $s^n$ is not the
    identity and since $G$ is adjoint, it is also not central, so
    $C_G(s^n)$ is a proper Levi subgroup of $G$. But
    $\Gamma \subset C_G(s^n)$ and this contradicts irreducibility once
    again.
  \end{proof}

\section{Completion of the argument}\label{finallift} 
In this section we combine the results of the previous sections with a standard Galois-cohomological argument (originating in \cite{ramakrishna:lifting} and \cite{taylor:icos2}) to prove our main theorem. We must first, however, take a small technical digression, only needed when $\mf{z}_{\mu} \neq 0$, to complement the results of \S \ref{auxiliarysection} and to explain a way to kill Selmer and dual Selmer classes supported on $\br(\mf{z}_{\mu})$ and $\br(\mf{z}_\mu)^*$. Recall that Proposition \ref{splitcase} does not apply to such classes, and indeed such classes cannot be killed with trivial primes. Following the template of \cite[Corollary 2.6.4]{clozel-harris-taylor}, however, we will explain how a hypothesis on the class group $\mr{Cl}(\tF)$ of $\tF$ and some knowledge of the local deformation conditions can rule out the existence of Selmer or dual Selmer classes supported on $\mf{z}_{\mu}$.

Of course, for the primes $S$ of ramification of our residual representation $\br$, we have not defined explicit local deformation conditions: that some such good conditions can be defined will be one of the hypotheses of our main theorem. The argument in \cite{clozel-harris-taylor} crucially depends on knowing that classes in $L_v \cap H^1(\gal{F_v}, \br(\mf{z}_{\mu}))$ (where $L_v$ is the tangent space of the local condition) are \textit{unramified}, and they deduce this from the explicit description of their sets $L_v$. What we will do instead is postulate the existence of good local conditions in the adjoint case, and then show these can be suitably lifted to good local conditions in the general case that will have this ``unramified on $\mf{z}_{\mu}$" property.
\begin{lemma}\label{adlocalcondition}
Let $v \nmid p$ be a finite place, let $\br_v \colon \gal{F_v} \to G(k)$ be a residual representation, and let $\br^{\mr{ad}}_v$ be its image under $G \to G/Z^0_{G^0}$. Assume that there exists a representable, $\wh{(G/Z^0_{G^0})}$-stable, subfunctor $\Lift_{\br^{\mr{ad}}_v}^{\mc{P}_v} \subset \Lift_{\br^{\mr{ad}}_v}$ of the lifting functor for $\br_v^{\mr{ad}}$ such that
\begin{itemize}
\item $\Lift_{\br^{\mr{ad}}_v}^{\mc{P}_v}$ is formally smooth.
\item $L^{\mr{ad}}_v= \Def_{\br^{\mr{ad}}_v}^{\mc{P}_v}(k[\epsilon])$ has dimension $h^0(\gal{F_v}, \br(\fg^{\mr{der}}))$.
\end{itemize} 
Choose a multiplier character $\nu \colon \gal{F_v} \to A(\mc{O})$ lifting $\mu \circ \br_v$ such that $\nu(I_{F_v})$ has order coprime to $p$.\footnote{This is a simplifying hypothesis that can probably be removed. It is always possible globally.} Then there exists a representable, $\wh{G}$-stable, subfunctor $\Lift^{\mc{P}_v}_{\br_v} \subset \Lift^\nu_{\br_v}$ that is formally smooth, has tangent space $L_v$ of dimension $h^0(\gal{F_v}, \br(\fgm))$, and satisfies $L_v \cap H^1(\gal{F_v}, \br(\mf{z}_{\mu}))= H^1_{\mr{unr}}(\gal{F_v}, \br(\mf{z}_{\mu}))$.
\end{lemma}
\proof 
For an object $R$ of $\mc{C}_{\mc{O}}$, we define $\Lift^{\mc{P}_v}_{\br_v}(R)$ to be the set of lifts $\rho \colon \gal{F_v} \to G(R)$ such that 
\begin{itemize}
\item $\mu \circ \rho= \nu$.
\item The projection of $\rho$ to $G/Z^0_{G^0}(R)$ belongs to $\Lift_{\br^{\mr{ad}}_v}^{\mc{P}_v}(R)$
\item Let $\Theta$ be the image of $\rho(I_{F_v})$ in $G/G^{\mr{der}}(R)$. Then $\Theta$ is isomorphic to its image under the map $G/G^{\mr{der}}(R) \to G/G_{\mu}(R) \times \pi_0(G)$.
\end{itemize}
We now check formal smoothness. Let $R \to R/I$ be a small extension in $\mc{C}_{\mc{O}}^f$, and let $\rho \in \Lift^{\mc{P}_v}_{\br_v}(R/I)$. Write $\rho^{\mr{ad}}$ for its image in $\Lift^{\mc{P}_v}_{\br_v^{\mr{ad}}}(R/I)$. By hypothesis, $\rho^{\mr{ad}}$ lifts to an element $\tilde{\rho}^{\mr{ad}} \in \Lift^{\mc{P}_v}_{\br_v^{\mr{ad}}}(R)$. The obstruction to lifting $\rho$ to an element of $\Lift^{\nu}_{\br_v}(R)$ is an element $\mr{obs}_{\rho} \in H^2(\gal{F_v}, \br(\fgm))\otimes_k I$, and we see that the image of $\mr{obs}_{\rho}$ in $H^2(\gal{F_v}, \br(\fgm/\mf{z}_{\mu}))\otimes_k I$ is zero. Furthermore, the push-forward of $\rho$ to $\gal{F_v} \to G(R/I) \to G/G^{\mr{der}}(R/I)$ factors through $\Gal(L/F_v)$, where $L$ is a finite extension of $F_v^{\mr{ur}}$ of order prime to $p$. A standard argument with the Hochschild-Serre spectral sequence (eg, \cite[Lemma 4.17]{stp:exceptional}) shows that the obstruction to lifting the homomorphism $\Gal(L/F_v) \to G/G^{\mr{der}}(R/I)$ to $G/G^{\mr{der}}(R)$ vanishes, and thus the image of $\mr{obs}_{\rho}$ in $H^2(\gal{F_v}, \br(\fg/\fg^{\mr{der}}))\otimes_k I$ also vanishes. We conclude that $\mr{obs}_{\rho}$ itself is trivial and obtain a lift $\tilde{\rho} \in \Lift_{\br_v}^\nu(R)$. This lift automatically satisfies the first bulleted point above, and we can modify it by an element of $H^1(\gal{F_v}, \br(\mf{z}_{\mu}))$ to satisfy the third bulleted point, since we just saw that there is a lift along $G/G^{\mr{der}}(R) \to G/G^{\mr{der}}(R/I)$ that allows no additional ramification. Moreover, we can modify $\tilde{\rho}$ by any element of $H^1(\gal{F_v}, \br(\fg^{\mr{der}}))\otimes_k I$ without altering either the first or third bulleted properties. In particular, we can alter it by a cocycle $\phi$ such that $\exp(\phi)\tilde{\rho}$ has image mod $Z^0_{G^0}$ lying in $\Lift_{\br^{\mr{ad}}_v}^{\mc{P}_v}(R)$, since the set of lifts of $\rho^{\mr{ad}}$ to $G/Z^0_{G^0}(R)$ is a torsor under $H^1(\gal{F_v}, \br(\fg^{\mr{der}}))$.

Finally, by construction the tangent space of $\Lift^{\mc{P}_v}_{\br_v}$ is isomorphic to the direct sum of $\Def^{\mc{P}_v}_{\br_v^{\mr{ad}}}(k[\epsilon]) \subset H^1(\gal{F_v}, \br(\mf{g}^{\mr{der}}))$ and $H^1_{\mr{unr}}(\gal{F_v}, \br(\mf{z}_{\mu}))$. The last assertions of the lemma follow from the assumption that $\dim L_v^{\mr{ad}}= h^0(\gal{F_v}, \br(\fg^{\mr{der}}))$.
\endproof

\begin{lemma}\label{classgroup}
Let $\br \colon \gal{F, S} \to G(k)$ be a residual representation, for any $G$ as in \S \ref{defprelimsection}. Assume that
\begin{itemize}
\item $\{L_v\}_{v \in S}$ is a collection of subspaces of the $H^1(\gal{F_v}, \br(\fg_{\mu}))$ for $v \in S\setminus \{v \vert p\}$ such that $L_v \cap H^1(\gal{F_v}, \br(\mf{z}_{\mu}))=H^1_{\mr{unr}}(\gal{F_v}, \br(\mf{z}_{\mu}))$.
\item $T$ is a finite set of primes containing $S$ such that $T \setminus S$ consists of trivial primes, and at which we define local subspaces $L_v$ as in either Lemma \ref{extracocycles} or Lemma \ref{ramextracocycles} (i.e. $L_v$ will be one of the subspaces $L^{\alpha}_{\br}$ of those lemmas).
\item At the primes $v \vert p$, we define $L_v$ to be either as in Lemma \ref{regord} or Lemma \ref{ordextracocycles}. Alternatively, if $G_{\mu}$ is a product of copies of general linear groups and $\br|_{\gal{F_v}}$ is Fontaine-Laffaille, we take $L_v$ to be the space of Fontaine-Laffaille infinitesimal deformations. 
\item $\Hom(\mr{Cl}(\tF)/p\mr{Cl}(\tF), \br(\mf{z}_{\mu}))^{\Gal(\tF/F)}=0$.
\end{itemize}
Set $\mc{L}=\{L_v\}_{v \in T}$. Then $H^1_{\mc{L}}(\gal{F, T}, \br(\mf{z}_{\mu}))=0$ and $H^1_{\mc{L}^\perp}(\gal{F, T}, \br(\mf{z}_{\mu})^*)=0$.
\end{lemma}
\proof
We first note that the subspaces $L_v$ that we have defined for $v \in T \setminus S \cup \{v \vert p\}$ all satisfy $L_v \cap H^1(\gal{F_v}, \br(\mf{z}_{\mu}))= H^1_{\mr{unr}}(\gal{F_v}, \br(\mf{z}_{\mu}))$. This is a case-by-case check from the explicit descriptions of these spaces. Moreover note that the analogous inequality then holds for the $L_v^\perp$ (at all places $v \in T$), since the unramified cohomology on the dual side is the annihilator of the unramified cohomology under local duality.

Thus, 
\begin{align*}
H^1_{\mc{L}}(\gal{F, T}, \br(\mf{z}_{\mu})) &= \ker(H^1(\gal{F}, \br(\mf{z}_{\mu}))\to \bigoplus_{v \nmid \infty} H^1(I_{F_v}, \br(\mf{z}))) \\
&
\xrightarrow[\mr{res}]{\sim} \left(\ker(H^1(\gal{\tF}, \br(\mf{z}_{\mu})) \to \bigoplus_{\tv \nmid \infty} H^1(\gal{\tF_{\tv}}, \br(\mf{z}_{\mu}))\right)^{\Gal(\tF/F)} \\
&= \Hom(\mr{Cl}(\tF)/p\mr{Cl}(\tF), \br(\mf{z}_{\mu}))^{\Gal(\tF/F)}=0
\end{align*}
(recall that $[\tF:F]$ is prime to $p$).
\endproof 
Now we come to the main theorem:
\begin{thm}\label{mainthm}
Let $F$ be a totally real field, and let $\br \colon \gal{F, S} \to G(k)$ be a continuous representation unramified outside a finite set $S$ of finite places containing the places above $p$. Let $\tF$ be the smallest extension of $F$ such that $\br(\gal{\tF})$ is contained in $G^0(k)$. Assume that $p \gg_{G, F} 0$ and that $\br$ satisfies the following:
\begin{itemize}
 \item $\br$ is odd, i.e. for all infinite places $v$ of $F$, $h^0(\gal{F_v}, \br(\fgder))= \dim(\mr{Flag}_{G^{\mr{der}}})$.
 \item $\br|_{\gal{\tF(\zeta_p)}}$ is absolutely irreducible.
\item Let $W$ be any simple $\Fp[\gal{\tF}]$-quotient $\br(\fgder)/W' \xrightarrow{\sim} W$. If the group extension
 \[
 1 \to \fg/W' \to \im(\rho_2)/W' \to \im(\br) \to 1
 \]
 splits, then we assume:
 \begin{enumerate}
 \item For each $G$-orbit of simple factors $\br(\oplus_i \fg_i)$ of $\br(\fgder)$, $W$ appears in $\br(\fg_i)$ with $\Fp[\gal{\tF}]$-multiplicity at most 1. 
 \item Moreover, each such $W$ satisfies $\End_{\Fp[\gal{\tF}]}(W_i) \cong k$ (see Lemma \ref{multfreelemma} below).
 \end{enumerate}
 \item For all $v \vert p$, $F_v$ does not contain $\zeta_p$, and $\br|_{\gal{F_v}}$ either 
 \begin{itemize}
 \item is trivial; or
 \item is ordinary in the sense of \S \ref{ordsection} and satisfies the conditions (REG) and (REG*); or
 \item $G^0$ is a product of groups of the form $\mr{GL}_N$, $\mr{GSp}_N$, and $\mr{GO}_N$, and the projection of $\br|_{\gal{F_v}}$ to each factor is Fontaine-Laffaille with distinct Hodge-Tate weights in an interval of length less than $p-1$ in the $\mr{GL}$ case and less than $\frac{p-1}{2}$ in the $\mr{GSp}$ and $\mr{GO}$ cases. 
 \end{itemize}
 \item The field $K= \tF(\br(\fgder), \mu_p)$ does not contain $\mu_{p^2}$ (which follows in many cases from the preceding condition at $v \vert p$: see Remark \ref{mup2}).
 \item For all $v \in S$ not above $p$, there is a formally smooth local deformation condition $\mc{P}_v$ for $\br|_{\gal{F_v}}$ whose tangent space $\Tan^{\mc{P}_v}_{\br|_{\gal{F_v}}} \subset H^1(\gal{F_v}, \br(\mf{g}_\mu))$ has dimension $h^0(\gal{F_v}, \br(\fgm))$. If $G$ is not connected, \emph{and} the center of $G^0$ is positive-dimensional, we instead impose the hypothesis of Lemma \ref{adlocalcondition} above, and further assume that the subspaces $L_v^{\mr{ad}}$ of Lemma \ref{adlocalcondition} satisfy $L_v^{\mr{ad}}= \bigoplus_{i=1}^s \left(L_v^{\mr{ad}} \cap H^1(\gal{F_v}, \br(\mf{y}_i)) \right)$, where $\mf{y}_1, \ldots, \mf{y}_s$ are the distinct $G$-orbits of simple factors of $\fgder$.\footnote{This may be automatic from the conditions on $\mc{P}_v$, but we have not checked this.}
\end{itemize}

Then there exist a finite set of places $T \supset S$ and a geometric lift $\rho$ of $\br$
 \[
 \xymatrix{
 & G(\mc{O}) \ar[d] \\
 \gal{F, T} \ar@{-->}[ur]^\rho \ar[r]_{\br} & G(k).
 }
 \]
 such that $\im(\rho)$ contains $\widehat{G^{\mr{der}}}(\mc{O})$, and in particular the Zariski closure of $\im(\rho)$ contains $G^{\mr{der}}$.
\end{thm}
\begin{rmk}\label{mup2}
As observed in Lemma \ref{p2} and Lemma \ref{FLp2}, the condition that $K$ not contain $\mu_{p^2}$ follows automatically once there is a single place $v \vert p$ such that $F_v/\Q_p$ is unramified and either $\br|_{\gal{F_v}}$ is trivial, Fontaine-Laffaille, or--supposing $\br|_{\gal{F_v}}$ is ordinary, valued in a Borel $B$ with associated set of positive roots $\Phi^+$, with $\overline{\chi}_{T_G} \colon I_{F_v} \to T_G(k)$ as in \S \ref{ordsection}--for all $\alpha \in \Phi^+$, $\alpha \circ \overline{\chi}_{T_G}$ is a non-trivial power of $\bar{\kappa}$ (this condition is a minor strengthening of the condition (REG)). 
\end{rmk}
\begin{rmk}\label{lnotplocal}
 When $G= \mr{GL}_n$, $\mr{GSp}_{2n}$, or $\mr{GO}_n$, formally smooth local conditions $\mc{P}_v$ for $v \in S \setminus \{w \vert p\}$ with large enough tangent space are always known to exist (for $p \gg_G 0$), by \cite[\S 2.4.4]{clozel-harris-taylor} and \cite[\S 7]{booher:minimal} after possibly replacing $k$ with a finite extension. The global image hypothesis of our theorem does not allow us to make such a replacement, so to apply \cite[\S 2.4.4]{clozel-harris-taylor} or \cite[\S 7]{booher:minimal} one would have to check either that in the case of interest their arguments do not require extending scalars, or one would have to show that their results imply analogous results prior to extending scalars. We have not pursued this. Other examples of good local conditions for general groups are discussed in \cite[\S 4]{stp:exceptional} and \cite[\S 4]{stp:exceptional2}. We expect that it is always possible to find such $\mc{P}_v$, but this remains an interesting open problem. 
\end{rmk}
\begin{rmk}\label{effective}
We make some remarks on the effectivity of the bound $p\gg 0$ in the theorem. The possible need to increase $p$ arises at several points in the paper. In \S \ref{trivialsection}, the bound on $p$ is explicit.
In \S \ref{auxiliarysection} and \S \ref{klrsection} we have not computed an explicit bound on $p$, but we could easily derive one by following the arguments of those sections; the bounds coming from these sections essentially amount to the condition that certain $\Fp$-vector spaces not be covered by a finite (bounded absolutely in terms of $G$) number of hyperplanes. In deducing the image hypotheses of \S \ref{auxiliarysection} and \S \ref{klrsection} from the irreducibility hypothesis of Theorem \ref{mainthm}, there is an explicit bound ensuring the cohomology ($H^0$ and $H^1$) vanishing, and an explicit bound (see Remark \ref{explicit}) to ensure disjointness of $\br(\fgder)$ and $\br(\fgder)^*$. Finally, the dependence on $F$ only intervenes to assume $[F(\zeta_p):F]$ is sufficiently large, and when we need to invoke Lemma \ref{classgroup} to ensure that $\Hom(\mr{Cl}(\tF)/p\mr{Cl}(\tF), \br(\mf{z}_{\mu}))^{\Gal(\tF/F)}=0$ (eg, $\mr{Cl}(\tF)[p]=0$). In particular, if $G$ is connected or has zero-dimensional center, there is no dependence on $F$. In sum, the bound in Theorem \ref{mainthm} can be made effective.
\end{rmk}

\begin{proof}
Under our assumptions on $p$ and absolute irreducibility of $\br|_{\gal{\tF(\zeta_p)}}$, Lemmas \ref{irr} and \ref{invariants} apply, and so we assume those conclusions from now on. Moreover, Lemma \ref{cyclicq} implies that $\br(\fgder)$ and $\br(\fgder)^*$ have no common $\Fp[\gal{\tF}]$-subquotient. To see this, first note that the lemma implies that the fixed field $\tF(\fgder)$ cannot contain $\tF(\zeta_p)$ (for then the adjoint image of $\br(\gal{\tF})$ would have a large cyclic quotient, as we take $p \gg_F 0$). Letting $\{W_i\}_{i \in I}$ be the simple $\Fp[\gal{\tF}]$-module constituents of $\br(\fgder)$, if $\br(\fgder)$ and $\br(\fgder) \cong \br(\fgder)(1)$ had a common constituent, there would be an isomorphism $W_i \cong W_j(1)$ for some $i, j \in I$. We can choose $\sigma \in \gal{\tF}$ acting trivially on $W_i$ and $W_j$ but non-trivially on $\tF(\zeta_p)$, contradicting the equivalence $W_i \cong W_j(1)$. Thus, all the conditions of Assumption \ref{multfree} hold.

Fix a multiplier character $\nu \colon \gal{F, S} \to A(\mc{O})$ lifting $\mu \circ \br$, and if necessary (i.e., if $\mf{z}_{\mu} \neq 0$) satisfying $p \nmid \# \nu(I_{F_v})$ for all $v \nmid p$. In the discussion that follows, $\nu$ will be fixed, but we will omit it from the notation (for, e.g., local deformation conditions). We begin by invoking the results of \S \ref{klrsection}. 
As in the argument surrounding Lemma \ref{localmodp^2}, we choose a finite set of primes $T \supset S$ with $T \setminus S$ consisting of trivial primes so that $\Sha^1_T(\gal{F, T}, \br(\fg^{\mr{der}}))= \Sha^1_T(\gal{F, T}, \br(\fg^{\mr{der}})^*)=0$ and find a lift $\rho_2 \colon \gal{F, T} \to G(\mc{O}/p^2)$ of $\br$ (with multiplier $\nu$). We then fix the following local lifts of $\br|_{\gal{F_w}}$ for all $w \in T$:
\begin{itemize}
\item For each $w \in S$ not above $p$ we consider the local deformation condition $\mc{P}_w$ whose existence is assumed in Theorem \ref{mainthm} (invoking Lemma \ref{adlocalcondition} if necessary), and we fix a lift $\lambda_w \in \Lift^{\mc{P}_w}_{\br|_w}(\mc{O}/p^2)$. 
\item For $w \vert p$, we consider two different cases: if $\br|_{\gal{F_w}}$ is ordinary (valued in some Borel) and satisfies the conditions (REG) and (REG*) of \S \ref{ordsection}, then we fix a character $\chi_w \colon I_{F_w} \to T_G(\mc{O})$ such that for all simple roots $\alpha$, $\alpha \circ \chi_w$ is a positive integer power of $\kappa$. Then we let $\mc{P}_w$ be $\Lift^{\chi_w}_{\br}$ (as in Definition \ref{ordinary}), and let $\lambda_w$ be any element of $\Lift^{\chi_w}_{\br}(\mc{O}/p^2)$. If on the other hand $\br|_{\gal{F_w}}$ is trivial, then we choose $\chi_w$ as in Lemma \ref{ordextracocycles}, so that $\beta \circ \chi_w$ is non-trivial mod $p^2$ for all negative roots $\beta$, and we again further require that for all simple roots $\alpha$, $\alpha \circ \chi_w$ is a positive power of $\kappa$. Having chosen a maximal torus lifting $T_G$ as in Definition \ref{trivord}, we then let $\lambda_w$ be an element of $\Lift^{\chi_w}_{\br, 2}(\mc{O}/p^2)$.
\item For $w \in T\setminus S$, fix a lift $\lambda_w$ in the set $\Lift^{\alpha}_{\br|_w, 2}(\mc{O}/p^2)$ of Definition \ref{unrtrivlifts}, for some pair $(T, \alpha)$ of a maximal torus and root (these depend on $w$, but will be omitted from the notation). Moreover, we if necessary enlarge the set $T$ and choose the lifts $\lambda_w$ so that the elements $\lambda_w(\sigma_w)$ generate $\wh{G}^{\mr{der}}(\mc{O}/p^2)$.
We showed this is possible in Lemma \ref{bigimage}. 
\end{itemize} 
By Lemma \ref{adlocalcondition}, or by inspection as in Lemma \ref{classgroup}, for any $w \in T$ and any $c_w \in H^1_{\mr{unr}}(\gal{F_w}, \br(\mf{z}_\mu))$, the lift $(1+pc_w)\lambda_w$ lies in the same space of local lifts ($\Lift^{\mc{P}_w}_{\br|_w}(\mc{O}/p^2)$, $\Lift^{\chi_w}_{\br}(\mc{O}/p^2)$, etc.) from which we have just drawn $\lambda_w$.

As before, we write $z_T=(z_w)_{w \in T}$ for the element of $\oplus_{w \in T} H^1(\gal{F_w}, \br(\fgm))$ measuring the discrepancy between $\rho_2$ and the collection $(\lambda_w)_{w \in T}$, and we let $z_T^{\mr{der}}$ denote its component in $H^1(\gal{F, T}, \br(\fg^{\mr{der}}))$. Now one of two things can happen: 
\begin{itemize}
\item The local classes $z_T^{\mr{der}}$ may already lie in $\im(\Psi_T)$ (in the notation of \S \ref{klrsection}). In this case, we replace $\rho_2$ by some $(1+ph)\rho_2$, where $h \in H^1(\gal{F, T}, \br(\fgder))$ satisfies $h|_T=z_T^{\mr{der}}$.
\item If $z_T^{\mr{der}}$ is not in $\im(\Psi_T)$, then Proposition \ref{doublingprop} produces a finite set $Q$ of trivial primes and a class $h \in H^1(\gal{F, T \cup Q}, \br(\fgder))$ such that $h|_T= z_T^{\mr{der}}$ and $(1+ph)\rho_2|_w$ has, for all $w \in Q$, the form required in Definition \ref{ramtrivlifts}. We then replace $\rho_2$ by $(1+ph)\rho_2$.
\end{itemize}
Letting $Z$ be either $\emptyset$ or $Q$ in the two cases just described, and now writing $\rho_2$ for the replacement just described, we have now arranged that $\rho_2$ is unramified outside $T \cup Z$ and satisfies:
\begin{itemize}
\item For all $w \in T$, $\rho_2|_{\gal{F_w}}= (1+pc_w)\lambda_w$ for some $c_w \in H^1(\gal{F_w}, \br(\mf{z}_{\mu}))$.
\item For $w \in Z$, $\rho_2|_w$ belongs to a suitable set $\Lift^{\alpha}_{\br|_{\gal{F_w}}, 2, \mr{ram}}(\mc{O}/p^2)$ as in Definition \ref{ramtrivlifts}.
\end{itemize}

Now we explain how to correct the $\mf{z}_\mu$ component of our lift. Recall that the reduction map $G/G^{\mr{der}}(\mc{O}/p^2) \to G/G^{\mr{der}}(k)$ has a (group-theoretic) section $s$. There is then a class $a \in H^1(\gal{F, T}, \br(\mf{a}))$ such that $(1+pa)(s(\br \mod G^{\mr{der}}))$ has multiplier character $\nu$, and then another class $h'$ belonging to $H^1(\gal{F, T \cup Z}, \br(\fg_{\mu}/\fg^{\mr{der}}))= H^1(\gal{F, T \cup Z}, \br(\mf{z}_{\mu}))$ such that (in $G/G^{\mr{der}}(\mc{O}/p^2)$)
\[
(1+pa) s \left( \br \mod {G^{\mr{der}}}\right)= (1+ph')(\rho_2 \mod {G^{\mr{der}}}).
\]
Now replacing $\rho_2$ by $(1+ph')\rho_2$, we may assume that the class $c_w$ above (for all $w \in T \cup Z$) measures the discrepancy between $(\lambda_w \mod G^{\mr{der}})$ and $(1+pa)s(\rho|_w \mod G^{\mr{der}})$; it follows that $c_w$ belongs to $H^1_{\mr{unr}}(\gal{F_w}, \br(\mf{z}_\mu))$, since $(\lambda_w \mod G^{\mr{der}})$ itself differs from $s(\br \mod G^{\mr{der}})$ by a cocyle unramified in the $\mf{z}_{\mu}$ component (by the choice of local condition, as in Lemma \ref{adlocalcondition}). We can then simply modify our choice of $\lambda_w$ by this unramified $c_w$ to arrange $\rho_2|_w= \lambda_w$ for all $w \in T$; moreover, as noted above, the resulting new choice of $\lambda_w$ lies in the same desired space of local lifts we specified at the start of the proof.

In particular, since $\rho_2$ is locally unobstructed, there is a lift $\rho_3 \colon \gal{F, T \cup Z} \to G(\mc{O}/p^3)$ with multiplier $\nu$.

Next we define the following Selmer system $\mc{L}= \{L_w\}_{w \in T \cup Z}$:
\begin{itemize}
 \item If $w \in S \setminus \{v \vert p\}$, let $L_w= \Tan^{\mc{P}_w}_{\br|_{\gal{F_w}}}$.
 \item If $w \vert p$, then there are three cases. If $\br|_{\gal{F_w}}$ satisfies (REG) and (REG*), then let $L_w= \Tan^{\mc{P}_w}_{\br|_{\gal{F_w}}}$; if $\br|_{\gal{F_w}}$ is trivial, then let $L_w$ be the space $L^{\chi_w}_{\br|_{\gal{F_w}}}$ of Lemma \ref{ordextracocycles}; and if $G^0= \mr{GL}_N$, $\mr{GSp}_N$, or $\mr{GO}_N$ (or a product of such), and $\br|_{\gal{F_w}}$ is Fontaine-Laffaille as in the theorem statement, then let $L_w$ be the tangent space of the Fontaine-Laffaille deformation functor as in \cite[\S 2.4.1]{clozel-harris-taylor} or \cite[\S 5]{booher:JNT}.
 \item If $w \in T \setminus S$, then let $L_w$ be the appropriate space $L^{\alpha}_{\br|_{\gal{F_w}}}$ of Lemma \ref{extracocycles}.  
 \item If $w \in Z$, then let $L_w$ be the appropriate space $L^{\alpha}_{\br|_{\gal{F_w}}}$ of Lemma \ref{ramextracocycles}.
\end{itemize}
In what follows, we will be slightly abusive in writing $L_w$ for both this tangent space and its intersection with $H^1(\gal{F_w}, \br(\fgder))$. For $p \gg 0$, the class group assumption of Lemma \ref{classgroup} holds, so $H^1_{\mc{L}}(\gal{F, T \cup Z}, \br(\mf{z}_\mu))=0$ and $H^1_{\mc{L}^\perp}(\gal{F, T \cup Z}, \br(\mf{z}_{\mu})^*)=0$. The Greenberg-Wiles formula, our oddness and global image hypotheses, and the calculations of \S \ref{trivialsection} imply that we are in the ``balanced'' situation, i.e.
\[
 h^1_{\mc{L}}(\gal{F, T \cup Z}, \br(\fgder))= h^1_{\mc{L}^\perp}(\gal{F, T \cup Z}, \br(\fgder)^*),
\]
and consequently
\[
 h^1_{\mc{L}}(\gal{F, T \cup Z}, \br(\fgm))= h^1_{\mc{L}^\perp}(\gal{F, T \cup Z}, \br(\fgm)^*).
\]
In fact, we can say something more precise (this refinement is only relevant when $\pi_0(G) \neq 1$): if $\psi \in H^1_{\mc{L}^\perp}(\gal{F, T \cup Z}, \br(\fgder)^*) \setminus 0$ is supported on a single $G$-orbit of simple factors of $\fgder$, then there is also a non-zero $\phi \in H^1_{\mc{L}}(\gal{F, T \cup Z}, \br(\fgder))$ supported on that $G$-orbit. To see this, we note that the oddness condition restricts to an oddness condition on $G$-orbits, and that the equality
\[
H^1_{\mc{L}}(\gal{F, T \cup Z}, \br(\fgder))= \bigoplus_{l=1}^s \left(H^1_{\mc{L}}(\gal{F, T \cup Z}, \br(\fgder)) \cap H^1(\gal{F, T \cup Z}, \br(\mf{y}_l))\right)
\]
follows from our assumptions on $L_v$, $v \in S \setminus \{v \vert p\}$, and by inspection for the conditions $L_v$ that we have explicitly constructed.

We now use the results of \S \ref{auxiliarysection} to annihilate these Selmer groups by allowing additional ramification. Suppose that $H^1_{\mc{L}^\perp}(\gal{F, T \cup Z}, \br(\fgm)^*)=H^1_{\mc{L}^\perp}(\gal{F, T \cup Z}, \br(\fgder)^*)$ is non-zero (else, as we will see, we can lift as in the theorem's conclusion), and let $\psi$ be any non-zero element, which we may assume to be supported on a single $G$-orbit of simple factors of $\fgder$. From the balanced condition, we see that there is a non-zero $\phi \in H^1_{\mc{L}}(\gal{F, T \cup Z}, \br(\fgder))=H^1_{\mc{L}}(\gal{F, T \cup Z}, \br(\fgm))$, also supported on the same $G$-orbit. We can therefore invoke the results of \S \ref{auxiliarysection}: by Proposition \ref{splitcase}, there is a trivial prime $q \not \in T \cup Z$ and a root $\alpha$ (for some maximal torus of $G$) such that $\rho_2|_{\gal{F_q}} \in \Lift_{\br, 2}^{\alpha}(\mc{O}/p^2)$, $\phi|_q \not \in L_{\br|_{\gal{F_q}}}^{\alpha}$($=L^{\alpha}_q$ to simplify notation), and $\psi|_q \not \in (L_{\br|_{\gal{F_q}}}^{\alpha})^\perp$ ($=(L^{\alpha}_q)^\perp)$). The key point now is that if we let $L_q^{\mr{unr}}$ denote the unramified cohomology at $q$, then $L_q= L^{\mr{unr}}_q \cap L^\alpha_q$ is codimension one in $L_q^{\mr{unr}}$. The argument is now standard, but we recall it for convenience. A double application of the Greenberg-Wiles formula implies that
\[
 h^1_{\mc{L}^\perp \cup L_q^\perp}(\gal{F, T \cup Z \cup q}, \br(\fgder)^*)- h^1_{\mc{L}^\perp}(\gal{F, T \cup Z}, \br(\fgder)^*)= h^1_{\mc{L} \cup L_q}(\gal{F, T \cup Z \cup q}, \br(\fgder))-h^1_{\mc{L}}(\gal{F, T \cup Z}, \br(\fgder))+1,
\]
and the right-hand side of this equation is zero since $\phi|_q$ spans the one-dimensional space $L_q^{\mr{unr}}/L_q$. Thus the inclusion 
\[
H^1_{\mc{L}^\perp}(\gal{F, T \cup Z}, \br(\fgder)^*) \subset H^1_{\mc{L}^\perp \cup L_q^\perp}(\gal{F, T \cup Z \cup q}, \br(\fgder)^*) 
\]
is an equality, and since $\psi \not \in (L_q^{\alpha})^\perp$, the inclusion
\[
 H^1_{\mc{L}^\perp \cup (L_q^{\alpha})^\perp}(\gal{F, T \cup Z \cup q}, \br(\fgder)^*) \subset H^1_{\mc{L}^\perp \cup L_q^\perp}(\gal{F, T \cup Z \cup q}, \br(\fgder)^*)
\]
must be strict. The Selmer groups $H^1_{\mc{L} \cup L_q^\alpha}(\gal{F, T \cup Z \cup q}, \br(\fgder))$ and $H^1_{\mc{L}^\perp \cup (L_q^\alpha)^\perp}(\gal{F, T \cup Z \cup q}, \br(\fgder)^*)$ are still balanced and have dimension strictly smaller than $h^1_{\mc{L}}(\gal{F, T \cup Z}, \br(\fgder))$. Thus we can proceed inductively, finding a finite set $Q$ of trivial primes such that for all $q \in Q$, the restriction $\rho_2|_{\gal{F_q}}$ lies in a set $\Lift^\alpha_{\br, 2}(\mc{O}/p^2)$, and now $h^1_{\mc{L} \cup \{L_q^{\alpha}\}_{q \in Q}}(\gal{F, T \cup Z \cup Q}, \br(\fgder))= h^1_{\mc{L}^\perp \cup \{(L_q^{\alpha})^\perp\}_{q \in Q}}(\gal{F, T \cup Z \cup Q}, \br(\fgder)^*)=0$. Again by Lemma \ref{classgroup}, we likewise conclude that 
\[
h^1_{\mc{L} \cup \{L_q^{\alpha}\}_{q \in Q}}(\gal{F, T \cup Z \cup Q}, \br(\fgm))= h^1_{\mc{L}^\perp \cup \{(L_q^{\alpha})^\perp\}_{q \in Q}}(\gal{F, T \cup Z \cup Q}, \br(\fgm)^*)=0.
\]
Now the Selmer group version of the Poitou-Tate sequence implies that
\begin{equation}\label{selmeriso}
 H^1(\gal{F, T \cup Z \cup Q}, \br(\fgm)) \xrightarrow{\sim} \bigoplus_{w \in T \cup Z} H^1(\gal{F_w}, \br(\fgm))/L_w \oplus \bigoplus_{q \in Q} H^1(\gal{F_q}, \br(\fgm))/L^{\alpha}_q
\end{equation}
is an isomorphism, and
\begin{equation}\label{h2inj}
 H^2(\gal{F, T \cup Z \cup Q}, \br(\fgm)) \to \bigoplus_{w\in T \cup Z \cup Q} H^2(\gal{F_w}, \br(\fgm))
\end{equation}
is injective.

By Equation (\ref{selmeriso}), Lemma \ref{extracocycles}, Lemma \ref{ramextracocycles}, and Lemma \ref{ordextracocycles}, there is a cohomology class $X \in  H^1(\gal{F, T \cup Z \cup Q}, \br(\fgm))$ such that $(1+p^2X)\rho_3$ satisfies the following local conditions:
\begin{itemize}
 \item At primes $w \in S$ not above $p$, it belongs to $\Lift^{\mc{P}_w}_{\br|_{\gal{F_w}}}$.
 \item At primes $w \vert p$, it belongs either to $\Lift^{\chi_w}_{\br|_{\gal{F_w}}}(\mc{O}/p^3)$ (in the (REG) and (REG*) case), or to $\Lift^{\chi_w}_{\br|_{\gal{F_w}},2}(\mc{O}/p^3)$ (in the trivial case), or to the appropriate Fontaine-Laffaille deformation ring (using \cite[\S 2.4.1]{clozel-harris-taylor} and \cite[\S 5]{booher:JNT}) in the Fontaine-Laffaille cases.
 \item At primes $w \in Q \cup (T \setminus S)$, it belongs to $\Lift^{\alpha}_{\br|_{\gal{F_w}}, 2}(\mc{O}/p^3)$.
 \item At primes $w \in Z$, it belongs to $\Lift^{\alpha}_{\br|_{\gal{F_w}}, 2, \mr{ram}}(\mc{O}/p^3)$.
\end{itemize}
Now the procedure for inductively lifting is clear: $(1+p^2X)\rho_3$ is locally unobstructed, so by Equation (\ref{h2inj}) it can be lifted to $\rho_4 \colon \gal{F, T \cup Z \cup Q} \to G(\mc{O}/p^4)$; this lift can again be adjusted by a one-cocycle so that locally it satisfies the above four bulleted conditions, and so on. The result is a compatible system of lifts $(\rho_m)_{m \geq 1}$, each satisfying the four bulleted local conditions, and their $p$-adic limit $\rho= \varprojlim \rho_m \colon \gal{F, T \cup Z \cup Q} \to G(\mc{O})$ is the lift promised in the theorem statement. It is de Rham at $v \vert p$ by Lemma \ref{regord} or Lemma \ref{trivordlim}.

The final claim about the image can be checked by inductively showing that $\im(\rho_n)$ must contain $\widehat{G^{\mr{der}}}(\mc{O}/p^n)$ for all $n \geq 2$, with the base case $n=2$ coming from the construction of $\rho_2$. Suppose the claim is known for $n$. We will show that every element of the kernel of $\wh{G^{\mr{der}}}(\mc{O}/p^{n+1}) \to \wh{G^{\mr{der}}}(\mc{O}/p^n)$ is in any subgroup $H$ surjecting onto $\wh{G^{\mr{der}}}(\mc{O}/p^n)$. Embedding $G$ into some $\mr{GL}_N$, we will argue with matrices. Let $s= 1+p^nX$ be in the above kernel. By assumption there is some element $y \in H$ of the form $y= 1+p^{n-1}X+p^n Y$. Then $H$ also contains
\[
y^p= (1+p^{n-1}X+p^n Y)^p= \sum_{i=0}^p {p \choose i} (p^{n-1}X + p^n Y)^i = 1+p^nX, 
\]
since $n \geq 2$, and we are working modulo $p^{n+1}$.
\end{proof}
We also note that the method of proof allows us, without assuming oddness of $\br$, to construct possibly non-geometric $p$-adic deformations, since the arguments of Theorem \ref{mainthm} only require that whenever we have a non-trivial dual Selmer class, we can also find a non-trivial Selmer class,:
\begin{thm}\label{notodd}
Let $F$ be any number field, and let $\br \colon \gal{F, S} \to G(k)$ be a continuous representation unramified outside a finite set of places $S$ containing those above $p$. Let $\tF$ be as in Theorem \ref{mainthm}. Assume that $p \gg_{G, F} 0$, and that $\br$ satisfies the following:
\begin{itemize}
\item $\br|_{\gal{\tF(\zeta_p)}}$ is absolutely irreducible.
 \item For each $G$-orbit of simple factors $\br(\oplus_i \fg_i)$ of $\br(\fgder)$, each $\br(\fg_i)$ is multiplicity-free as $\Fp[\gal{\tF}]$-module. Moreover, each simple $\Fp[\gal{\tF}]$-constituent $W_i$ of $\br(\fgder)$ satisfies $\End_{\Fp[\gal{\tF}]}(W_i) \cong k$.
 \item For all $v \vert p$, $H^2(\gal{F_v}, \br(\fgm))=0$.
 \item The field $K= \tF(\br(\fgder), \mu_p)$ does not contain $\mu_{p^2}$ (again, see Remark \ref{mup2}).
 \item For all $v \in S$ not above $p$, there is a formally smooth local deformation condition $\mc{P}_v$ for $\br|_{\gal{F_v}}$ whose tangent space $\Tan^{\mc{P}_v}_{\br|_{\gal{F_v}}} \subset H^1(\gal{F_v}, \br(\mf{g}_\mu))$ has dimension $h^0(\gal{F_v}, \br(\fgm))$. If $G$ is not connected, \emph{and} the center of $G^0$ is positive-dimensional, we instead impose the hypothesis of Lemma \ref{adlocalcondition} above.
\end{itemize}
Then for some finite set of primes $T \supset S$, $\br$ admits a lift $\rho \colon \gal{F, T} \to G(\mc{O})$.
\end{thm}
\begin{proof}
The argument is the same as that of Theorem \ref{mainthm}, except at places $v \vert p$ we take the local deformation condition to be all lifts of $\br|_{\gal{F_v}}$. Our hypothesis ensures that this condition is formally smooth. The corresponding application of the Greenberg-Wiles formula (notation as in the proof of Theorem \ref{mainthm}) yields 
\[
h^1_{\mc{L}}(\gal{F, T \cup Z}, \br(\fgder))- h^1_{\mc{L}^\perp}(\gal{F, T \cup Z}, \br(\fgder)^*)= \sum_{v \vert p} [F_v:\Q_p]\dim_k(\fgder)- \sum_{v \vert \infty} h^0(\gal{F_v}, \br(\fgder)) \geq 0.
\]
(Equality holds when $F$ is totally real, and $\br(c_v)=1$ for all complex conjugations $c_v$). This inequality suffices to proceed as in the proof of Theorem \ref{mainthm}. 
\end{proof}
\section{Examples}\label{examples}
In this section we gather a few examples to which our method applies.
\subsection{The principal $\mr{SL}_2$}
In \cite{stp:exceptional} (and \cite{stp:exceptional2}) it was shown how the original lifting argument of \cite{ramakrishna02} and \cite{taylor:icos2} could be adapted to prove lifting results for $\br \colon \gal{F} \to G(k)$ whose image was (approximately) a principal $\mr{SL}_2$. In fact, the argument in that paper was carried out for the exceptional groups, at one point relying on a brute-force Magma computation (see \cite[Lemma 7.6]{stp:exceptional}); for the classical Dynkin types except for $D_{2n}$, case-by-case matrix calculations (not carried out in \cite{stp:exceptional}, but some of which appear in \cite{tang:thesis}) complete the argument. The arguments of the present paper apply to these examples without relying on case-by-case calculation; the multiplicity-free restriction in Theorem \ref{mainthm} still requires that we exclude type $D_{2n}$, however.

Let $G^0$ be a split connected reductive group over $\Z_p$. Recall that for $p \gg_{G^0} 0$, there is a unique conjugacy class of principal homomorphisms $\varphi \colon \mr{SL}_2 \to G^0$ defined over $\Z_p$ (see \cite{serre:principalsl2}). Assume that $G= {}^L H$, the L-group of a connected reductive group $H$ over $F$; that is, we choose over $\overline{F}$ a maximal torus and Borel subgroup $T_{\overline{F}} \subset B_{\overline{F}} \subset H_{\overline{F}}$ to obtain a based root datum, and then a choice of pinning allows us to define an $L$-group $G= {}^L H= H^\vee \rtimes \Gal(\tF/F)$ for some finite extension $\tF/F$. The principal $\mr{SL}_2$ extends to a homomorphism $\varphi \colon \mr{SL}_2 \times \gal{F} \to {}^L H$ (\cite[\S 2]{gross:principalsl2}), and we assume that $\varphi$ extends to a homomorphism $\mr{GL}_2 \times \gal{F} \to {}^L H$ (this is always the case if, eg, $H$ is simply-connected, and in general it can be arranged by enlarging the center of $H$). The following crucial assumption is needed to use the principal $\mr{SL}_2$ to produce \textit{odd} homomorphisms valued in ${}^L H$:
\begin{assumption}\label{oddgroup}
Assume that $\tF/F$ is contained in a quadratic totally imaginary extension of the totally real field $F$, and that the automorphism of $H^\vee$ given by projecting any complex conjugation $c \in \gal{F}$ to $\Gal(\tF/F)$ preserves each simple factor $\mf{h}^\vee_i$ of $\mf{h}^\vee= \Lie(H^\vee)$, and acts on $\mf{h}^\vee_i$ as the identity if $-1 \in W_{\mf{h}_i}$ and as the opposition involution if $-1 \not \in W_{\mf{h}_i}$. 
\end{assumption}
This assumption leads to the following archimedean calculation:
\begin{lemma}
Let $\theta_v \in {}^L H(k)$ be the element 
\[
\theta_v= \varphi \left( \begin{pmatrix}
-1 &0 \\ 0 & 1
\end{pmatrix} \times c_v \right).
\]
Then 
\[
\dim_k(\fgder)^{\Ad(\theta_v)=1}= \dim_k (\mf{n}),
\]
where $\mf{n}$ is the unipotent radical of a Borel subgroup of $G$ (or of $G^{\mr{der}}$).
\end{lemma}
\begin{proof}
Combining \cite[Lemma 4.19, 10.1]{stp:exceptional}, we find that $\dim_k (\fgder)^{\Ad(\theta_v)=1}= \dim_k (\mf{n})$. 
\end{proof}
\begin{thm}
Let $G= {}^L H$ be constructed as above, satisfying Assumption \ref{oddgroup}, let $p \gg_G 0$, let $S$ be a finite set of places of the totally real field $F$ containing all $v \vert p$, assume for simplicity that all places in $S$ are split in $\tF/F$, and let $\bar{r} \colon \gal{F, S} \to \mr{GL}_2(k)$ be a continuous representation satisfying the following properties:
\begin{enumerate}
\item The projective image of $\bar{r}$ contains $\mr{PSL}_2(k)$. 
\item $\det \bar{r}(c)=-1$ for all complex conjugations $c \in \gal{F}$.
\item For each $v \vert p$, $\bar{r}|_{\gal{F_v}}$ is
\begin{enumerate}
\item trivial;
\item ordinary, i.e. of the form
\[
\bar{r}|_{\gal{F_v}} \sim \begin{pmatrix} 
\chi_{1, v} & * \\ 0 & \chi_{2, v} 
\end{pmatrix}
\]
where $\frac{\chi_{1, v}}{\chi_{2, v}}|_{I_{F_v}}= \bar{\kappa}^{r_v}$ for some integer $r_v$, and $\bar{r}|_{\gal{F_v}}$ also has the property that $\varphi \circ \bar{r}|_{\gal{F_v}}$ satisfies both of the conditions (REG) and (REG*) in Lemma \ref{regord} (see \cite[Theorem 7.4]{stp:exceptional} and \cite[Proposition 4.2]{stp:exceptional2} for easily-checkable conditions under which these are both satisfied); 
\item or, in the case $G^0=\mr{GL}_N$, Fontaine-Laffaille with distinct Hodge-Tate weights in an interval of length less than $\frac{p-1}{N-1}$.
\end{enumerate}
\end{enumerate}
Then there exists a finite set of trivial primes $Q$ such that $\br= \varphi \circ \bar{r} \colon \gal{F, S \cup Q} \to G(k)$ has a geometric lift $\rho \colon \gal{F, S \cup Q} \to G(\mc{O})$, with the Zariski closure of $\im(\rho)$ containing $G^{\mr{der}}$. 
\end{thm}
\begin{proof}
It is straightforward to verify the hypotheses of Theorem \ref{mainthm} (compare \cite[Theorem 7.4, Theorem 10.4]{stp:exceptional}). To check that $\br(\fgder)$ is in fact $\Fp[\gal{\tF}]$-multiplicity free with simple factors having endomorphism algebra $k$, we just compare eigenvalues of elements $\varphi \begin{pmatrix} x & 0 \\ 0 & x^{-1} \end{pmatrix}$, for $x$ a generator of $k^\times$, on $\sigma(\Sym^{2m}(k^2))$ and $\Sym^{2n}(k^2)$, for $\sigma \in \Aut(k/\Fp)$. The claim follows from the fact that $p$ is sufficiently large compared to (the Coxeter number of) $G$. To satisfy the local hypotheses of Theorem \ref{mainthm}, first assume $v \in S \setminus \{v \vert p\}$. If $\br(I_{F_v})$ has order prime to $p$, we take minimal deformations in the sense of \cite[Lemma 4.17]{stp:exceptional}. In general, there are three cases to consider:
\begin{itemize}
\item If $\bar{r}|_{I_{F_v}}$ is decomposable, then $\br(I_{F_v})$ has order prime to $p$, so the above remark applies.
\item If $\bar{r}|_{I_{F_v}}$ is irreducible and $p \vert \#\bar{r}(I_{F_v})$, then Dickson's theorem forces $p \leq 5$, excluded by our assumption that $p \gg_G 0$.
\item If $\bar{r}|_{I_{F_v}}$ is indecomposable but reducible, then it has a unique $I_{F_v}$-stable line, which must therefore be $\gal{F_v}$-stable as well. In other words, $\bar{r}|_{\gal{F_v}}$ is reducible, of the form 
\[
\bar{r}|_{\gal{F_v}} \sim \begin{pmatrix}
\psi_1 & u \\
0 & \psi_2
\end{pmatrix}
\]
where $u$ defines an element of $H^1(\gal{F_v}, k(\psi_1/\psi_2))$ having non-trivial restriction to $I_{F_v}$. This restriction corresponds (since $v$ does not divide $p$) to an element of $\Hom_{\gal{F_v}}(\mu_p(\overline{F}_v), k(\psi_1/\psi_2))$, so in order to be non-zero we must have $\bar{\kappa}= \psi_1/\psi_2$, and $\br(I_{F_v}) \cong u(I_{F_v}) \cong \Z/p$. We then use the ``minimal" deformation condition for $\br|_{\gal{F_v}}$ described in \cite[Definition 4.4, Lemma 4.5]{stp:exceptional2}.
\end{itemize}
Finally, if $v \vert p$, we just need to say one thing about the Fontaine-Laffaille case. For $G^0=\mr{GL}_N$, the principal $\mr{SL}_2$ is equivalent to $\Sym^{N-1}(k^2)$, so we have to check that $\Sym^{N-1}(\bar{r}|_{\gal{F_v}}$ is also Fontaine-Laffaille. This follows from the fact that the Fontaine-Laffaille functor is compatible with tensor products in the range we have restricted to (see \cite[Fact 4.12]{booher:JNT}, which is proven in an appendix included in the arXiv version of \textit{ibid.}).
\end{proof}
\begin{rmk}
In particular, starting with $\bar{r} \colon \gal{\Q} \to \mr{GL}_2(k)$ coming either from classical modular forms or elliptic curves, we can construct geometric representations $\rho \colon \gal{\Q} \to G(\mc{O})$ whose image has Zariski closure containing $G^{\mr{der}}$. In fact, if $k= \Fp$, the fact that the mod $p^2$ lifts $\rho_2$ that we construct satisfy $\rho_2(\gal{\Q}) \supset \wh{G}^{\mr{der}}(\mc{O}/p^2)$ implies (for $p \gg_G 0$) that the image of $\rho$ even contains an open subgroup of $\wh{G}^{\mr{der}}(\Z_p)$.
\end{rmk}
\subsection{Normalizers of tori}
In this subsection we make no effort to be maximally general. For simplicity we assume that $G^0/Z_{G^0}$ is simple. Let $T$ be a (split) maximal torus of $G^0$. Residual representations valued in $N_G(T)(k)$ lift to $G(\mc{O})$, since (provided $p$ does not divide $\# W_{G^0}$) the image of $\br$ has order prime to $p$. Our main theorem shows that non-trivial lifts (with Zariski-dense image in $G$) also exist under suitable hypotheses on $\br$.

\begin{thm}
Let $p\gg_{G, F} 0$. Let $\br \colon \gal{F, S} \to N_G(T)(\Fp)$ satisfy the following:
\begin{itemize}
\item $\br(\gal{\tF})$ is equal to $N_G(T)(\Fp)$.
\item $\br$ is odd. For instance, we can make one of the following assumptions:
\begin{itemize}
\item If $-1 \in W_{G^0}$, then for all $v \vert \infty$, $\br(c_v)$ either projects to $-1 \in W_{G^0}$, or projects to $\rho^\vee(-1) \in G^{\mr{ad}}$ (where $\rho^\vee$ is the usual half-sum of the positive co-roots of $G$). 
\item If $-1 \not \in W_{G^0}$, then for all $v \vert \infty$, $\br(c_v)$ either equals $(w_0, \tau) \in G^0 \rtimes \pi_0(G)$, where $w_0$ lifts the longest element of $W_{G^0}$, and $\tau$ is a pinned outer automorphism of $G^0$ acting as the opposition involution on $T \cap G^{\mr{der}}$; or it projects to $(\rho^\vee(-1), \tau) \in G^{\mr{ad}} \rtimes \pi_0(G)$.
\end{itemize}
\item For all $v \vert p$, $\br|_{\gal{F_v}}$ satisfies one of the hypotheses (trivial, ordinary, Fontaine-Laffaille) in Theorem \ref{mainthm}.
\end{itemize}
Then for some finite set of places $T \supset S$, $\br$ admits a geometric lift $\rho \colon \gal{F, T} \to G(\mc{O})$ whose image has Zariski closure containing $G^{\mr{der}}$.
\end{thm}
\begin{proof}
First we check the local hypotheses of Theorem \ref{mainthm}. At primes in $S \setminus \{v \vert p\}$ we take a minimal deformation condition as in \cite[Lemma 4.4]{stp:exceptional}, since the order of $\im(\br)$ is prime to $p$. At $v \vert p$, we take the local condition as in Theorem \ref{mainthm}. That the examples given of possible $\br(c_v)$ are in fact involutions follows from \cite[Lemma 2.3]{yun:exceptional}, \cite[Lemma 10.1]{stp:exceptional}, and a similar check in the case $\br(c_v)= (w_0, \tau)$.

Now we proceed to the global hypotheses. The field $K= \tF(\br(\fgder), \mu_p)$ does not contain $\mu_{p^2}$, again since $\#\im(\br)$ is prime to $p$. Since $\br(\gal{\tF})$ equals $N_G(T)(\Fp)$, it is easy to check that $\br(\fgder)$ decomposes (for $p \gg_G 0$) as absolutely irreducible, non-isomorphic summands $\mf{t}^{\mr{der}} \oplus \mf{g}_{\mr{long}} \oplus \mf{g}_{\mr{short}}$ consisting of the (intersection with $G^{\mr{der}}$ of the) maximal torus and the subspaces of all long or short roots (if $\fg$ is simply-laced, then there are just two constituents in this decomposition). 
\end{proof}
This is not the most useful result, since it can be difficult to realize $N_G(T)(\Fp)$ as a Galois group over $\Q$ (the sequence $1 \to T \to N_G(T) \to W_G \to 1$ need not split). Theorem \ref{mainthm} is easily seen to apply when $\br(\gal{\tF})$ equals certain somewhat smaller subgroups of $N_G(T)(\Fp)$. For instance, in \cite{tang:thesis}, Tang classifies those connected reductive groups $G$ that arise as the Zariski closure of the image of a homomorphism $\gal{\Q} \to G(\overline{\Q}_p)$. The main theorem of \cite{tang:thesis} gives a complete answer to this question (for $p \gg 0$) modulo some elusive cases, consisting of certain simply-connected groups (e.g. $E_7^{\mr{sc}}$) for $p$ failing to satisfy some congruence condition (see \cite[Theorem 1.3]{tang:thesis}). As explained in \cite[Theorem 1.5, \S 3.4]{tang:thesis}, our main theorem allows Tang to treat these remaining cases. 
\subsection{An open case: deforming exotic finite subgroups}
We conclude by constructing some odd irreducible representations $\br \colon \gal{\Q} \to G(k)$ of a less Lie-theoretic flavor that neither our Theorem \ref{mainthm} nor potential automorphy theorems will succeed in lifting to Zariski-dense representations $\br \colon \gal{\Q} \to G(\mc{O})$. Our results do not apply because in these examples $\#\br(\gal{\Q})$ is coprime to $p$ (so all of the potentially problematic group extensions as in Assumption \ref{big} split), and $\gal{\Q}$ acts on $\br(\fg)$ with multiplicities greater than 1. Recall that over $\CC$ we have an embedding $F_4(\CC) \into E^{\mr{sc}}_6(\CC)$ given by identifying $F_4$ to the stabilizer of a vector in one of the 27-dimensional minuscule representations $V_{\mr{min}}$ of $E_6^{\mr{sc}}$. Letting $H$ and $G$ be the split groups (over $\Z$) of type $F_4$ and $E_6^{\mr{sc}}$, we can realize this embedding $H \into G$ over $R= \mc{O}_E[\frac{1}{N}]$ for some number field $E$ and integer $N$ (a quantitative refinement of this soft ``spreading-out" assertion is of course possible). By \cite[1.1 Main Theorem]{cohen-wales}, the finite groups $A_6$ and $\mr{PSL}_2(\mathbb{F}_{13})$ embed into $F_4(\CC)$, and, perhaps after replacing $E$ (and hence $R$) by a finite extension, we may assume these groups are embedded into $H(R)$. This theorem also tells us the characters of $A_6$ and $\mr{PSL}_2(\mathbb{F}_{13})$ in $V_{\mr{min}}$ and the adjoint representation of $E_6$. Recalling the decompositions as $F_4$-representations
\begin{align*}
\Lie(E_6)&= \Lie(F_4) \oplus U, \\
V_{\mr{min}}&= \mathbbm{1} \oplus U,
\end{align*}
where $U$ is the irreducible 26-dimensional representation of $F_4$, we compute the following decompositions of $\Lie(F_4)$ as $A_6$ and $\mr{PSL}_2(\mathbb{F}_{13})$-representations:
\[
\Lie(F_4) \cong \begin{cases}
\chi_4 \oplus 3 \cdot \chi_5 \oplus 2\cdot \chi_7 \quad \text{(case $A_6$)}; \\
\chi_4 \oplus \{\text{$\chi_5$ or $\chi_6$}\} \oplus 2 \cdot \chi_9 \quad \text{(case $\mr{PSL}_2(\mathbb{F}_{13})$)},
\end{cases}
\]
where we use the ATLAS notation (\cite{ATLAS}) for characters. It turns out that for our purposes knowing whether $\chi_5$ or $\chi_6$ appears in the decomposition in the $\mr{PSL}_2(\mathbb{F}_{13})$ case is irrelevant. In particular, letting $c$ denote the unique conjugacy class of order 2 in either case, the ATLAS character tables tell us that the trace of $c$ acting on $\Lie(F_4)$ is $-4= -\rk(F_4)$, and so
\[
\dim \Lie(F_4)^{\Ad(c)=1}= \frac{\dim(F_4)- \rk(F_4)}{2}= \dim \mr{Flag}_{F_4},
\]
i.e. $\Ad(c)$ is an odd involution of $\Lie(F_4)$.
\begin{eg}\label{f4eg}
For a positive density set of primes $p$, there are representations $\br_1 \colon \gal{\Q} \to F_4(\overline{\mathbb{F}}_p)$ and $\br_2 \colon \gal{\Q} \to F_4(\overline{\mathbb{F}}_p)$ that have images $\im(\br_1) \cong A_6$, $\im(\br_2) \cong \mr{PSL}_2(\mathbb{F}_{13})$, and that satisfy all the hypotheses of Theorem \ref{mainthm} except the multiplicity-free condition on $\br(\fg)$. Moreover, composing either $\br_i$ with any faithful representation of $\mr{F}_4$ yields a residual representation that does not satisfy the hypotheses of the potential automorphy theorems of \cite{blggt:potaut}.
\end{eg}
\begin{proof}
There are Galois extensions $L_1/\Q$ and $L_2/\Q$ satisfying $\Gal(L_1/\Q) \cong A_6$, $\Gal(L_2/\Q) \cong \mr{PSL}_2(\mathbb{F}_{13})$, and complex conjugation $c$ is non-trivial in each $\Gal(L_i/\Q)$: the constructions of $L_1$ and $L_2$ are due to Hilbert and Shih, respectively, and both are explained in \cite[\S 4.5, Theorem 5.1.1]{serre:topics}. It is easy to see that we can take $c$ to be non-trivial, and note that to apply Shih's theorem we use that $\left(\frac{2}{13}\right)=-1$. Let $p$ be any sufficiently large (in the sense of Theorem \ref{mainthm} for $F_4$, not dividing $N$, and not dividing $\# \im(\br_i)$) prime that is split in $L_i/\Q$. Reducing the inclusions $\Gal(L_i/\Q) \into H(R)$ modulo a prime of $R$ above $p$, we obtain residual representations $\br_i \colon \gal{\Q} \to H(\overline{\Fp})$ satisfying the conditions of Assumption \ref{multfree} (note that $L_i/\Q$ is disjoint from $\Q(\mu_p)/\Q$, and that none of the characters in the above decompositions of $\Lie(F_4)$ are trivial) \textit{except for} the multiplicity-free requirement. Moreover, by the character calculation preceding Proposition \ref{f4eg}, both $\br_i$ are odd. At the prime $p$, $\br_i|_{\gal{\Q_p}}$ is trivial, and at primes $\ell$ of ramification of $L_i/\Q$, $\br_i|_{I_{\Ql}}$ has order prime to $p$, so to satisfy the local hypotheses of Theorem \ref{mainthm} we can take minimal deformations at $\ell$ as in \cite[Lemma 4.17]{stp:exceptional}. 

Finally, we note that every non-trivial irreducible representation of $F_4$ has some multiplicity greater than 1 in its formal character, so $\br_i$ with such a representation cannot have any geometric lift whose composite with such a representation is Hodge-Tate regular. Thus we cannot apply the potential automorphy theorems of \cite{blggt:potaut} (as in \cite{bcelmpp}) to lift our $\br_i$. (Even worse, the actions of the subgroups $A_6$ and $\mr{PSL}_2(\mathbb{F}_{13})$ on $U$ (the irreducible 26-dimensional representation of $F_4$) are reducible.)
\end{proof}
We conclude by noting that, of course, the representations $\br_i$ just constructed do satisfy the hypotheses of Corollary \ref{infiniteram}, so they do admit (non-geometric) lifts to $G(\mc{O})$ with image as large as possible.
\bibliographystyle{amsalpha}
\bibliography{biblio.bib}
\end{document}